%% file: 2014_08_26_FPMST_arxiv_v3.tex
\documentclass[leqno,oneside]{amsart}
\usepackage{latexsym,amssymb,amsmath,graphics,enumerate,amsfonts,amsthm,natbib}

\usepackage{graphicx}
\usepackage[english]{babel}
\usepackage[T1]{fontenc}

 \usepackage{tgtermes} 

\newtheorem{thm}{Theorem}
\newtheorem*{thm*}{Theorem}
\newtheorem{theorem}[thm]{Theorem}
\newtheorem{cor}[thm]{Corollary}

\newtheorem{lem}[thm]{Lemma}
\newtheorem*{lem*}{Lemma}
\newtheorem{lemma}[thm]{Lemma}

\newtheorem{prop}[thm]{Proposition}
\newtheorem{proposition}[thm]{Proposition}
\newtheorem{defnlem}[thm]{Definition and Lemma}
\theoremstyle{definition}
\newtheorem{defn}[thm]{Definition}
\newtheorem{rem}[thm]{Remark}

\numberwithin{equation}{section}
\numberwithin{thm}{section}

\usepackage{color}
\newcommand{\red}{\color{red}}

\usepackage[nomain,toc]{glossaries} %
\newglossary[slg]{symbols}{sym}{sbl}{List of Symbols}
\makeglossaries

\input{Makros_FPMST.tex}

 \renewcommand{\est}[1][\alpha]{H^{#1}}
 \renewcommand{\fprv}{X}
 
 \usepackage{dsfont}
 \renewcommand{\1}[1][]{\mathds{1}_{#1}}
  \renewcommand{\deins}{\mathbf{1}}
    \renewcommand{\eins}{\tilde{\mathbf{1}}}
    \renewcommand{\tree}{\mathds{V}}
    \renewcommand{\interior}[1]{\mathrm{int}\left(#1\right)}
 
\allowdisplaybreaks

\begin{document}

\title{The Fixed Points of the Multivariate Smoothing Transform }
\author{Sebastian Mentemeier}
\thanks{Sebastian Mentemeier. Instytut Matematyczny, Uniwersytet Wroc\l awski,
pl. Grunwaldzki 2/4, 50-384 Wroc\l aw, Poland. e-mail:
mente@math.uni.wroc.pl . Research supported by Deutsche Forschungsgemeinschaft
(SFB 878). }
\maketitle

\begin{abstract}
Let be given a sequence $(\mT_1, \mT_2, \dots)$ of random $d \times d$ matrices with nonnegative entries and a random vector $Q$ with nonnegative entries. 
Consider random vectors $\fprv$ with nonnegative entries, satisfying
\begin{equation} \label{SFPE_abstract}\tag{*} \fprv ~\eqdist~ \sum_{i \ge 1} \mT_i \fprv_i + Q, \end{equation}
where $\eqdist$ denotes equality of the corresponding laws, $(\fprv_i)_{i \ge 1}$ are i.i.d.~copies of $\fprv$ and independent of $(Q, \mT_1, \mT_2, \dots)$.

For $d=1$, this equation, known as fixed point equation of the smoothing transform, has been intensively studied. 
Under assumptions similar to the one-dimensional case, we obtain a complete characterization of all solutions $\fprv$ to \eqref{SFPE_abstract} in the non-critical case, and existence results in the critical case.

%
%


\end{abstract}

\section{Introduction}

Let $d >1$. Let $(\mT_i)_{i \ge 1}$ be a sequence of random $d\times d$-matrices  and $Q$ a random vector in $\Rd$.  Assume that $$N := \# \{ i \, : \, \mT_i \neq 0 \}$$ is finite a.s. We will presuppose throughout that the $(\mT_i)_{i \ge 1}$ are ordered in such a way that $\mT_i \neq 0$ if and only if $i \le N$. For any random variable $X \in \Rd$, let $(X_i)_{i \ge 1}$ be a sequence of i.i.d.~copies of $X$, and independent of $(Q, (\mT_i)_{i \ge 1})$. Then the random variable
$ \sum_{i =1}^N \mT_i X_i + Q$
is well defined, and one can ask the question, whether it holds that
\begin{equation}\label{eq:SFPE} X \eqdist \sum_{i=1}^N \mT_i X_i + Q, \end{equation}
where $\eqdist$ denotes equality in law. If $X$ is a random variable such that \eqref{eq:SFPE} holds, then we call its law $\law{X}$ a solution of Eq. \eqref{eq:SFPE}, and ask for a characterization of all solutions of Eq. \eqref{eq:SFPE}.
In other words, introducing the mapping $\STi$ on the set $\Pset(\Rd)$ of probability measures on $\Rd$, defined by
\begin{equation}
\label{def:STi} \STi \eta ~:=~ \law{\sum_{i=1}^N \mT_i X_i + Q},
\end{equation}
where $(X_i)_{i \ge 1}$ are i.i.d. random variables with law $\eta$, and independent of $(Q, (\mT_i)_{i \ge 1})$, we want to find all $\eta \in \Pset{\Rdnn}$, such that $\STi \eta = \eta$, i.e. all fixed points of $\STi$. Following \cite{DL1983}, we call $\STi$ the multivariate smoothing transform, and 
\begin{equation}
\label{def:STi} \STh \eta ~:=~ \law{\sum_{i=1}^N \mT_i X_i},
\end{equation}
the homogeneous multivariate smoothing transform. By an slight abuse of notation, we will also call a random variable $X$ a fixed point (FP) or solution, if its $\law{X}$ is a fixed point of the multivariate smoothing transform.

Note that $X \equiv 0$ is always a solution of the  homogeneous equation with $Q \equiv 0$, to which we refer as the trivial one. If there is $X \neq 0$ which is a fixed point of the homogeneous equation, then $cX$ will be a fixed point of the homogeneous equation as well, for any $c \in \R$. 
A more formal way to define the fixed point property is to consider the mapping 

Multivariate stochastic fixed point equations of the form \eqref{eq:SFPE} arise naturally in the study of branching processes, divide-and-conquer algorithms or generalized Polya urn models, some instances appear in \cite[Theorem 4.1 and Theorem 4.3]{NR2006}, \cite[Theorem 1]{Blum2006} or in \cite[Eq. (3.5)]{Janson2004}, which can be written in matricial form, too. 
Recently, \cite{Bassetti2014} used Eq. \ref{eq:SFPE} with $Q \equiv 0$ to study generalized kinetic models, aiming at the description of particle velocities in a Maxwell gas.

Understanding Eq. \eqref{eq:SFPE} (with $Q \equiv 0$) is an important step to study the so-called {\em multiplicative chaos equation} for matrices, i.e. to find random matrices $\matrix{X}$, satisfying
\begin{equation}\label{eq:SFPEm}  \matrix{X} ~\eqdist~ \sum_{i=1}^N \mT_i \matrix{X}_i,\end{equation}
where $(\matrix{X}_i)_{i \ge 1}$ is a sequence of i.i.d.~copies of $\matrix{X}$ and independent of $(\mT_i)_{i \ge 1}$. 
To wit, if $X, X_1, X_2, \dots$ satisfy Eq. \eqref{eq:SFPE} (with $Q \equiv 0$), then for every $a=(a_1, \cdots, a_d)^\top \in \Rd$ it holds that
$$ \left(\begin{array}{cccc} a_1 X & a_2 X & \dots & a_d X \end{array} \right) ~\eqdist~ \sum_{i=1}^N \mT_i \, \left(\begin{array}{cccc} a_1 X_i & a_2 X_i & \dots & a_d X_i \end{array} \right),$$
i.e. the rank-one matrix $\matrix{X}=Xa^\top$ is a fixed point of Eq. \eqref{eq:SFPEm}. There is substantial interest in understanding \eqref{eq:SFPEm}, for it may serve as a discrete approximation to Gaussian matrix-multiplicative chaos, a field of study which has just started, see \cite{Chevillard2013,Rhodes2013}.

\medskip

The long-term goal is to characterize all solutions of the multivariate stochastic fixed point equation \eqref{eq:SFPE} for general matrices. As the historic survey for the one-dimensional case will show, techniques (for $d=1$) were first developed to understand nonnegative fixed points of the {\em homogeneous} equation (i.e. $Q=0$) with nonnegative {\em weights} $(T_i)_{i \ge 1}$. This equation then served as a r\^ole model for the study of the more complicated case of real-valued fixed points, or even real-valued weights; a problem which was solved very recently in \cite{Iksanov2014}. Here, we set out our investigations by studying the multidimensional analogue of the nonnegative case, i.e. we assume that 
\begin{equation}\tag{A0}\label{A0} \text{$(Q,(\mT_i)_{i \ge 1})$ is a random variable in $\Rdnn \times M(d\times d, \Rnn)^\N$}\end{equation}
and, as mentioned before, that
\begin{equation}\tag{A1}\label{A1} \text{the r.v. $N := \# \{ i \, : \, \mT_i \neq 0 \}$ equals $\sup\{ i \, : \, \mT_i \neq 0 \} $ and is finite a.s.} \end{equation}
Here, $\Rnn :=[0,\infty)$.

In this setting, \cite{Menshikov2005} asked (Remark after Proposition 2.4) for sufficient conditions for the existence of a fixed point of the matrix-multiplicative chaos equation. As indicated above, by characterizing the fixed points of Eq. \eqref{eq:SFPE}, we solve this open problem. Further applications 
will be given below.

\medskip

The basic example the reader should keep in mind are multivariate stable laws: Let $N>1$ be fixed and the $\mT_i$, $1 \le i \le N$ be deterministic diagonal matrices with entries $N^{-1/\alpha}$ on the diagonal, $\alpha \in (0,1]$.  The equation
\begin{equation} \label{eq:stablelaws}
X \eqdist \sum_{i=1}^N N^{-1/\alpha} X_i
\end{equation} for $X \Rdnn$ is satisfied by the multivariate strictly $\alpha$-stable laws.
See \cite{ST1994} for further information on multivariate stable laws.

Equation \eqref{eq:SFPE} can be understood as a generalized equation of stability, with the scalar weights $N^{-1/\alpha}$ replaced by nonnegative matrices. It will turn out that the fixed points are in fact mixtures of multivariate $\alpha$-stable laws, with $\alpha$ being defined in such a way that it coincides with the index of stability in the particular case of Eq. \eqref{eq:stablelaws}. Moreover, we will introduce conditions, under which Eq. \eqref{eq:SFPE} has exactly a one-parameter class of solutions, in contrast to Eq. \eqref{eq:stablelaws}, which is solved by all multivariate strictly $\alpha$-stable laws on $\Rdnn$, which are parametrized by the set of finite measures on $\Sp$.

In order to explain what assumptions and results can be expected, we will now review the history of the problem in dimension $d=1$, . Then we state all our assumptions and the main result. 


\subsection{History of the Problem: Results for $d=1$}
For $d=1$, $(Q, T_1, T_2, \dots)$ is a sequence of nonnegative random variables, and one can define, for $s \ge 0$,  the quantity
$$ \hat{m}(s) := \E \sum_{i=1}^N T_i^s \in [0, \infty].$$ Then it is immediate, upon taking expectations, that $\hat{m}(1)=1$ ($\hat{m}(1)<1$) is a necessary condition for the homogeneous (inhomogeneous with $\E Q >0$) equation to have a nontrivial, i.e.~non-zero, fixed point with finite expectation. But the r\^ole of the function $s \mapsto m(s)$ is much more fundamental: It is log-convex on its domain of definition, and subject to the assumption $\E N>1$, which is imposed in all the studies below, there are at most two values $0 < \alpha < \beta$ with $m(\alpha)=m(\beta)=1$. If derivatives exists, then $m'(\alpha^-) \le 0$, $m'(\beta) \ge 0$, with strict inequality holding unless $\beta=\alpha$, which then is called the {\em critical case}. The set of fixed points is structured by the value of $\alpha$, which generalizes the index of stability. Roughly speaking, there are two classes of fixed points: mixtures of $\alpha$-stable laws, which have an infinite moment of order $\alpha$, and fixed points with a finite moment of order $\alpha$, which then might have heavy tails with tail index $\beta$, if $\beta$ exists. The second class of fixed points attracts point masses (and can be found by using Banach's fixed point theorem), while the first class attracts $\alpha$-stable laws and needs  more involved techniques to be characterized.

It was shown in \cite{Doney1972,Biggins1977,KP1976}, that $m(1)=1$ together with $m'(1^-)<0$ is a sufficient condition for the existence of fixed points of $\STh$ with finite expectation. The first two papers  are motivated by questions about  convergence of martingales in the branching random walk, while the third paper studies so-called Mandelbrot cascades, a model introduced in \cite{Mandelbrot1974}.
Motivated by questions from interacting particle systems, \cite{DL1983} considered the homogeneous equation for general $\alpha \in (0,1]$ under the assumption that $N$ is bounded. Directly related to Eq. \eqref{eq:SFPE} is the functional equation $$f(t)=\E \prod_{i=1}^N f(T_i t),$$ where $f(t)=\E \exp(-tX)$ is the Laplace transform (LT) of a nonnegative random variable. \cite{DL1983} proved that there is a function $L$, slowly varying at 0, such that the Laplace transform of each nontrivial fixed point satisfies (assuming non-arithmeticity)
\begin{equation}\label{eq:alphaelementaryd1} \lim_{r \to 0} \frac{1 - f(t)}{L(t)t^\alpha} = K\end{equation}
for some $K >0$. The function $L$ is constant in the non-critical case, and equals $L(t)=\abs{\log t}$ in the critical case $\alpha=1$. Using the Tauberian theorem for LTs, \cite[XIII.(5.22)]{Feller1971}, this relation also proves that fixed points have tails of order $\alpha$ for $\alpha <1$, and finite expectation if $\alpha=1$ and $m'(\alpha^-)<0.$
  This result was extended to the case of more general $N$, namely random $N$, in \cite{Liu1998}, which also contains many references and additional results. 
The homogeneous equation in the {\em critical case}, i.e. $m'(\alpha^-)=0$ was studied in \cite{Biggins1997,Kyprianou1998,Biggins2005b,Aidekon2014}. 

\cite{Iksanov2004} studied  the functional equation for functions $f$  which are not necessary Laplace transforms, but satisfying (a generalization of) \eqref{eq:alphaelementaryd1}.  It was proved only recently in \cite{ABM2012}, under the weakest assumptions known, that any monotone function with values in $[0,1]$, satisfying the functional equation, is of the form \begin{equation}\label{eq:ft}f(t)=\E \exp(-K t^\alpha W),\end{equation} where $W$ is the (up to scaling) unique fixed point of $W \eqdist \sum_{i\ge1} T_i^\alpha W_i$, where $W, W_1, \cdots$ are i.i.d.~and independent of $(T_i)_{i \ge 1}$. Here $\alpha \in (0, \infty)$ and $\hat{m}'(\alpha) \le 0$ is considered. Such functions are completely monotone (hence Laplace transforms of a random variable) if and only if $\alpha \in (0,1]$. In that case, $f(t)$ describes a mixture of $\alpha$-stable laws, the LTs of which are given by $f_{\alpha,c}(t):=\exp(-c t^\alpha)$, $c>0$.  

The set of fixed points for the inhomogeneous case was described in \cite{AM2010a}: There is a particular fixed point $W^*$ (called minimal solution there), which has a finite moment of order $\alpha$, and in general, there is a one-parameter family of fixed points, the Laplace transform of which satisfies
$$ \E \exp(-tX) ~=~ \exp(- K t^\alpha W - tW^*)$$ for some $K \ge 0$, where $W$ is as above. If $K> 0$, these fixed points have an infinite moment of order $\alpha$. The inhomogeneous equation in the critical case was studied in \cite{Buraczewski2014}. Real-valued fixed points (still for $T_i$ being nonnegative) were considered in \cite{Caliebe2003,CR2003,AM2010b}, and the most general question about fixed points for real-valued { weights} $(Q,T_1, T_2, \dots)$ was solved very recently by \cite{Iksanov2014}.

All the papers cited above studied as their main objective the structure of the set of fixed points of $\STh$ or $\STi$ in the one-dimensional case, the aim we want to pursue in the multivariate case.
There is a lot of related work concerned with properties of particular fixed points from the second class, namely those of finite expectation (i.e. $\alpha=1$) for the homogeneous equation, or properties of $W^*$ for the inhomogeneous equation. The tail behavior of $L^1$-fixed points of $\STh$ was studied in \cite{Gui1990} and subsequently in \cite{Liu2000}. \cite{BDGM2014} is the generalization of the first paper to the multivariate situation. It turns out that these fixed points have Pareto-like tails with tail index $\beta$, where $\hat{m}(\beta)=1$, $\hat{m}(\beta)>0$. \cite{JO2010a} proved that the particular solution $W^*$ of the inhomogeneous equation satisfies
$$ \lim_{t \to \infty} t^\beta \P{W^* > t}= C \ge 0,$$ where positivity of $C$ was proved in some cases, and under general assumptions later in \cite{Alsmeyer2013a,BDZ2014+}. 
The results of \cite{Mirek2013} constitute  the multidimensional analogue of \cite{JO2010a}. Tail behavior of fixed points in the critical case was studied in \cite{Buraczewski2009} for the homogeneous, and in  \cite{Buraczewski2014} for the inhomogeneous equation. The case of real-valued weights was considered in \cite{JO2010b}. The r.v. $W$ in \eqref{eq:ft} is the limit of a martingale $W_n$ (more details given below), tails of $\sup_n W_n$ were studied in \cite{Iksanov2006}. Continuity properties of the density of fixed points were considered in \cite{Liu2001}.

\subsection{Statement of the Main Result}
\subsubsection*{The Weighted Branching Model for Matrices}
Let $\tree = \bigcup_{n=0}^\infty \N^n $ be the infinite
Ulam-Harris tree with root $\emptyset$. For a node $v=v_1 \dots v_n$ write
$\abs{v}=n$ for its generation, $v|k=v_1 \dots v_k$ for its {\em ancestor} in the $k$-th generation,
$k \le n$, and $vi$ for its $i$-th child. Attaching to each node an independent
copy $\T(v)=(Q(v),\mT_1(v), \mT_2(v), \dots)$ of $\T=(Q,\mT_1, \mT_2, \dots)$, $\mT_i(v)$
can be interpreted as a weight along the path from $v$ to $vi$. The product of the weights along the
unique shortest path from $\emptyset$ to $v$ is defined recursively by
$\mL(\emptyset)= \Id$, the identity matrix, and
\begin{equation}\label{def:Lv}
\mL(vi) = \mL(v)\mT_i(v) . 
\end{equation}
Due to the assumption $N < \infty$ a.s., the number of nonzero weights $\mL(v)$ is a.s. finite in every generation, in fact, it follows a Galton-Watson process with offspring law $N$. Define a filtration by
$$ \B_n := \sigma \left( \T(v), \abs{v} < n\right).$$ Observe that $\mL(v)$ is $\B_{\abs{v}}$-measurable, while $Q(v)$ is independent of $\B_{\abs{v}}$.

\subsubsection*{A Particular Fixed Point}
With these definitions,  a natural candidate for a fixed point of \eqref{eq:SFPE} is given by the law of the random variable
\begin{equation} \label{eq:defWstar}W^*:= \sum_{n=0}^\infty \, \sum_{\abs{v}=n} \mL(v) Q(v), \end{equation}
as soon as this sum converges.
A sufficient condition therefore is that $\E \abs{W^*}^s < \infty$ for some $s>0$. Considering e.g. $s \le 1$, 
one obtains
\begin{equation} \label{eq:Wstar} \E \abs{W^*}^s < \E \abs{Q}^s  \sum_{n=0}^\infty \, \E \sum_{\abs{v}=n}  \norm{\mL(v)}^s,  \end{equation}
and the right hand side is finite if and  only if $\E \abs{Q}^s < \infty$ and the quantity
$$m(s) := \limsup_{n \to \infty} \left(\E \sum_{\abs{v}=n}  \norm{\mL(v)}^s\right)^{1/n} $$
is smaller than 1. The function $s \mapsto m(s)$ will play a fundamental role in the characterization of the set of fixed points below, it is the multivariate analogue of the function $\hat{m}(s)$ that appeared in the one-dimensional case. 

\subsubsection*{Assumptions on N} 
Besides \eqref{A1} which was already introduced, assume that
\begin{equation}
\tag{A2}\label{A2} N \ge 1 \text{ a.s. and } 1 < \E N < \infty.
\end{equation}

The finiteness of $\E N$ allows to introduce a probability measure $\mu$ on nonnegative matrices, defined by
\begin{equation}\label{eq:defmu} \int \, f(\ma) \, \mu(d\ma) ~:=~ \frac{1}{\E N} \, \E \left( \sum_{i=1}^N f(\mT_i)\right). \end{equation}

If now $(\mM_n)_{n \ge 1}$ is a sequence of i.i.d.~random matrices with law $\mu$,  then
\begin{equation}
m(s) ~=~ \E N \, \left(\lim_{n \to \infty} \left(\E   \norm{\mM_n \cdots \mM_1}^s\right)^{1/n} \right).
\end{equation}
We write 
$$ I_\mu := \{ s \ge 0 \, : \, m(s) < \infty\}.$$ On $I_\mu$, the function $m(s)$ is log-convex, thus continuous and differentiable on the interior $\interior{I_\mu}$ of $I_\mu$. 

The assumption $N \ge 1$ a.s. is for convenience, it assures that the underlying branching process (the subtree pertaining to nonzero weights $\mL(v)$) survives with probability 1 and is supercritical due to the assumption $\E N>1$. 

\subsubsection*{Geometrical Assumptions}
A {\em nonnegative} matrix $\ma \in \Mset:=M(d\times d, \Rnn)$ is called {\em allowable}, if it has no zero row nor column. We say that $\ma$ is {\em positive}, if all its entries are positive ($>0$), i.e.  $\ma \in \interior{\Mset}$.
Let $\Gamma$ be a semigroup of nonnegative matrices. 
\begin{defn}\label{defn:condc}
We say that $\Gamma$ satisfies condition $\condC$, if 
\begin{enumerate}
\item every $\ma$ in $\Gamma$ is allowable, and
\item $\Gamma$ contains a positive matrix.
\end{enumerate} 
\end{defn}
Then we impose the following assumption:
\begin{equation}
\tag{A3} \label{A3} \text{ The subsemigroup $[\supp \mu]$ generated by $\supp \mu$ satisfies $\condC$.}
\end{equation}

We are also going to impose a non-arithmeticity assumption. Recall that a positive matrix $\ma$ has, due to the Perron-Frobenius theorem, a unique dominant eigenvalue $\lambda_\ma$, exceeding all other eigenvalues in modulus, which is furthermore positive and algebraically simple, with corresponding eigenvector $u_\ma \in \Sp := \Sd \cap \Rdnn$.
\begin{equation}
\tag{A4} \label{A4} \text{The additive group generated by $\{ \log \lambda_{\ma} \, : \, \ma \in [\supp \, \mu] $ is positive$\}$ is dense in $\R$}
\end{equation} 

\subsubsection*{Moment Assumptions}

The results from the one-dimensional case indicate that it is natural to assume 
\begin{equation}
\label{A5}\tag{A5} \text{There is $\alpha \in (0,1] \cap \interior{I_\mu}$ such that $m(\alpha)=1$ and $m'(\alpha) \le 0$.}
\end{equation}
For $d=1$, assuming only ${\E N>1}$ and $\P{(T_i)_{i \ge 1} \in \{0,1\}^\N} <1$, it is shown in \cite[Theorem 6.1]{AM2010b} that $\alpha \in (0,1]$ is necessary for the existence of fixed points of $\STi$. Our assumption \eqref{A5} is slightly stronger and guarantees in addition the existence of $\epsilon >0$ such that $m(\alpha + \epsilon) <1$ as soon as $m'(\alpha) <0$. 

Let $\abs{\cdot}$ be the Euclidean norm on $\Rd$ with unit sphere $\Sd$, and let $\norm{\ma}:= \sup_{x \in \Sd}\abs{\ma x}$ be the corresponding operator norm on the set of matrices. Every allowable matrix $\ma$ acts on $\Sp$ by the definition
$$ \ma \as x := \frac{\ma x}{\abs{\ma x}}$$
and  maps $\ma \as \interior{\Sp} \subset \interior{\Sp}$. Moreover, the quantity 
$$ \iota(\ma) := \inf_{x \in \Sp} \abs{\ma x}$$
is strictly positive, and we have for any $x \in \Sp$, that $\iota(\ma) \le \abs{\ma x} \le \norm{\ma}$. 
We will use the following moment conditions, with $\mM$ having law $\mu$.
\begin{align}
\label{A6}\tag{A6} & \E \norm{\mM}^\alpha \log(1+ \norm{\mM}) < \infty, \qquad  \E \norm{\mM}^\alpha \abs{ \log \iota(\mM^\top)} < \infty \\
 \tag{A6a}\label{A6a}
& \E \norm{\mM}^\alpha \log(1+ \norm{\mM}) < \infty, \qquad \E (1+\norm{\mM})^\alpha \abs{\log \iota(\mM^\top)} < \infty \\
\label{A7}\tag{A7} & \E \norm{\mM} < \infty \\
\label{A8}\tag{A8} & 0 < \E \abs{Q}^{\alpha+\epsilon} < \infty \text{ for some } \epsilon >0.
\end{align}

\subsubsection*{A Martingale}

One additional piece of information is needed to formulate our main result in full detail. If $\mu$ is any measure on nonnegative matrices which satisfies $\condC$, $\law{\mM}=\mu$, then
the following operator in the set $\Cf{\Sp}$ of continuous functions on $\Sp$ is well defined for any $s \in I_\mu$:
\begin{align}
\label{def:Ps} \Ps[s] : \Cf{\Sp} \to \Cf{\Sp},& \qquad \Ps[s] f(x) = \E
\abs{\mM x}^s f(\mM \as x).
\end{align}

Its adjoint operator $(\Ps)'$ is a mapping on the set of bounded measures on
$\Sp$. Defining $\tilde{\Ps} \nu := [(\Ps)' \nu(\Sp)]^{-1} (\Ps)' \nu$, it
induces a continuous self-map of $\Pset(\Sp)$. By
the Schauder-Tychonoff theorem 
$\tilde{\Ps}$ has a fixed point $\nus$, say, which is in turn an eigenmeasure
of $(\Ps)' $. Denote its eigenvalue by $k(s)$. 
 It is  shown in \cite[Proposition 3.1]{BDGM2014} that $\nus$ is unique up to scaling
 and that $k(s)=(\E N)^{-1}m(s)$, i.e. 
 $$ \nus[\alpha] (\Ps[\alpha])' = \frac1{\E N} \nus[\alpha].$$ 

This property, together with the definition of $\mu$ in \eqref{eq:defmu} is fundamental for showing that
\begin{equation}\label{def:martingale} W_n(u) := \sum_{\abs{v}=n} \int_{\Sp} \, \skalar{\mL(v)^\top  u, y} \, \nus[\alpha](dy) \end{equation}
defines a martingale w.r.t.~the filtration $(\B_n)_{n \in \No}$. It is nonnegative, thus has an almost sure limit $W(u)$ for every $u \in \Sp$.

\subsubsection*{Statement of the Main Result}

\begin{thm}\label{thm:main thm}
Assume \eqref{A0}--\eqref{A7} and $m'(\alpha) <0$.
\begin{enumerate}
\item  Then $W(u)$ is positive a.s. with $0 < \E W(u) < \infty$ for each $u \in \Sp$. A random vector  $X \in \Rdnn$  is a solution of \eqref{eq:SFPE} with $Q \equiv 0$, if and only if its Laplace transform satisfies
\begin{equation} \label{eq:fphom} \E \exp(-r \skalar{u,X}) ~=~ \E \exp (-K \, r^\alpha \, W(u)) \end{equation}
for some $K \ge 0$ and all $r \in \Rnn$, $u \in \Sp$. 
\item Assume in addition \eqref{A8} and $\alpha <1$. Then $W^*$ is a.s.~finite, and a random vector $X \in \Rdnn$ is a solution of \eqref{eq:SFPE} if and only if its Laplace transform satisfies
\begin{equation}\label{eq:fpinhom} \E \exp(-r \skalar{u,X}) ~=~ \E \exp (-K \, r^\alpha \, W(u) - r \skalar{u,W^*}) \end{equation}
for some $K \ge 0$ and all $r \in \Rnn$, $u \in \Sp$. 
\end{enumerate}
If \eqref{A0}--\eqref{A6} hold with \eqref{A6} replaced by \eqref{A6a} and $m'(\alpha)=0$ then Eq. \eqref{eq:SFPE} with $Q \equiv 0$ has a nontrivial fixed point.

\end{thm}

\begin{rem}
In terms of random variables, Eq. \eqref{eq:fpinhom} states that if $X$ is a solution of \eqref{eq:SFPE}, then for all $u \in \Sp$,
$$ \skalar{u,X} \eqdist \skalar{u,W^*} + (KW(u))^{1/\alpha} Y_\alpha,$$
where $Y_\alpha$ is a one-dimensional $\alpha$-stable random variable with LT $\E \exp(-tY) = \exp(-t^\alpha)$, independent of $(W(u),W^*)$. 

A sufficient condition for the existence of $\alpha \in (0,1)$ with $m(\alpha)=1$, $m'(\alpha) <0$ is that the nonnegative matrix $\E \mM$ has spectral radius smaller than 1, see {\cite[Lemma 4.14]{BDGM2014}}.

Additional properties of the fixed points, like multivariate regular variation or a representation as a mixture of multivariate $\alpha$-stable laws are given in Theorem \ref{thm:propertiesfp}.

It can be shown that if Assumption \eqref{A4} is violated, then there are more fixed points, the situation being similar to the one-dimensional arithmetic case. See \cite[Section 18]{Mentemeier2013}.
\end{rem}

\subsection{Comparison to Previous Work.}\label{subject:comparison} Up to now, only partial results about existence and / or uniqueness of fixed points of Eq. \eqref{eq:SFPE} have been achieved: Fixed points on $\Rd$ with a finite variance were studied in \cite[Theorem 4.4]{NR2006} and in \cite{Bassetti2014}. Properties of $W^* \in \Rdnn$ were studied in \cite{Mirek2013}, its counterpart in $\Rd$ in \cite{BDMM2013}. Fixed points (in $\Rdnn$) of the homogeneous equation with finite expectation were studied in \cite{BDGM2014}. Let us point out that, except for the last one, all these existence results rely on contraction arguments, i.e. an application of the Banach fixed point theorem on  a suitable subspace of probability measures, namely those having a finite moment of order $\alpha + \epsilon$. The fixed points described below have an infinite moment of order $\alpha$ if $K>0$ and thus are not accessible by this methods. Note also, that the focus of the last three papers mentioned is mainly on studying heavy tail properties of particular fixed points.

To pursue our main subject, i.e.~finding all fixed points, we naturally will follow ideas from the one-dimensional setting, in particular those of \cite{DL1983} and \cite{ABM2012}. Besides finding the right multivariate formulation being far from trivial, the multivariate setting adds a lot of features, but also technical problems to overcome; a fundamental one being, how to generalize Eq. \eqref{eq:alphaelementaryd1}? The first question is how to define $\alpha$ in the case of matrices. Next, one has to prove the existence of the limit
\begin{equation} \label{eq:limregular}\lim_{r \to 0} \frac{1-\LTfp(ru)}{r^\alpha L(r)}, \end{equation} which will turn out to be the most involved question, and to understand how it depends on $u$. Is there a unique function $f$, such that the limit always equals $K f(u)$ and the fixed points can be characterized by this $K$, or are there several possible functions describing the directional component, each of them parametrizing a family of fixed points? 
All these questions will be solved, but therefore we have to draw on a
broad variety of methods from the theory of multitype branching random walk,
products of random matrices, Markov random walks, general branching processes
and harmonic analysis.

\subsection{Applications}
\subsubsection*{Random Walk in Random Environment on Multiplexed Trees}
The random walk in random environment (RWRE) on trees was studied in detail  first by  \cite{Lyons1992}, there very general underlying trees were considered. Of particular interest for the present situation is the work of \cite{Menshikov2002}, where it was shown that recurrence properties of a RWRE on Galton-Watson trees are intimately connected to the existence of fixed points of the smoothing transform.
 
This model has been extended in \cite{Menshikov2003,Menshikov2005} to RWRE on multiplexed trees: Let $d,N >1$. Attach to every vertex $v \in \V=\bigcup_{n=0}^\infty \{1, \cdots, N\}^n$ $d$ different levels $\{1, \dots, d\}$. To every edge connecting the vertex $vi$ with its ancestor $v$ belongs a $d \times d$-matrix $\mT(v)_i$, the $a$-th column of which consists of the transition probabilities for going from $v$ to $vi$ and changing from level $a$ to $b$, divided by the transition probability of going from $vi$ to $v$ and returning to level $a$.   Due to the (i.i.d.) random environment, the $\mT_i$ become random matrices with nonnegative entries. It is assumed in \cite{Menshikov2005} that the law of $\mT_i$ satisfies condition $\condC$, and it is proved there \cite[Section 2]{Menshikov2005}, that positive recurrence holds in this model, if and only if
$$ \skalar{\deins,Y} ~:=~\skalar{\deins, \sum_{k=0}^\infty \sum_{\abs{v}=k} \mL(v) \deins} < \infty, $$
where $\deins=(1, \dots, 1)^\top$ is the vector with all entries equal to 1, and $Y$ satisfies the inhomogeneous fixed point equation
$$ Y \eqdist \sum_{i=1}^N \mT_i Y_i + \deins.$$

\ 
\subsubsection*{Uniqueness of Fixed Points of the Inhomogeneous Multivariate
Smoothing Transform}  \cite{Mirek2013} proved the existence of the
particular fixed point $\law{W^*}$ of the multivariate inhomogeneous smoothing
transform under stronger assumptions (namely, the matrices are invertible and $\alpha \le 1/2$.)
He proved that $\law{W^*}$ is an attracting fixed point within the class of probability laws on $\Rdnn$ with a finite moment of order $s_2 =\alpha+\epsilon$ for some $\epsilon >0$, see (ibid., top of p. 679). Therefore, it is the unique fixed point within this class; but, as our result shows, not in general, for there are more fixed points with an infinite moment of order $\alpha$. 
Thus, the statement of \cite[Theorem 1.7]{Mirek2013} is misleading, there is no unique fixed point. Our result gives a full characterization of all fixed points of the inhomogeneous multivariate smoothing transform. This particularly clarifies the meaning of (ibid., Remark 1.8).

\subsubsection*{Existence of Fixed Points in the Boundary Case}
\cite{BDGM2014} considered the multivariate
homogeneous smoothing transform in the particular case $\alpha=1$ and the
left-hand derivative $m'(1^-)< 0$. In this situation, they proved existence
of a fixed point with finite expectation and moreover, that the existence of such a fixed point
implies $\alpha=1$ and $m'(1^-) < 0$. Our result now allows to treat the
general case not considered there: We prove existence and
uniqueness of fixed points for general $\alpha \in (0,1)$. Note that, in contrast to the one-dimensional case, the
existence of FPs for $\alpha <1$ cannot be deduced from the result for
$\alpha=1$ by applying the stable transformation. Moreover,
our results allow to deduce the existence of a fixed point in the boundary case
$\alpha=1$, $m'(1)=0$. 

\subsection{Acknowledgements}
The major part of this work was done while
the author held a  position at the Institute of Mathematical Statistics,
University of M\"unster which was supported by Deutsche Forschungsgeemeinschaft
via SFB 878.  An earlier version of a part of these results has been published
within the author's PhD thesis (\cite{Mentemeier2013}), written under the
supervision of Gerold Alsmeyer, M\"unster, to whom I'd like to express my
gratidude; as well as to Dariusz Buraczewski, Ewa Damek, Yves
Guivarc'h, Konrad Kolesko, Daniel Matthes and Matthias Meiners for very helpful discussions on the subject.

\section{Organization of the Paper and Outline of the Proof}

As indicated in the Introduction, a major point will be to understand the behavior of \eqref{eq:limregular} for the Laplace transforms of fixed points. To structurize the subsequent discussion, we introduce, following \cite{Iksanov2004}, two special classes of Laplace transforms:

\begin{defn}\label{defn:regular.elementary}\label{def:alpha_regular} 
Let $\alpha \in (0,1]$ and $L$ be a positive function which is slowly varying
at 0. We call a Laplace transform $\LTfp$ of a probability measure on $\Rdnn$  
\emph{$L$-$\alpha$-elementary}, if it satisfies
$$ \lim_{r \to 0} \frac{1- \LTfp(r\deins)}{L(r) r^\alpha}  \in (0, \infty).$$

It is called 
\emph{$L$-$\alpha$-regular}, if it satisfies $$ 0 < \liminf_{r \to 0} \frac{1 -
\LTfp(r\deins)}{L(r)r^\alpha} \le \limsup_{r \to 0} \frac{1-\LTfp(r\deins)}{L(r)r^\alpha} < \infty.$$ 
If $L \equiv 1$, we say that $\LTfp$ is $\alpha$-elementary ($\alpha$-regular).
\end{defn}

Then the organization of the paper is as follows: In Section \ref{sect:condC}, we study implications of assumptions \eqref{A3} and \eqref{A4}, giving additional information on $\Ps$ and introducing a Markov random walk, associated with the measure $\mu$. It then appears in a many-to-one identity in Section \ref{sect:BRWmatrices}, which contains detailed informations about the weighted branching model given by $(\mL(v))_{v \in \tree}$, inter alia the mean convergence of $W_n$. A remarkable result is Theorem \ref{thm:l1_conv_wf}, which is in the spirit of results for general branching processes as in \cite{Nerman1981,Jagers1989}, applying for the first time the renewal theorem of \cite{Kesten1974} in this context. 

With this prerequisites at hand, we turn to the proof of Theorem \ref{thm:main thm}, which will be done in several steps, each of them given in a self-contained section.
The first step is to prove that each Laplace transform of the form \eqref{eq:fphom} resp. \eqref{eq:fpinhom} is indeed a fixed point, which is done in Section \ref{sect:existence}. Subsequently, some properties of these fixed points, including a representation as mixtures of multivariate stable laws, are proved in Section \ref{sect:propFP}. 
Having proved the existence of fixed points, we are going to prove that all fixed points of $\STh$ are of the form \eqref{eq:fphom} (this will imply the analogous result for $\STi$ , see Section \ref{sect:inhom}). In order to so, we show in Section \ref{sect:allfpregular} that all fixed points of $\STh$ are $\alpha$-regular. This is done by generalizing the approach of \cite{ABM2012} to the multidimensional situation.

At the heart of our proof lies Theorem \ref{thm:uniqueness_final} which shows that an $\alpha$-regular fixed point is already $\alpha$-elementary. Its proof is given in Sections  
\ref{sect:arzelaascoli} and \ref{kreinmilman}. The basic idea is to use the Arzel\'a-Ascoli theorem to infer the existence of subsequential limits of $r^{-\alpha}(1-\LTfp(ru))$, there the proof of equicontinuity poses the most problems. This issue is solved by observing that $\alpha$-regularity implies that $u \mapsto r^{-\alpha}(1-\LTfp(ru))$ is $\alpha$-H\"older continuous for each $r$, see Section \ref{sect:arzelaascoli}. Then one has to show that all subsequential limits coincide, this is done in Section \ref{kreinmilman} by identifying the limits as bounded harmonic functions for the associated Markov random walk, which are then constant due to a Choquet-Deny type result. 
Having proved that all fixed points of $\STh$ are $\alpha$-elementary, we identify such fixed points to be of the form \eqref{eq:fphom} in Section \ref{sect:uniqueness}.

In Section \ref{sect:inhom}, we use the results proved for $\STh$ to deduce the uniqueness of fixed points of $\STi$. The final assertion of Theorem \ref{thm:main thm} about the critical case $m'(\alpha)=0$ is proved in Section \ref{sect:crit}. The final Sections \ref{sect:proofs1} and \ref{sect:proofs2} contain proofs for the results stated in Section \ref{sect:BRWmatrices}. A list of symbols and some helpful inequalities for multivariate Laplace transforms are given in the appendix.

\section{Implications of Condition $\condC$ and a Change of Measure}\label{sect:condC}

In this section, we collect important implications of our geometrical assumptions and define an {\em associated Markov random walk}, which constitutes the multivariate analogue of the random walk associated with a branching random walk. 

Let $\mM$ as before be a random matrix with law $\mu$, satisfying condition $\condC$.
Besides the operator $\Ps[s]$ introduced in \eqref{def:Ps}, 
consider also the operator \begin{equation}\label{def:Pst}\Pst[s] : \Cf{\Sp} \to \Cf{\Sp}, \qquad \Pst[s] f(x) = \E
\abs{\mM^\top x}^s f(\mM^\top \as x). \end{equation}

Initiated by 
\cite{Kesten1973}, a detailed study of these operators under condition $\condC$ can be found in \cite{BDGM2014}, based on the fundamental works
 \cite{Guivarc'h2004,Guivarch2012} for invertible matrices.
We cite for the reader's convenience the following most important properties:
\begin{proposition}\label{prop:Ps}
Assume that $[\supp \mu]$ satisfies $\condC$ and let $s \in I_{\mu}$. 
Then the spectral radii of $\Ps[s]$ and $\Pst[s]$ are both equal to $k(s)=(\E N)^{-1} m(s)$, and there is a unique probability measure $\nus[s]$ satisfying $(\Ps[s])' \nus[s] = k(s) \nus[s]$ and an (up to scaling) unique eigenfunction $H^s$ satisfying $\Pst[s] H^s = k(s) H^s$.

The function $H^s$ is strictly positive on $\Sp$, $\min\{s,1\}$-H\"older-continuous and given by
\begin{equation}\label{eq:defHs} H^s(u) ~=~ \int_{\Sp} \, \skalar{u,y}^s \, \nus[s](dy). \end{equation}
Moreover, if there is a nonnegative, nonzero continuous function $f$, satisfying $\Pst[s] f=\lambda f$ for some $\lambda >0$, then $\lambda=k(s)$, $f=cH^s$ for some $c >0$.
\end{proposition}

\begin{proof}[Source:]
\cite[Proposition 3.1]{BDGM2014}.
\end{proof}

Using formula \eqref{eq:defHs}, we see that  $H^s$ extends to a $s$-homogeneous function on $\Rdnn$, i.e.~$H^s(x)=\abs{x}^s H^s(\abs{x}^{-1} x)$.   In particular,  the identity $\Pst[s] H^s = k(s) H^s$ becomes
\begin{equation}\label{eq:propHs} H^s(x) ~=~ \frac{\E N}{m(s)} \Erw{ H^s(\mM^\top x)} ~=~ \frac{1}{m(s)} \Erw{\sum_{i =1}^N H^s(\mT_i^\top x)} \end{equation}

This allows us to introduce for any $s \in I_{\mu}$ and $u \in \Sp$ a {\em $s$-shifted} probability measure $\Prob_u^\alpha$ on $\Sp \times \Mset^\N$ by setting
\begin{align} \label{eq:manytoone1}
&~ \E_u^s \left[ f(U_0, \mM_1, \cdots, \mM_n) \right] ~:=~ \frac{(\E N)^n }{m(s)^n H^s(u)} \Erw{ H^s(\mM_n^\top \cdots \mM_1^\top u) \, f(u, \mM_1, \cdots, \mM_n) } 
\end{align} 
for all $n \in \N$ and any bounded measurable function $f : \Sp \times \Mset^n \to \R$. See \cite[Section 3.1]{BDGM2014} for details.
Note that $(\mM_n)_{n \in \N}$ are i.i.d.~with law $\mu$ under $\Prob_u^0$ for each $u \in \Sp$, which implies $H^0 \equiv 1$; and that $\Prob_u^s(U_0=u)=1$ for all $u \in \Sp$ and $s \in I_\mu$.

\subsection{The Associated Markov Random Walk}
Under $\Prob_u^s$, define the following sequences of random variables:
\begin{align*} \mPi_n ~:=&~ \mM_n^\top \cdots \mM_1^\top,  \\
U_n ~:=&~ \mPi_n \as U_0 ~=~ \mM_n^\top \as U_{n-1}, \\
S_n ~:=&~ - \log \abs{\mPi_n U_0} ~=~ S_{n-1} + (- \log \abs{\mM_n^\top U_{n-1}}).
\end{align*}
Then $(U_n, S_n)_{n \in \No}$ constitutes a {\em Markov random walk} under each $\Prob_u^s$, i.e. $(U_n)_{n \in \No}$ is a Markov chain, conditioned under which $(S_n)$ has independent (but not identically distributed) increments. It is this Markov random that will take the r\^ole of the {\em associated random walk} appearing i.a.~in \cite{DL1983} or \cite{Liu1998}.

From \eqref{eq:manytoone1}, one obtains the following comparison formula:
\begin{align} \label{eq:manytoone10}
&~ \E_u^s \left[ f((U_k, S_k)_{k \le n}) \right] ~=~ \frac{1}{k(s)^n H^s(u)} \E_u^0 \left[{ H^s(e^{-S_n}) \,  f((U_k, S_k)_{k \le n}) } \right]
\end{align} 

Assumptions \eqref{A3}--\eqref{A6}, which are in fact assumptions on $\mu$, guarantee that $S_n$ satisfies a SLLN (see \cite[Theorem 6.1]{BDGM2014}), that $(U_n, S_n)$ satisfies the assumptions of Kesten's renewal theorem (see (ibid., Proposition 7.2)) and a Choquet-Deny type lemma (see \cite[Theorem
2.2]{Mentemeier2013a}). We will make use of all of these results.

Since it is the particular instance, where \eqref{A6} enters, we state here the SLLN for $S_n$.
\begin{proposition}\label{prop:SLLN}
Assume \eqref{A0}--\eqref{A3}, \eqref{A5} and \eqref{A6}. Then
\begin{equation}\label{eq:convergenceSn} \lim_{n \to \infty} \frac{S_n}{n} ~=~ - {m'(\alpha)} \qquad \Prob_u^\alpha\text{-a.s.} \end{equation}
\end{proposition}

\begin{proof}[Source:] \cite[Theorem 6.1]{BDGM2014} \end{proof}

\section{The Weighted Branching Model with Matrices}\label{sect:BRWmatrices}

In this section, we give exhaustive information about the behavior of the family of branch weights $(\mL(v))_{v \in \tree}$, which can also be seen as a branching random walk on the semigroup of allowable matrices. By studying its action on particular vectors $u \in \Sp$ (which includes the standard basis vectors), we obtain detailed information about its asymptotic behavior. In particular, Theorem \ref{thm:l1_conv_wf} might be of interest in its own right.
Some of the results presented here require quite involved proofs, which we postpone to the (self-contained!) Sections \ref{sect:proofs1} and \ref{sect:proofs2}.

We start with some additional notation concerned with $(\mT(v))_{v \in \tree}$. Here and below, $(\Omega, \B, \Prob)$ denotes a generic probability space, rich enough to carry all the random variables we define. If no specific $\sigma$-field is mentioned, measurability is always understood with respect to the Borel-$\sigma$-field.

\subsection{The Weighted Branching Model with Matrices}\label{subsect:branchingmodel}
Recall that $\tree = \bigcup_{n=0}^\infty \N^n $ denotes the infinite
Ulam-Harris tree with root $\emptyset$, that $\abs{v}$ denotes the generation of the node $v$, that $(\T(v))_{v \in \tree}$ is a family of i.i.d.~copies of $\T$ with corresponding filtration $\B_n := \sigma \left( \T(v), \abs{v} < n\right)$ and that we defined recursively $ \mL(vi)=\mL(v) \mT_i(v)$, with $\mL(\emptyset)=\Id$.

Furthermore, introduce shift operators
$\tshift{\cdot}{v}$, $v \in \tree$: Given any function $F=f(\Ttree)$ of the weight family $\Ttree=(\T(w))_{w \in \tree}$ pertaining to $\tree$, define
\begin{equation}\label{eq:tshift} \tshift{F}{v} :=
f\left((\T(vw))_{w \in \tree}\right)\end{equation} as the same function evaluated at the weight family $(\T(vw))_{w \in \tree}$ pertaining to the subtree rooted in $w$.
Note that the family $(\T(vw))_{w \in \tree}$ has the same distribution as $\Ttree$ and is independent of
$\B_{\abs{v}}$ as well as of all other weight families pertaining to subtrees rooted at
the same level.

Instances of such functions are e.g. the branch weights $\mL(w)=\matrix{l}(w)(\Ttree)$, we obtain
$$ \tshift{\mL(w)}{v}= \mT_{w_1}(v) \mT_{w_2}(vw_1) \cdots \mT_{w_m}(vw_1 \cdots w_{m-1})$$ if $w=w_1 \cdots w_m$, and in particular
$$ \mL(vw)=\mL(v)\tshift{\mL(w)}{v}$$
for any $v, w \in \tree$. 

\medskip
Define $ R_n :=\max_{\abs{v}=n} \norm{\mL(v)}.$ Then we have the following result, which in contrast to \cite[Lemma 7.2]{Liu1998} needs the stronger assumption that $m(s)<1$ for some $s > \alpha$, which follows from Assumption \eqref{A5} combined with $m'(\alpha)<0$.

\begin{lem}\label{Lemma Rn}
Assume \eqref{A0}--\eqref{A2} and \eqref{A5} and let $m'(\alpha)<0$. Then
$ \lim_{n \to \infty} R_n = 0 \quad \Prob\text{-a.s.}$
\end{lem} 

\begin{proof}
See Section \ref{sect:proofs1}.
\end{proof}

Upon fixing $u \in \Sp$, define for each $v \in \tree$ the induced random variables
$$ U^u(v):= \mL(v)^\top \as u , \qquad S^u(v) := - \log
\abs{\mL(v)^\top u}.$$
Note that $S^u(v), U^u(v)$ are measurable w.r.t
$\B_{\abs{v}}$ and that $\mL(v)^\top u = e^{-S^u(v)} U^u(v)$.
Let $v \in \N^\No$ be an infinite branch. On the set $\{\mL(v|k) \neq 0 \, \forall \, k \in  \No\}$, the sequence $(U(v|k))_{k \in \No}$ constitutes an (inhomogeneous) Markov chain with state space $\Sp$, conditioned on which the increments of $(S^u(v|k))_{k \in \No}$ are independent. This is why we call $(U^u(v), S^u(v))_{v \in \tree}$  a {\em branching Markov random walk}. By Lemma \ref{Lemma Rn}, we have that $\lim_{\abs{v} \to \infty} S^u(v) = \infty$ a.s.

\subsection{A Many-to-one identity and a Martingale}

\begin{lem}
Assume \eqref{A0}--\eqref{A3} and \eqref{A5}. Then, with $\Prob_u^\alpha$ as in \eqref{eq:manytoone1}, it holds for all $u \in \Sp$, all $n \in \N$ and any  measurable function $f : \Sp \times \Mset^n \to \Rnn$, that
\begin{equation}\label{spinaltreeidentity}
\E_u^\alpha \left( f(U_0, \mM_1, \cdots, \mM_n) \right) ~=~ \frac{1}{H^\alpha(u)} \Erw{ \sum_{\abs{v}=n} H^\alpha(\mL(v)^\top u) \ f(u, \mL(v|1), \cdots, \mL(v)) }
\end{equation}
\end{lem}

\begin{proof} By combining the definition of $\mu$, \eqref{eq:defmu}, with the definition \eqref{eq:manytoone1} of the $\alpha$-shifted measure. \end{proof}

For the branching Markov random walk, we immediately infer the following:
\begin{cor}
Assume \eqref{A0}--\eqref{A3} and \eqref{A5}. Then for all $u \in \Sp$, $n \in \N$ and any nonnegative measurable function $f : (\Sp \times \R)^{n+1}$,
\begin{align} 
\label{eq:manytoone2} &~\frac{1}{H^\alpha(u)} \, \Erw{\sum_{\abs{v}=n} e^{-\alpha S^u(v)} H^\alpha(U^u(v)) \, f\Bigl((U^u(v|k), S^u(v|k))_{k \le n} \Bigr)} \\
=&~ \E_u^\alpha \left( \ f\Bigl((U_k, S_k)_{k \le n}) \right). 
\end{align} 
\end{cor}

Recalling the definitions \eqref{eq:defHs} of $H^\alpha(u)$ and \eqref{def:martingale} of $W_n(u)$, we obtain that
$$ W_n(u) ~=~ \sum_{\abs{v}=n} H^\alpha(\mL(v)^\top u) ~=~ \sum_{\abs{v}=n-1} \sum_{i \ge 1} H^\alpha(\mT(v)_i^\top \mL(v)^\top u), $$
and consequently, using identity \eqref{eq:propHs} for $H^\alpha$, that
$$ \E[ W_n(u) \, | \, \B_{n-1}] ~=~ \sum_{\abs{v}=n-1} \E \left[ \left. \sum_{i \ge 1} H^\alpha(\mT_i^\top \mL(v)^\top u) \, \right| \B_{n-1} \right] ~=~ W_{n-1}(u) \quad \Pfs,$$
i.e. $W_n(u)$ is a nonnegative $\Prob$-martingale for each $u \in \Sp$ which, thus it has a limit $W(u)$.

Mean convergence of
this martingale (\emph{the} intrinsic martingale in multitype branching random walk) is studied in \cite{Biggins2004} (see also \cite{Athreya2000,Jagers1989,Kyprianou2001a,Olofsson2009}). We obtain the following result:

\begin{prop}\label{thm:mean_convergence_Wn}
Assume that \eqref{A0}--\eqref{A3}, \eqref{A5} and \eqref{A6} hold, and that $m'(\alpha) <0$. In \eqref{A5}, we may assume $\alpha \in (0,\infty)$. 
Then for all $u \in \Sp$, $W_n(u)$ converges in mean to $W(u)$, i.e.
\begin{equation}
\E W(u)=W_0(u)=\int_{\Sp} \skalar{u,y}^\alpha  \nus[\alpha](dy) = H^\alpha(u) >0.
\end{equation}
Furthermore, it holds that $W(u)=w(u, (\mL(v))_{v \in \tree})$ $\Pfs$ for a measurable function $w : \Sp \times \Mset^\tree \to [0,\infty)$.
\end{prop}

\begin{proof}
By combining the SLLN \ref{prop:SLLN} with \cite[Theorem 1.1 (i)]{Biggins2004}, see Section \ref{sect:proofs1} for details. The positivity of $H^\alpha$ is an assertion of Proposition \ref{prop:Ps}.
\end{proof}

Applying Lemma \ref{Lemma Rn}, we obtain the following very useful Corollary:

\begin{cor}\label{cor:convWnF}
Let the assumptions of Proposition \ref{thm:mean_convergence_Wn} hold. If a function $F : \Rdnn \to \Rnn$ satisfies $\lim_{r \to 0} \sup_{u \in \Sp} \abs{F(ru) - \gamma}=0$ for some $\gamma \ge 0$, then
\begin{equation} \lim_{n \to \infty} \sum_{\abs{v}=n} H^\alpha(\mL(v)^\top u) F(\mL(v)^\top u)  =  \gamma \, W(u) \qquad \Pfs
  \end{equation}\label{mm3}
\end{cor}

\subsection{Stopping Lines and the Martingale}\label{subsect:stopping}

Let $\tau = \tau((\matrix{m}_k)_{k \in \N})$ be a stopping time for a sequence of matrices, of the
form $$ \tau((\matrix{m}_k)_{k \in \N}) = \inf \left\{ n \ge 0 \ : \
(\matrix{m}_k)_{k=1}^n \in A_n \right\} $$
for some sets $A_n$.
This gives rise to the \emph{homogeneous
stopping line (HSL)} $\sline[\tau]$ for a matricial branching process by the definition
\begin{equation}
\sline[\tau] := \left\{ \,  v|\tau\bigl(
\bigl(\mT_{v_k}^\top(v|k-1)\bigr)_{k \in \N}\bigr) \ : \ v \in \{1, \dots,
N\}^\N \right\}
\end{equation}
The pre-$\sline$ $\sigma$-algebra $\B_{\sline}$ associated with the
stopping line $\sline$ is defined  as $$ \B_{\sline} := \sigma\left(
 (\T(v))_{ v \text{ has no ancestor in } \sline} \right). $$
A HSL $\sline$ is called \emph{anticipating}, if $\{v \in \sline\} \in \B_{\sline}$ for all $v \in \tree$. It is called a.s.
\emph{dissecting}, if $\max\{ \abs{v} \ : \ v \in \sline \}$ is finite a.s. (see \cite[Section 7]{ABM2012}).

 The many-to-one identity remains valid under the application of a dissecting HSL:
 
 \begin{lem}\label{lem:dissectingHSL}
Assume \eqref{A0}--\eqref{A3} and \eqref{A5}. Then for all $u \in \Sp$,  any a.s. dissecting HSL defined as above, any bounded
measurable $f$,
\begin{equation}
\frac{1}{H^\alpha(u)} \E \left( \sum_{v \in \sline[\tau]}
H^\alpha(\mL(v)^\top u) f(\mL(v|1), \cdots, \mL(v))\,  \right) = \E_u^\alpha
f(\mM_1, \cdots, \mM_{\tau}).
\end{equation}
\end{lem}

\begin{proof}
Just by summing the many-to-one identity \eqref{spinaltreeidentity} over $n \in
\No$ when considering the sets $\{\tau =n \}$.
\end{proof}

Subsequently, we will focus on a particular class of HSL, namely 
$$ \slineu[t] := \{ v \in \tree \ :  S^u(v)> t, \ S^u(v|k) \le t \
\forall k < \abs{v}\}, $$
 for arbitrary $t \in \Rp$ and $u \in \Sp$. Note that these stopping lines 
 are dissecting by Lemma \ref{Lemma Rn} and anticipating as well, since $S^u(v)$
 only depends on the initial state $u$ and $(T(v|k))_{k < \abs{v}}$. Moreover, $\B_{\slineu[t]}$ is a filtration with 
 $$\B_\infty = \lim_{t \to \infty} \sigma( (\B_{\slineu[s]}, s \le t)) =
 \sigma((\T(v), v \in \tree)),$$ see the proof of \cite[Lemma
 8.7]{ABM2012} for details.
 
 The first part of the following lemma is then a direct consequence of 
Lemma \ref{lem:dissectingHSL}, applied with $f \equiv 1$.
\begin{lem}\label{prop:martingale_stopping_lines}
For each $u \in \Sp$, the family (indexed by $t \in \Rnn$)
$$ W_{\slineu[t]}(u) := \sum_{v \in \slineu[t]} \int_{\Sp} 
\skalar{\mL(v)^\top u, y}^\alpha \, \nus[\alpha](dy) = \sum_{v \in
\slineu[t]} H^\alpha(\mL(v)^\top u)$$ is a $\Prob$-martingale with
respect to the filtration $\B_{\slineu[t]}$.

Subject to $m'(\alpha)<0$, it holds that
$$ W_{\slineu[t]}(u) = \E\left( W(u) | \B_{\slineu[t]} \right) \quad \Pfs,$$
and consequently, $\Pfs$ and in $L^1(\Prob)$,
$$ \lim_{t \to \infty} W_{\slineu[t]}(u) = W(u).$$
\end{lem}
 
 Here and in what follows, $\Pfs$-convergence for $t \to \infty$ means that for
every sequence $t_n \to \infty$, there is a set of full measure, on which the
convergence takes place. This will be enough for our purposes.

\begin{proof}[Proof of Lemma \ref{prop:martingale_stopping_lines}]
See Section \ref{sect:proofs1}.
\end{proof}


\subsection{Restricted Versions of $W_n$}\label{subsect:restrictedW}

If $m'(\alpha) <0$, then  the MRW $(U_n, S_n)_{n \in \No}$ is transient  under $\Prob_u^\alpha$ with
$\lim_{n \to \infty} S_n = \infty$ a.s.~by Proposition \ref{prop:SLLN}.  Thus  $\tau_t:= \inf\{n \, : \, S_n
>t\}$ is a.s. finite, and one can define a semi-Markov process by
$$ U(t) := U_{\tau_t}, \quad R(t):= S_{\tau_t} - t$$
for all $t \in \Rnn$. Noting that $\Prob_u^\alpha$-a.s., $\tau_t = \inf \{ n \, : \, - \log \abs{\mM_n^\top \cdots \mM_1^\top u} > t\}$, we see that $\tau_t$ corresponds to the stopping line $\slineu[t]$, and Lemma
\ref{lem:dissectingHSL} yields the following very helpful identity:

\begin{cor}\label{cor:spinaltree}
For all $t \in \Rnn$, all bounded measurable $f$
\begin{equation}
\frac{1}{H^\alpha(u)} \E \left( \sum_{v \in \slineu[t]} H^\alpha(\mL(v)^\top u)
f(U^u(v), S^u(v)-t)\,  \right) = \E_u^\alpha
f(U(t), R(t)). \label{eq:spinaltreejump}
\end{equation} 
\end{cor}

It is very remarkable, that the renewal theorem of \cite{Kesten1974} applies to $(U(t), R(t))$, which then gives a nice description of the asymptotic composition of the matricial branching process:

\begin{thm}\label{thm:kestenforW}
Assume that \eqref{A0}--\eqref{A6} hold and that $m'(\alpha) <0$.
Then there is a probability measure $\rho$ on $\Sp \times \R$, with $\rho(\interior{\Sp} \times \R)=1$, such that for all $u \in \Sp$ and all $f \in \Cbf{\Sp \times \R}$,
$$ \lim_{t \to \infty}  \frac{1}{H^\alpha(u)} \E \left( \sum_{v \in \slineu[t]} H^\alpha(\mL(v)^\top u)
f(U^u(v), S^u(v)-t)\,  \right) ~=~ \int f(y,s) \, \rho(dy, ds).$$
The result remains valid, if $f$ is a nonnegative radial function, or $f(y,s)=g(y)h(s)$ for a continuous function $g: \Sp \to \Rnn$ and a bounded measurable function $h: \R \to \Rnn$.
\end{thm}

\begin{proof} It is proved in \cite[Proposition 7.2]{BDGM2014}, that the Markov random walk $(U_n, S_n)_{n \ge 0}$ under the measure $\Prob_u^\alpha$, defined in terms of a measure $\mu$ which satisfies $\condC$ and $$\int \norm{\ma^\top} \left( \abs{\log \norm{\ma^\top}} + \abs{\log \iota(\ma^\top)} \right) \, \mu(d\ma) < \infty $$ and the aperiodicity assumption \eqref{A4},  fulfills  the conditions I.1 - I.4 on \cite[page 359]{Kesten1974}. Assumptions \eqref{A4}--\eqref{A6} warrant these properties for our measure $\mu$. Thus, after an application of the many-to-one identity \eqref{eq:spinaltreejump}, we can use  \cite[Theorem 1.1]{Kesten1974}, which gives
$$ \lim_{t \to \infty} \E_u^\alpha f(U(t), R(t)) ~=~ \int f(y,s) \, \rho(dy,ds)$$ for a probability measure $\rho$ on $\Sp \times \R$, to infer the asserted convergence.

The convergence result can be rephrased as $ \Q_u^\alpha(U(t), R(t) \in \cdot) \to \rho$ weakly. The measure $\rho$  is sometimes referred
 to as the stationary Markov delay distribution. 
It follows from the expression for $\rho$ given in \cite[Theorem 1.1]{Kesten1974}, that $\rho(\Sp \times \cdot)$ is absolutely continuous, while $\rho(\cdot \times \R)$ may have atoms. Thus the weak convergence implies convergence of
 $$ \E_u^\alpha f(R(t)) \to \int f(s) \rho(\Sp \times ds)$$ for all bounded measurable radial functions, as well as for functions $f(u,s) =g(u) h(s)$, where $g: \Sp \to \R$ is bounded continuous, and $h: \R \to \R$ is bounded measurable. 

The weak convergence implies in particular that $\rho$ is a stationary measure for $(U(t), R(t))$. By part (2) of $\condC$, the stopping time $T :=\{ \inf_{n \in \in} \, : \, \mPi_n \text{ is positive }\}$ is finite $\Prob_u^\alpha$-a.s., and $\mPi_{T+k}$ is positive for all $k \ge 0$. This implies that $U_{T+k} \in \interior{\Sp}$ for any initial vector $U_0=u$ and all $k \ge 0$, hence $\rho$ as a stationary measure satisfies $\rho(\interior{\Sp} \times \R)=1$. 
\end{proof}

The main result of this section is that the above convergence in mean also holds in probability: Define 
for a nonnegative measurable function $f$,
$$ W^f_{\slineu[t]}(u) ~:=~ \sum_{v \in
\slineu[t]} H^\alpha (\mL(v)^\top u) f(U^u(v),S^u(v)-t).$$

Such random variables are particular cases of so called $\chi$-counted
populations, appearing in the study of general branching processes, see
\cite{Jagers1975,Jagers1989,Nerman1981,Olofsson2009}. Our approach uses ideas from
\cite{Jagers1989,Cohn1994}, but takes advantage by using  the renewal theorem as an ergodic theorem for the $(U(t),R(t))$, rather than using the potential of $S_n$ as in the previous works.
We are going to prove the following result:

\begin{thm}\label{thm:l1_conv_wf}
Under the assumptions of Theorem \ref{thm:kestenforW}, it holds for all $u \in \Sp$ and all $f \in \Cbf{\Sp \times \R}$, that
$$ \lim_{t \to \infty} W_{\slineu[t]}^f(u) =  \left( \int f(y,s) \, \rho(dy, ds) \right) W(u)$$
in $\Prob$-probability and in $L^1(\Prob)$.
The result remains valid, if $f$ is a nonnegative radial function, or $f(y,s)=g(y)h(s)$ for a continuous function $g: \Sp \to \Rnn$ and a bounded measurable function $h: \R \to \Rnn$.
\end{thm}

\begin{proof}
The proof consists of several steps and is given in Section \ref{sect:proofs2}.
\end{proof}

\section{Existence of Fixed Points}\label{sect:existence}

In this section, we prove the {\em existence part} of Theorem \ref{thm:main thm}, i.e. we show that every random variable with a Laplace transforms given by \eqref{eq:fphom} or \eqref{eq:fpinhom} is a solution of the homogeneous reps. inhomogeneous equation.

Therefore, we introduce first the weighted branching process, which allows us to describe iterations of $\STi$ and its action on Laplace transforms. Then we will prove the following main result.

\begin{thm}\label{thm:existence hom}
Assume \eqref{A0}--\eqref{A3}, \eqref{A5}--\eqref{A6} and $m'(\alpha) <0$. Then for all $K > 0$, the function $\LTa_0(ru):= \exp(- K r^\alpha H^\alpha(u))$, $(u,r) \in \Sp \times \Rnn$, is the LT of a multivariate $\alpha$-stable law on $\Rdnn$ with spherical measure $\nus[\alpha]$. 
\begin{enumerate}
\item $ \LTfp_0(ru) ~:=~ \lim_{n \to \infty} \STh^n \LTa_0(ru) ~=~ \E \exp(-K r^\alpha W(u))$
is a nontrivial fixed point of \eqref{eq:SFPE} with $Q \equiv 0$. 
\item Assuming in addition that $\alpha <1$ and \eqref{A8} holds, then $W^*$ as defined in Eq. \eqref{eq:defWstar} is finite a.s., and
$$ \LTfp(ru) ~:=~ \lim_{n \to \infty} \STi^n \LTa_0(ru) ~=~ \E \exp(-K r^\alpha W(u) - r \skalar{u,W^*})$$
is a nontrivial fixed point of \eqref{eq:SFPE}.
\end{enumerate}
\end{thm}

\subsection{Weighted Branching Process}

The best way to describe iterations of $\STi$ is via the weighted branching process. Given a random variable $Y \in \Rdnn$, let $(Y(v))_{v \in \tree}$ be a family of i.i.d.~copies of $Y$, which are independent of $(\T(v))_{v \in \tree}$. Then the sequence 
\begin{equation}
\label{def:WBP} Y_n := \sum_{\abs{v}=n} \mL(v) Y(v) + \sum_{\abs{w}<n}
\mL(w)Q(w),
\end{equation}
$n \in \No$, is called the \emph{weighted branching process} (WBP) associated {\red with
$Y$ and $\T$}.
It  can easily be shown then that 
\begin{equation}\label{eq:stwbp}
\law{Y_n}=\STi^n \law{Y}
\end{equation}
and moreover, that $Y_n$ satisfies the identity
$$ Y_n ~=~ \sum_{i=1}^N \mT_i(\emptyset) \tshift{Y_{n-1}}{i} + Q(\emptyset),$$
and $(\tshift{Y_{n-1}}{i})_{i \ge 1}$ is a sequence of i.i.d.~copies of $Y_{n-1}$ and independent of $(Q, (\mT_i)_{i \ge 1})$. 
Thus, if $Y_n$ converges a.s.~to a random variable $Y^*$, then this $Y^*$ is a fixed point of $\STi$. 

Observe that, subject to the assumption \eqref{A0}--\eqref{A2}, \eqref{A8} and \eqref{A5} with $m'(\alpha)<0$ and $\alpha <1$, there is $s \in (\alpha,1)$, such that $m(s)<1$ and $\E \abs{Q}^s < \infty$. Referring to Eq. \eqref{eq:Wstar}, $W^*$ then is finite  a.s.~ and is the limit of the WBP associated with $0$ and $\T$. Moreover, if $Y$ is any random variable in $\Rdnn$ with a finite moment of order $\alpha+\epsilon$ for some $\epsilon>0$, then the associated WBP converges to a.s.~to $W^*$ as well, this follows from moment calculations as in \cite[Section 3]{Mirek2013}. Thus we have the following result:

 \begin{lem}\label{lem:conv_Wn}
Assume \eqref{A0}--\eqref{A2}, \eqref{A5} with $m'(\alpha)<0$ and $\alpha <1$ and \eqref{A8}. Then the series
$$ W^*_n := \sum_{\abs{v} < n} \mL(v) Q(v)$$
converges a.s. to a random variable $W^*$, and $\law{W^*}$ is the unique FP of $\STi$ with a finite
moment of order $\alpha +\epsilon$, for any $\epsilon >0$.
\end{lem}

\subsection{Laplace Transforms}

For a random variable $Y \in \Rdnn$, its LT $\LTa(x) = \E \exp(-\skalar{x,Y})$ is well defined for all $x \in \Rdnn$.
From Eq. \eqref{eq:stwbp}, one obtains iteration formulas for the action of $\STi$ on LTs, namely
\begin{equation}
\label{eq:STofLT} \STi^n\LTa(x) ~=~ \E \left[\, \exp\left(-\skalar{x, \sum_{\abs{v}<n} \mL(v) Q(v)} \right) \, \prod_{\abs{v}=n} \, \LTa(\mL(v)^\top x) \right]. \
\end{equation}

Observe that $\STi$
defines a continuous mapping on $\Cbf{\Rdnn}$. Consequently, if $\LTa$ is the
LT of a distribution on $\Rdnn$, and $\LTfp := \lim_{n \to
\infty} \STi^n \LTa$ exists and is a LT of a distribution as well, then
this is a fixed point of $\ST$, since 
$$ \STi \LTfp = \STi ( \lim_{n\to \infty} \STi^n \LTa) = \lim_{n \to \infty}
\STi^{n+1} \LTa = \LTfp.$$

Below, we will consider the WBP associated with particular "initial" random variables, namely multivariate $\alpha$-stable ones. The next lemma describes their Laplace transforms.

\begin{lemma}\label{lem:multivariate stable laws}
Let $\nu$ be a probability measure on $\Sp$, $K >0$ and $\alpha \in (0,1]$. Then
\begin{equation}\label{eq:lt_stablelaw}
\LTa(x) := \exp\left(-K \int_{\Sp} \skalar{x,y}^\alpha \nu(dy)\right)
\end{equation}
is the LT of the multivariate $\alpha$-stable law on $\Rdnn$ with
spherical measure $\nu$.
\end{lemma}

\begin{proof}[Source:]
An idea of the proof is given in \cite{Nolan2012,Zolotarev1986}, see
\cite[Section 5.2]{Mentemeier2013} for a detailed account based on these works.
\end{proof}

\subsection{Existence of Fixed Points}\label{sect:existence}

Now we turn to the proof of Theorem \ref{thm:existence hom}. Note that below $W^* \equiv 0$ in the homogeneous case $ Q \equiv 0$.

\begin{proof}[Proof of Theorem \ref{thm:existence hom}]
\Step[1]: By Lemma \ref{lem:multivariate stable laws}, 
$$\LTa_0(x) := \exp \left( -K H^\alpha(x) \right) ~=~ \exp\left( - K \int_{\Sp} \skalar{x,y}^\alpha \nus[\alpha](dy)  \right)$$
is the LT of a probability law on $\Rdnn$.
Hence $\LTa_n := \STi^n \LTa_0$ is a sequence of LTs, and by \eqref{eq:STofLT}, for $(u,r) \in \Sp \times \Rnn$,
\begin{align*}
\LTa_n (ru) ~=&~ \E \left[   \exp\left( - r \skalar{u, W_n^*} \right) \,  \prod_{\abs{v}=n} \exp\left(- K r^\alpha H^\alpha(\mL(v)^\top u)  \right) \right] \\
=&~   \E \exp(- r \skalar{u,W_n^*} - K r^\alpha W_n(u))  .
\end{align*}
The random variables $W_n^*$ and $W_n(u)$ converge (for any $u \in \Sp$) a.s.  by Lemma \ref{lem:conv_Wn} resp. Proposition \ref{thm:mean_convergence_Wn}, hence, using the bounded convergence theorem, the sequence $\LTa_n$ converges  pointwise to a limit $\LTfp$, given by
\begin{equation}\label{eq:LTfp} \LTfp(ru) ~=~ \E \exp\left(-r \skalar{u, W^*} - K r^\alpha W(u)\right). \end{equation}
Using the continuity theorem for
multivariate LTs (see e.g. \cite[Lemma 4]{Stadtmuller1981}),
$\LTfp$ is the LT of a probability measure on $\Rdnn$,
which is then a FP of $\ST$ by the considerations above. Since $W(u)$ is not trivial by Proposition \ref{thm:mean_convergence_Wn}, Eq. \eqref{eq:LTfp} describes a one-parameter class of fixed points. 
\end{proof}

\section{Properties of the Fixed Points}\label{sect:propFP}

In this section, we will describe properties of the fixed points given by Theorem \ref{thm:existence hom}, such as multivariate regular variation, or a representation as a mixture of multivariate $\alpha$-stable laws. We are going to prove the following result. 

\begin{thm}\label{thm:propertiesfp}
Assume \eqref{A0}--\eqref{A3}, \eqref{A5}--\eqref{A6} with $m'(\alpha) <0$ and $\alpha <1$. Let either $Q \equiv 0$ (then $W^* \equiv 0$) or let \eqref{A8} hold. 
Then there is a random finite measure $\Theta$ on $\Sp$, such that $\int_{\Sp} \, \skalar{u,y}^\alpha \, \Theta(dy)= W(u)$ $\Pfs$ for all $u \in \Sp$, and hence
\begin{equation} \label{eq:levy}\LTfp(ru) ~=~ \E \exp\left(- r \skalar{u, W^*} - K r^\alpha \, \int_{\Sp} \, \skalar{u,y}^\alpha \, \Theta(dy)\right), \qquad (u,r) \in \Sp \times \Rnn.\end{equation}
Moreover, $\E \Theta = \nus[\alpha]$, and if $X$ is a r.v. with LT $\LTfp$ with $K >0$, then
\begin{equation}\label{eq:multregvar} \lim_{r \to \infty} \frac{ \P{\frac{X}{\abs{X}} \in \cdot, \abs{X}>sr}}{\P{\abs{X}>r}} ~=~ s^{-\alpha} \nus[\alpha], \end{equation}
and in particular, 
\begin{equation}\label{eq:kestenregvar} \lim_{r \to \infty} r^\alpha \P{\skalar{u,X}>r} ~=~ \frac{K H^\alpha(u)}{\Gamma(1-\alpha)}.\end{equation}
\end{thm}

\begin{rem}
Note that the heavy tail properties \eqref{eq:multregvar} and \eqref{eq:kestenregvar} are subject to the assumptions $K>0$ and $\alpha <1$. If one of those fails, the tail behavior is governed by $\beta$ rather then by $\alpha$, as shown in \cite[Theorem 2.4]{BDGM2014} for the homogeneous case $\alpha=1$ and $Q \equiv 0$, and in \cite[Theorem 1.9]{Mirek2013} for the inhomogeneous case with  $\alpha < 1/2$ and $K=0$, i.e., for the particular fixed point $W^*$.

It is remarkable that the directional dependence in the results cited above and in Eq. \eqref{eq:kestenregvar} is given by corresponding functions, namely by $H^\beta$ there resp. $H^\alpha$ here.
\end{rem}

\medskip

Define a sequence of Radon measures on $\Rd \setminus\{0\}\simeq \Sd \times \Rp$ by
$$ \Lambda_0:=\nus[\alpha] \otimes \llam^\alpha,$$
where $\llam^\alpha(dr)=\frac{1}{r^{\alpha+1}} dr$ (L\'evy measure of a one-dimensional $\alpha$-stable random variable),
and $ \Lambda_n$ being the random measure defined by
$$  \Lambda_n(\omega)(A) ~:=~ \sum_{\abs{v}=n}   \Lambda_0(\mL(v)(\omega)^{-1}(A))$$
for all measurable sets $A \in \Rd \setminus \{0\} \to \C$ which are bounded away from the origin.
Consider for $u \in \Sd$, $t>0$ the half-space 
$H_{u,t}=\{ x \in \Rd \setminus \{0\} \, : \, \skalar{u,x} > t \}$.

\begin{lem}
Under the assumptions of Theorem \ref{thm:propertiesfp}, for a.e.~$\omega \in \Omega$ the sequence $\Lambda_n(\omega)$ converges vaguely to a Radon measure $\Lambda(\omega)$ on $\Rd \setminus \{0\}$, which is of the form
$$ \Lambda(\omega)~=~\Theta(\omega) \otimes \llam^\alpha$$
for a random finite measure $\Theta$ on $\Sd$,  supported on $\Sp$. We have that $\E \Theta = \nus[\alpha]$, and for
all $u \in \Sp$, it holds that
\begin{equation} \label{eq:Theta} W(u) ~=~ \alpha \int_{\Sp} \skalar{u,y}^\alpha \, \Theta(dy) \qquad \Pfs. \end{equation}

Moreover, for all $u \in \Sp$ and $t >0$, \begin{equation}\label{eq:Lambda1}\Lambda_0(H_{u,t})= \frac{t^{-\alpha}}{\alpha} H^\alpha(u). \end{equation}
\end{lem}

\begin{proof}

Using the substitute $s=r \skalar{u,\mL(v)y}$,
\begin{align}
\Lambda_n(H_{u,t}) ~&=~ \sum_{\abs{v}=n} \Lambda_0 \left( \left\{ x \in  \Rd \setminus \{0\} \, : \, \skalar{u, \mL(v) x} > t \right\} \right) \nonumber\\
~&=~ \sum_{\abs{v}=n} \Lambda_0 \left( \{ yr \, : \, (y,r) \in \Sd \times \Rp, \skalar{u,\mL(v) y}>0; \, r > t/\skalar{u, \mL(v)y} \} \right) \nonumber \\
~&=~ \sum_{\abs{v}=n} \int_{t}^\infty \, \int_{\Sp} \1[(0,\infty)](\skalar{u, \mL(v)y}) \, \skalar{u,\mL(v)y}^\alpha  \nus[\alpha](dy) \, \frac{1}{s^{\alpha+1}} \, ds  \nonumber\\
~&=~ \frac{t^{-\alpha}}{\alpha} \, \sum_{\abs{v}=n} \int_{\Sp} \1[(0,\infty)](\skalar{u, \mL(v)y}) \, \skalar{\mL(v)^\top u,y}^\alpha  \nus[\alpha](dy)  \label{wlambda}
\end{align}
Observe that $\overline{W}_n(u):= \sum_{\abs{v}=n} \int_{\Sp} \1[(0,\infty)](\skalar{u, \mL(v)y}) \, \skalar{\mL(v)^\top u,y}^\alpha  \nus[\alpha](dy)$ defines a nonnegative martingale for all $u \in \Sd$, which coincides with  $W_n(u)$ for $u \in \Sp$. Thus $\overline{W}_n(u)$ converges a.s. to a random limit $\overline{W}(u)$, which is nontrivial for $u \in \Sp$, and equal zero for $u \in - \Sp$. 

Having thus shown that for all $u \in \Sd$ and all $t >0$,
\begin{equation}\label{eq:lambda2} \lim_{n \to \infty} \Lambda_n(H_{u,t}) ~=~ t^{-\alpha} \frac{\overline{W}(u)}{\alpha} \qquad \text{a.s.}, \end{equation}
we use \cite[Theorem 3']{BL2009} to infer that for a.e.~$\omega \in \Omega$, there is a Radon measure $\Lambda(\omega)$ on $\Rd \setminus \{0\}$ such that
$$ \lim_{n \to \infty} \, \int f(x) \, \Lambda_n(\omega)(dx) ~=~ \int f(x) \Lambda(\omega)(dx)$$
for all bounded continuous functions on $\Rd$ with support bounded away from the origin (this implies in particular the asserted vague convergence). The theorem moreover states that $\Lambda(\omega)$ is $-\alpha-d$-homogeneous (c.f. \cite[p.692]{BL2009}), i.e. of the form 
$\Lambda(\omega) = \Theta(\omega) \otimes \llam^\alpha$ for a finite measure $\Theta(\omega)$ on $\Sd$. The assertion on the support follows, since every $\Lambda_n$ is supported on $\Rdnn\setminus\{0\}$.

Applying \eqref{wlambda} with $n=0$ yields the identity \eqref{eq:Lambda1}. Then, taking expectations in Eq. \eqref{eq:lambda2}, we deduce that 
$$ \E  \Lambda(H_{u,t}) ~=~ (\E \Theta \otimes \llam^\alpha)(H_{u,t}) ~=~ \Lambda_0(H_{u,t}) ~=~ (\nus[\alpha] \otimes \llam^\alpha)(H_{u,t}). $$
By \cite[Theorem 3]{BL2009},  the measure of the half-spaces $H_{u,t}$ uniquely determines $\E \Lambda$ and $\Lambda_0$, which thus coincide. We conclude that $\E \Theta = \nus[\alpha]$.

Applying \eqref{wlambda} with $t=1$ and $u \in \Sp$, we finally obtain that
$$ \alpha \int_{\Sp} \skalar{u,y}^\alpha \, \Theta(dy) ~=~ \lim_{n \to \infty} \alpha \Lambda_n(H_{u,1}) ~=~ \lim_{n \to \infty} W_n(u) ~=~ W(u) \qquad \Pfs$$
\end{proof}

\begin{proof}[{Proof of Theorem \ref{thm:propertiesfp}}]

In the lemma above, we have already proven the first part of Theorem \ref{thm:propertiesfp}, in particular the representation formula Eq. \eqref{eq:levy}.

\medskip

Next, we prove \eqref{eq:kestenregvar}. Observe that for fixed $u \in \Sp$, the
sequence of random variables $r^{-\alpha} \left[1- \exp(-r^\alpha W(u)) \right]$
is decreasing in $r \in \Rp$: Replacing $s=r^\alpha$ and fixing a realisation of
$W(u)$, $s \mapsto [1-\exp(-sW(u))]/s$ is a LT (see \cite[XIII]{Feller1971}),
hence decreasing. Since $\alpha <1$ is assumed, then also the sequence
$ r^{-\alpha} \left[1- \exp(-K r^\alpha W(u)) -r \skalar{u,W^*} \right] $ is ultimately decreasing.
This allows to use the monotone convergence theorem and Proposition \ref{thm:mean_convergence_Wn} to
infer
\begin{align}
\lim_{r \downarrow 0} \frac{1 - \LTfp(ru)}{r^\alpha} = & \lim_{r \downarrow 0}
\E \left[ \frac{1- \exp(- K r^\alpha W(u) - r \skalar{u,W^*})}{r^\alpha} \right] \nonumber \\
= & \E \left[ \lim_{r \downarrow 0} \frac{1- \exp(-K r^\alpha
W(u) - r \skalar{u,W^*})}{r^\alpha} \right]  \\
=& \E \left[ \lim_{r \downarrow 0} \frac{1- \exp(-K r^\alpha
W(u)) }{r^\alpha} \right] = \E K W(u) = K H^\alpha(u). \label{eq:linear_reg_var}
\end{align}

Let now $X$ have Laplace
transform $\LTfp$. Using the Tauberian theorem for LTs
\cite[XIII.5, (5.22)]{Feller1971}, the relation \eqref{eq:linear_reg_var} with $\alpha <1$
implies that
\begin{equation} \lim_{r \to \infty} r^\alpha \P{\skalar{u,X}>r} =
\frac{K H^\alpha(u)}{\Gamma(1-\alpha)} \label{eq:lin_reg_var2}\end{equation}
for all $u \in \Sp$. 

\medskip

 Let's turn to the proof of \eqref{eq:multregvar}. For $\alpha \in (0,1)$, the property \eqref{eq:lin_reg_var2} implies multivariate
regular variation, i.e. there is a uniquely determined probability measure
$\rho$ on $\Sd$ such that $$ \lim_{r \to \infty} \frac{\P{\abs{X}>sr, \, X/\abs{X} \in
\cdot}}{\P{\abs{X}>r}} =  s^{-\alpha} \rho,$$
see \cite[Theorem 1.1]{BDM2002}, or equivalently,
$$ \lim_{r \to \infty} r^\alpha \E f(r^{-1}X) ~=~ \int_{\Sp} \int_0^\infty f(ru) \frac{1}{r^{1+\alpha}} \, dr \, \rho(du)$$ 
for all bounded continuous functions $f$ on $\Rd$ with support bounded away from the origin.

Referring again to \cite[Theorem 3]{BL2009}, the measure on the right hand side is uniquely identified by the value it takes on the half-spaces $H_{u,t}$. By \eqref{eq:lin_reg_var2}, 
$$ \rho \otimes \llam^\alpha (H_{u,t}) ~=~ \frac{K H^\alpha(u)}{\Gamma(1-\alpha)}, $$
and thus, recalling Eq. \eqref{eq:Lambda1}, $\rho \otimes \llam^\alpha$ is a scalar multiple of $\Lambda_0$, hence $\rho=\nus[\alpha]$.
\end{proof}

\section{Every Fixed Point of $\STh$ is $\alpha$-regular}\label{sect:allfpregular}

Having proven the existence part of the main theorem, we turn now to the much more involved proof of uniqueness. We will focus on fixed points of the homogeneous smoothing transform $\STh$, for the inhomogeneous case can then be treated by exploiting a one-to-one correspondence between fixed points of $\STh$ and $\STi$, to be proved later. 

The proof of uniqueness consists of two fundamental steps: Firstly, we show that each fixed point of $\STh$ is $\alpha$-regular; secondly, we are going to prove that every $\alpha$-regular fixed point is already $\alpha$-elementary, and that there is a unique function, namely $H^\alpha(u)$, that describes the directional behavior. Then we conclude by showing that $\alpha$-elementary fixed points are unique, using in essence Corollary \ref{cor:convWnF}.


In this section, we provide the first fundamental step.
In what follows, let $\LTa$ be (the LT of) such  a fixed point of $\STh$, and
write $$ \D(x) := \frac{1-\LTa(x)}{H^\alpha(x)} $$ as well as 
 $\deins:=(1,\dots,1)^\top$.

\begin{thm}\label{thm:allfixedpointsregular}
Assume \eqref{A0}--\eqref{A6} and $m'(\alpha)<0$. If $\LTa$ is the Laplace transform of a nontrivial fixed point of $\STh$ on $\Rdnn$, then
$ 0 <  \liminf_{r \to 0} \D(r\deins) \le \limsup_{r \to 0} \D(r\deins) < \infty$, in particular, $\LTa$ is $\alpha$-regular.
\end{thm}

The proof of this theorem will be given by the Lemmata \ref{lemK1} and \ref{lemK2} at the end of this section, extending the approach of \cite{ABM2012} to the multidimensional situation. Beforehand, we have to introduce the concept of disintegration, and provide several prerequisites, including an application of Theorem \ref{thm:l1_conv_wf}. It might be helpful to have a first glance at the proofs of Lemmata \ref{lemK1} and \ref{lemK2} after studying the section about disintegration, in order to understand where the prerequisites are needed.

\subsection{Disintegration}\label{sect:disintegration}

Assume \eqref{A0}--\eqref{A2} and let $\LTa$ be a FP of $\ST$. For each $x \in \Rdnn$, the sequence 
$$ M_n(x):= \prod_{\abs{v}=n} \LTa(\mL(v)^\top x)$$
is a bounded nonnegative martingale, which thus converges in $L^1$.

\begin{defn}\label{def:disintegration}
The random variable $M(x) ~=~ \lim_{n \to \infty} M_n(x)$ is called the {\em disintegration of
$\LTa$}. Set $Z(x):= - \log M(x) \in [0,\infty]$.
\end{defn}

Since $M_n(x)$ converges in $L^1$, it holds that $\LTfp(x) = \E\, M(x) = \E\, e^{-Z(x)}$.

\begin{lem}\label{lem:disintegration}

Assume \eqref{A0}--\eqref{A2} and \eqref{A5} with $m'(\alpha)<0$. Then for all $x \in \Rdnn$, $u \in \Sp$,
\begin{enumerate}
  \item   $\lim_{n \to \infty} \sum_{\abs{v}=n} (1- \LTa(\mL(v)^\top x) ~=~ Z(x) \qquad \Pfs,$ \label{mm2}
  \item if a function $F : \Rdnn \to \Rnn$ \label{mm3} satisfies $\lim_{r \to 0} \sup_{u \in \Sp} \abs{F(ru) - \gamma}=0$ for some $\gamma \ge 0$, then
  $$\lim_{n \to \infty} \sum_{\abs{v}=n} (1- \LTa(\mL(v)^\top x) \, F(\mL(v)^\top x)  ~=~  \gamma \, Z(x) \qquad \Pfs$$
  \item $ \lim_{t \to \infty} \sum_{v \in \slineu[t]} (1- \LTa(\mL(v)^\top u)) ~=~  Z(u)$ \qquad $\Pfs$ \label{dis2}
\end{enumerate}
\end{lem}

\begin{proof}

\eqref{mm2} Use the fact that $\lim_{n \to \infty} \max_{\abs{v}=n}
\norm{\mL(v)}=0$ from Lemma \ref{Lemma Rn} together with the convergence $M_n(x) \to M(x)$,
the approximation $- \log r \approx 1- r$, valid for $r$ close to $1$, and
$\lim_{\abs{x} \to 0} \LTa(x)=1$  to infer 
$$ Z(x) = - \log M(x) = - \log \lim_{n \to \infty} \prod_{\abs{v}=n}
\LTa(\mL(v)^\top x)  = \lim_{n \to \infty} \sum_{\abs{v} =n} (1-
\LTa(\mL(v)^\top x)).$$

\eqref{mm3} Writing
\begin{align*} &~\sum_{\abs{v}=n} (1- \LTa(\mL(v)^\top x) \, F(\mL(v)^\top x) \\ ~=&~  \gamma \sum_{\abs{v}=n} (1- \LTa(\mL(v)^\top x)  + \sum_{\abs{v}=n} (1- \LTa(\mL(v)^\top x) \bigl(F(\mL(v)^\top x) -\gamma \bigr),\end{align*} the assertion follows from \eqref{mm2}, using the uniform convergence of $F$ and Lemma \ref{Lemma Rn}.

\eqref{dis2} The same proof as for Lemma \ref{prop:martingale_stopping_lines} (given in Section \ref{sect:proofs1}) yields that
$$ \prod_{v \in \slineu[t]} \LTfp(\mL(v)^\top u)= \E \left[ M(u)| \B_{\slineu[t]} \right],
 $$ from which we infer $\lim_{t \to \infty} \prod_{v \in \slineu[t]} \LTa(\mL(v)^\top u) = M(u)$
 $\Pfs$. Then the argument is the same as for
\eqref{mm2}, using now that for $v \in \slineu[t]$, $\abs{\mL(v)^\top u} \le
e^{-t}$.
\end{proof}

%
%
%
%

\subsection{Prerequisites}\label{sect:characteristics}

We start by applying Theorem \ref{thm:l1_conv_wf} to particular
functions $f$.

\begin{lemma}\label{lem:GBP1}
For all $\epsilon >0$ there is a $c>0$  such that for all $u \in \Sp$
\begin{equation}  \lim_{t \to \infty} \sum_{v \in \slineu[t]} H^\alpha(\mL(v)^\top u) \1[\{S^u(v)-t>c
\}] ~=~ \epsilon W(u) \qquad \text{ in $\Prob$-probability}.\label{wtf1}\end{equation} 
\end{lemma}

\begin{proof}
The bounded measurable function $f_c:=\1[{(c, \infty)}] $ is radial, so by Theorem \ref{thm:l1_conv_wf},
$$ \lim_{t \to \infty} \sum_{v \in \slineu[t]} H^\alpha(\mL(v)^\top u) \1[{(c,\infty)}](S^u(v)-t)
 ~=~ \lim_{t \to \infty} W_{\slineu[t]}^{f_c}(u) ~=~ W(u) \, \int f_c(r) \, \rho(\Sp \times dr) $$
 in $\Prob$-probability, and the integral becomes arbitrarily small for $c$ large.
\end{proof}

\begin{lemma}\label{lem:GBP2}
For all sufficiently large $c>0$ there is $C>0$ such that for all $u \in \Sp$
\begin{equation} \lim_{t \to \infty} \sum_{v
\in \slineu[t]} H^\alpha(\mL(v)^\top u) \, \min_{j} U^u(v)_j \, \1[\{S^u(v)-t\le c
\}] ~=~ C W(u) \qquad \text{ in $\Prob$-probability.}\label{wtf2}\end{equation} 
\end{lemma}

\begin{proof}
The function $u \mapsto \min_j u_j$ is continuous on $\Sp$, while $r \mapsto \1[{[0,c]}](r)$ is radial and bounded measurable, so by Theorem \ref{thm:l1_conv_wf},
$$ \lim_{t \to \infty} \sum_{v
\in \slineu[t]} H^\alpha(\mL(v)^\top u) \, \min_{j} U^u(v)_j \, \1[\{S^u(v)-t\le c
\}] ~=~  \int (\min_j y_j) \, \1[{[0,c]}](r) \, \rho(dy \times dr) $$
in $\Prob$-probability. By Theorem \ref{thm:kestenforW}, $\rho(\cdot \times \Rp)$ is concentrated on
$\interior{\Sp}$, hence $\min_j u_j >0$ $\rho$-a.s., and thus upon choosing $c$ sufficiently large,
$$ C~:=~ \int (\min_j y_j) \, \1[{[0,c]}](r) \, \rho(dy \times dr) >0.$$ 
\end{proof}

\label{sect:allfps}

The next lemma generalizes \cite[Lemma 11.4]{ABM2012} to the multidimensional case.
\begin{lemma}\label{lem:estimates_D}
For all $u \in \Sp$, $t \in \R$ it holds that
$$ \frac{\D(e^{-t}u)}{\D(e^{-t}\deins)} \le
\frac{H^\alpha(\deins)}{H^\alpha(u)}:= R \qquad \text{and} \qquad
\frac{\D(e^{-t}u)}{\D(e^{-t}\deins) \, \min_j u_j} \ge R.$$ Moreover,
for all $c>0$ there is $\delta >0$ such that for all $u \in \Sp$ and all $0 \le a \le c$,
\begin{equation}
\frac{\D(e^{-(t+a)}u)}{\D(e^{-t}\deins)} \le R e^{\delta} \quad \text{
and } \quad \frac{\D(e^{-(t-a)}u)}{\D(e^{-t}\deins) \min_j u_j} \ge
R e^{-\delta}.
\end{equation}
\end{lemma}

\begin{proof}
The first inequality results from Inequality \eqref{ineq5}:
$$ \frac{\D(e^{-t}u)}{\D(e^{-t}\deins)} = \frac{1- \LTa(e^{-t}u)}{1-
\LTa(e^{-t}\deins)} \frac{H^\alpha(e^{-t}\deins)}{H^\alpha(e^{-t}u)} \le 1 \cdot
\frac{e^{-\alpha t}H^\alpha(\deins)}{e^{-\alpha t}H^\alpha(u)} = R.$$
For the second inequality, we use \eqref{ineq8},
$$ \frac{\D(e^{-t}u)}{\D(e^{-t}\deins) \, \min_j u_j} =  \frac{1-
\LTa(e^{-t}u)}{(\min_j u_j)(1- \LTa(e^{-t}\deins)}
\frac{e^{-\alpha t}H^\alpha(\deins)}{e^{-\alpha t}H^\alpha(u)} \ge R.$$

To derive the remaining inequalites, use that
 $e^{-\alpha  t}
\D(e^{-t}u)$ is decreasing in $t$:
$$ \frac{\D(e^{-(t+a)}u)}{\D(e^{-t}u)}  =
\frac{e^{-\alpha(t+a)}\D(e^{-(t+a)}u)}{e^{-\alpha(t+a)}\D(e^{-t}u)}
\le \frac{e^{-\alpha t}\D(e^{-t}u)}{e^{-\alpha(t+a)}\D(e^{-t}u)}
\le  e^{\alpha c}. $$
Setting $\delta:=\alpha c $, one obtains by taking the reciprocal that
$$ \frac{\D(e^{-(t-a)}u)}{\D(e^{-t}u)} \ge  e^{-\delta}. $$
Now plug in the first and second inequality to derive the third resp.
fourth one.\end{proof}

A priori, $Z(x) = -\log M(x) \in [0,\infty]$. Now we are going to show that in fact $Z(u) \in (0,\infty)$ $\Pfs$ for all $u \in \interior{\Sp}$.

\begin{lemma}\label{lem:WWpositive} 
For all $u \in \interior{\Sp}$,
$$ \P{Z(u)< \infty, W(u)>0}=1, \qquad \P{Z(u)>0, W(u)< \infty}=1.$$ 
\end{lemma}

\begin{proof}
As in the one-dimensional case (see \cite[Theorem 3.2]{DL1983}), one can show that $\lim_{\abs{x} \to \infty} \LTa(x) = \P{Z(u)=0}$ for all $u \in \interior{\Sp}$  equals the extinction probability of the underlying branching process which is zero due to assumption \eqref{A2}, and  is independent of $u$. In particular, $\P{Z(u) >0, W(u) < \infty} = \P{Z(u) >0}=1$, and $\P{W(u)>0}=1$ as well. Note that we do not exclude the possibility $\P{Z(u)=0}=1$ for $u \in \partial \Sp$, which might appear if the fixed point is concentrated on a subspace orthogonal to $u \in \partial \Sp$.

%
Since we assumed the fixed point to have no mass at $\infty$, $\LTa$ is continuous in $0$ and since $\LTa(u)=\E e^{-Z(u)}$,  necessarily $\P{Z(u)< \infty}=1$.
Consequently, $\P{Z(u)<\infty, W(u)>0}=1.$
%
\end{proof}


\subsection{Proof of Theorem \ref{thm:allfixedpointsregular} }

Now we can prove that all non-degenerate fixed points of $\STh$ are $\alpha$-regular. Though the final argument is close to the one given in \cite[Lemma 11.5]{ABM2012}, a lot of additional technical machinery has been applied,  inter alia the application of Kesten's renewal theorem in the lengthy proof of Theorem \ref{thm:kestenforW}.

\begin{lem}\label{lemK1}
For all $u \in \interior{\Sp}$, $\overline{K}_u :=\limsup_{r \to 0} D(ru) < \infty$.
\end{lem}

\begin{proof}
Fix $u \in \interior{\Sp}$. Then
\begin{align*}
\sum_{v \in \slineu[t]} (1- \LTa(\mL(v)^\top u)) =&~ \sum_{v \in \slineu[t]}
H^\alpha(\mL(v)^\top u) \frac{1- \LTa(\mL(v)^\top u)}{H^\alpha(\mL(v)^\top u)}
\\
\ge&~ \sum_{v \in \slineu[t]} H^\alpha(\mL(v)^\top u) \D(\mL(v)^\top u)
\1[\{ S^u(v) \le t+c\}] \\
=&~ \sum_{v \in \slineu[t]} H^\alpha(\mL(v)^\top u) \D(e^{-S^u(v)}U^u(v))
\1[\{ t < S^u(v) \le t+c\}] \\
\ge&~ R e^{-\delta} \D(e^{-(t+c)}\deins)  \, \sum_{v \in
\slineu[t]} (\min_j (U^u(v))_j) \, H^\alpha(\mL(v)^\top u) \1[\{S^u(v) \le t+c \}]
\end{align*}
Referring to Lemma \ref{lem:GBP2}, we have for $c$ large enough and letting $t
\to \infty$ along a suitable subsequence, that
$$ Z(u) \ge R e^{-\delta} \overline{K}_u C W(u)  \qquad \Pfs$$
By Lemma \ref{lem:WWpositive}, for any $u \in \interior{\Sp}$,
$\P{Z(u) < \infty, W(u) >0} >0$, thus  $\overline{K}_u<\infty$.
\end{proof}

\begin{lem}\label{lemK2}
For all $u \in \interior{\Sp}$, $\underline{K}_u := \liminf_{r \to 0} D(ru) >0$.
\end{lem}

\begin{proof}
Fix $u \in \interior{\Sp}$. Again using Lemma \ref{lem:estimates_D}, we estimate 
\begin{align*}
&~\sum_{v \in \slineu[t]} (1- \LTa(\mL(v)^\top u))\\
 =&~ \sum_{v \in \slineu[t]}
H^\alpha(\mL(v)^\top u) \D(\mL(v)^\top u) \1[\{S^u(v)\le t+c \}] + \sum_{v
\in \slineu[t]} H^\alpha(\mL(v)^\top u) \D(\mL(v)^\top u) \1[\{S^u(v)> t+c \}] \\
 =&~ \sum_{v \in \slineu[t]}
H^\alpha(\mL(v)^\top u) \D(e^{-S^u(v)}U^u(v)) \1[\{t <S^u(v)\le t+c \}]  \\
&~+
\sum_{v \in \slineu[t]} H^\alpha(\mL(v)^\top u) \D(e^{-S^u(v)}U^u(v))
\1[\{S^u(v)> t+c \}] \\ 
\le&~ e^\delta \D(e^{-t}\deins) \sum_{v \in \slineu[t]} H^\alpha(\mL(v)^\top u)
+ R\left( \sup_{s\ge t+c} \D(e^{-s}\deins)\right) \sum_{v \in \slineu[t]}
H^\alpha(\mL(v)^\top u) \1[\{S^u(v)>t+c \}].
\end{align*}
Considering Lemma \ref{lem:GBP1}, we have for $t \to \infty$ along a suitable
subsequence, that 
$$ Z(u) \le e^\delta \underline{K}_u W(u) + R \overline{K}_u \epsilon
W(u) \qquad \Pfs $$ By Lemma \ref{lem:WWpositive} $\P{Z(u)>0, W(u) <\infty}=1$. Since
$\epsilon$ can be made arbitrarily small for $c$ large, we infer that
$\underline{K}_u>0$.
\end{proof}

\section{Uniqueness of Fixed Points of the Homogeneous Equation}\label{sect:uniqueness}

In the subsequent sections, we will prove the following theorem.

\begin{theorem}\label{thm:regularfp}
Assume \eqref{A0}--\eqref{A7} and $m'(\alpha) <0$. If $\LTa$ is an $\alpha$-regular FP of $\STh$, then
there is $K >0$ such that $$ \LTa(ru) = \E \exp(- K r^\alpha  W(u)) , \qquad \text{ for all } r \in \Rnn, \, u \in \Sp.$$ 
\end{theorem}

Together with Theorem \ref{thm:allfixedpointsregular}, this completes the proof of Theorem \ref{thm:main thm} for the homogeneous equation, by proving the uniqueness of fixed points (recall that the existence of fixed points was shown in Theorem \ref{thm:existence hom}).

In this section, we are going to conclude Theorem \ref{thm:regularfp} from the general result given below, the proof of which will be given in Sections \ref{sect:arzelaascoli} and \ref{kreinmilman}.
 
 \begin{thm}\label{thm:uniqueness_final}
Assume \eqref{A1}--\eqref{A4} and \eqref{A7}. Let $\alpha \in (0,1]$, not necessarily satisfying \eqref{A5}.
\begin{enumerate}
  \item If there is an $L$-$\alpha$-regular FP of $\STh$, then  $m(\alpha)=1$, $m'(\alpha) \le 0$.
 \item  If $\LTa$ is an
 $L$-$\alpha$-regular FP of $\ST$, then for all fixed $s >0$, \begin{equation}\label{eq:regvarLT0} \lim_{r \to 0} \, \sup_{u,v \in \Sp}
 \abs{\frac{1-\LTa(sru)}{1- \LTa(rv)}- s^\alpha \frac{H^\alpha(u)}{H^\alpha(v)}} =0.\end{equation}
  \item If $\LTfp$ is a
$L$-$\alpha$-elementary FP  of $\ST$, then there is $K >0$ such that
\begin{equation}\label{eq:limLTnull} \lim_{r \to 0} \sup_{u \in \Sp}\abs{\frac{1- \LTfp(ru)}{L(r) r^\alpha} - K H^\alpha(u)} =0  . \end{equation}
\end{enumerate}
\end{thm}
 
\begin{rem}\label{rem:importantthm}This far reaching result covers the case of general slowly varying functions $L$ (non-constant $L$ become relevant in the critical case $m'(\alpha)=0$, see \cite{KM2014}) and proves that the directional dependence of $\lim_{r \to 0} ({1 - \LTfp(ru)})/({L(r) r^\alpha})$ is always given by $H^\alpha$ (answering the question raised in the Introduction in Subsection \ref{subject:comparison}) and that the convergence is uniform.  \end{rem}
 
\subsection{Proof of Theorem \ref{thm:regularfp}} 
 
Using Disintegration (see Subsection \ref{sect:disintegration}), each fixed point has the representation $\LTa(x) = \E \exp(- Z(x))$. We will show that $\Pfs$, the function $Z(x)$ is $\alpha$-homogeneous, i.e. $Z(x)=\abs{x}^\alpha Z(x/\abs{x})$, and subsequently that $Z(u) = KW(u)$ $\Pfs$ for all $u \in \Sp$.
 
 \begin{proof}[Proof of Theorem \ref{thm:regularfp}]
 We follow the proof given in \cite[Lemma 7.6 \& Theorem 10.2]{ABM2012} for the one-dimensional case.
 
 \medskip
 
 \Step[1], $Z$ is $\alpha$-homogeneous: Using  Lemma \ref{lem:disintegration} together with
property \eqref{eq:regvarLT0} from Theorem \ref{thm:uniqueness_final}, we obtain that for all $u \in \Sp$ and $s \in \Rnn$,
\begin{align}
Z(s u) ~=~ & \lim_{n \to \infty} \sum_{\abs{v}=n} (1-
\LTa(r \mL(v)^\top u)) \nonumber \\
~=~ & \lim_{n \to \infty} \sum_{\abs{v}=n} \frac{1- \LTa(s
\mL(v)^\top u)}{1- \LTa(\mL(v)^\top u)}
\left[ 1- \LTa(\mL(v)^\top u) \right] \nonumber \\
~=~ &  s^\alpha Z(u) \quad \Pfs \label{eq:Mpower}
\end{align}
Note that this implies $\LTa(su) = \E \exp(-s^\alpha Z(u))$ for all $u \in \Sp$, $s \in \Rnn$.

\medskip

\Step[2], $0 <\E\, Z(u) < \infty$ for all $u \in \interior{\Sp}$: By Lemma \ref{lem:WWpositive}, $0 < Z(u) < \infty$ $\Pfs$ for all $u \in \interior{\Sp}$, which implies $\E Z(u) >0$. Moreover, we have the monotone convergence $\lim_{s \to 0} s^{-\alpha} (1-
\exp(-s^\alpha Z(u))) = Z(u)$, thus
\begin{equation}\label{eq:limZ} \lim_{s \to 0} \frac{1- \LTa(su)}{s^\alpha}=
\lim_{s \to 0} \E \left[ \frac{1- e^{-s^\alpha Z(u)}}{s^\alpha} \right] = \E
Z(u) , \end{equation} being finite or not. But now we can refer to Lemma \ref{lemK1}, which yields finiteness of $\E Z(u)$. 

\medskip

\Step[3]: Eq. \eqref{eq:limZ} yields that $\LTa$ is $\alpha$-elementary. Thus, Eq. \eqref{eq:limLTnull} of Theorem \ref{thm:uniqueness_final} gives that there is $K>0$ with 
$$\lim_{s \to 0}  \sup_{u \in \Sp} \abs{\frac{1- \LTa(su)}{s^\alpha H^\alpha(u)} - K} ~=~ \lim_{s \to 0}  \sup_{u \in \Sp} \abs{\frac{1- \LTa(su)}{ H^\alpha(su)} - K} ~=~0,$$ which implies $\E Z(u) = K H^\alpha(u)$ for all $u \in \Sp$, and, together with Corollary \ref{cor:convWnF}, that
\begin{align*}
Z(u) ~=~  & \lim_{n \to \infty} \sum_{\abs{v}=n} 
\left( 1- \LTa(\mL(v)^\top u) \right) \nonumber \\
~=~ &  \lim_{n \to \infty} \sum_{\abs{v}=n} H^\alpha(\mL(v)^\top u) \,
\frac{1- \LTa(\mL(v)^\top u)}{ H^\alpha(\mL(v)^\top u)} ~=~ K W(u)\quad \Pfs \label{eq:Mpower}
\end{align*}
\end{proof}

\section{Proof of Theorem \ref{thm:uniqueness_final}: Compactness Arguments via the Arzel\'a-Ascoli Theorem}
\label{sect:arzelaascoli}

Let $\LTa$ be $L$-$\alpha$-regular. Here and below, we will study the family  $(\D[L,s])_{s \in \R}$ of functions on $\Sp \times \R$, given by
\begin{equation}
\D[L,s](u,t) := \frac{1- \LTa(e^{s+t}u)}{H^\alpha(e^t u) e^{\alpha s}L(e^s)} ~=~ \frac{1- \LTa(e^{s+t}u)}{H^\alpha(u) e^{\alpha(s+t)} L(e^s)}.
\end{equation}

Note that $\LTa$ is $L$-$\alpha$-elementary, if $\lim_{s \to -\infty} \D[L,s](\eins,0)$ exists and is finite and positive, where $\eins = d^{-1/2} \deins$ and that $L$-$\alpha$-regularity implies that for each fixed $(u,t) \in \Sp \times \R$, the family $(\D[L,s](u,t))_{s \in \R}$ is uniformly bounded.  

This already hints at using the Arzel\'a-Ascoli theorem. In fact, in this section, we are going to show that $L$-$\alpha$-regularity implies that the family $(\D[L,s])_{s \in \R}$ are $\min\{1,\alpha\}$-H\"older continuous in $u$ for any fixed $t \in \R$. This will imply equicontinuity of a restricted family $(\D[L,s])_{s \le s_0}$. The results of this section carry a lot of technical details, but their essence can be phrased as follows:

\begin{lem}\label{lem:arzelamain}
Assume \eqref{A0}--\eqref{A3} and let $\LTa$ be $L$-$\alpha$-regular for $\alpha \in (0,1]$ and a positive function $L$, slowly varying at 0. Then there is $s_0 \in \R$, such that the family $(\D[L,s])_{s \le s_0}$ is contained in a compact subset of $\Cf{\Sp \times \R}$ with respect to the topology of uniform convergence on compact sets. In particular, each sequence $(\D[L,s_n])_{n \ge 0}$ with $\lim_{n \to \infty} s_n =-\infty$ has a convergent subsequence. 
\end{lem}

An immediate application is given by the following corollary:

\begin{cor}\label{cor:uniform convergence}
If $\LTa(ru)=\E \, \exp(-K r^\alpha W(u))$, then
$$ \lim_{r \to 0} \sup_{u \in \Sp} \abs{\frac{1-\LTa(ru)}{r^\alpha} - K H^\alpha(u)}=0.$$
\end{cor}

\begin{proof}
On the one hand, for $\LTa$ of the given form, 
$$ \lim_{s \to -\infty} \frac{1- \LTa(e^{s+t}u)}{e^{\alpha s}} = 
e^{\alpha t} H^\alpha(u) = H^\alpha(e^t u), $$
i.e., for all $u \in \Sp$, $t \in \R$,
\begin{equation} \label{eq:conv_hs} \lim_{s \to -\infty} \D[1,s](t,u) = 1.
\end{equation}
On the other hand, by Lemma \ref{lem:arzelamain}, any sequence $(\D[1,{s_n}])_{n \in
\N}$ with $s_n \le s_0$ and $\lim_{n \to \infty} s_n =-\infty$ has a convergent subsequence, and this
convergence is uniform on compact subsets of $\Sp \times \R$. But due to
\eqref{eq:conv_hs}, the limit is always the same, hence $\lim_{s \to -\infty}
\D[L,s] =1 $ uniformly on compact subsets of $\Sp \times \R$.\end{proof}

\subsection{H\"older Continuity}

Recall that for $\LTa$ being $L$-$\alpha$-regular, 
$$  \underline{K}:= \liminf_{r \to 0} \frac{1-\LTa(r\deins)}{L(r)r^\alpha} >0,  \qquad  \overline{K}:=\limsup_{r \to 0} \frac{1-\LTa(r\deins)}{L(r)r^\alpha} < \infty.$$

\begin{lemma}\label{lem:Hoelder}
Let $\LTa$ be $L$-$\alpha$-regular. Then there is $t_0>0$ and $K > 0$
such that for all $t \in [0,t_0]$, $u,w \in \Sp$, 
\begin{equation}
\abs{ \frac{1-\LTa(tru)}{t^\alpha L(t)} - \frac{1-\LTa(trw)}{t^\alpha L(t)} }
\le K (1 \vee r) \abs{u-w}^\alpha.
\end{equation}
Moreover, let $C \subset \interior{\Sp}$ compact. Then with 
$$ K_C := \left( \min_{y \in C} \min_i y_i \right),$$
it holds that for each $r \in \Rp$ there is $t_1=t_1(r) \le t_0$ such that for
all $u,w \in C$, $t \in [0,t_1]$,
\begin{equation}\label{quotient:Hoelder}
\abs{\frac{1- \LTa(tru)}{1-\LTa(tu)} - \frac{1- \LTa(trw)}{1-\LTa(tw)}} \le
4 K K_C {\red (1 \vee r) }\abs{u-w}^\alpha .
\end{equation}
\end{lemma}

\begin{proof}
For $u,w \in \Sp$ define the vector $u \wedge w$ by $(u \wedge w)_i=\min\{u_i,
w_i\}$, $i=1,\dots,d$. Then $u-u\wedge w, w-u \wedge w \in \Rdnn$. 
Let $X$ be a r.v. with LT $\LTa$. Consider
\begin{align*}
& {\red \abs{1- \LTa(tru) - (1-\LTa(trw))}} \\ \le & \E \abs{ \exp(-tr\skalar{u,X}) -
\exp(-tr\skalar{w,X}) } \\
\le  & \E \abs{ \exp(-tr \skalar{u \wedge w, X}) \left( 1 - \exp(-tr
\skalar{u-u\wedge w,X}) \right) } \\ & + \E \abs{ \exp(-tr \skalar{u \wedge w,
X}) \left( 1 - \exp(-tr
\skalar{w-u\wedge w,X}) \right) } \\
\le & \E \abs{ 1 - \exp(-tr
\skalar{u-u\wedge w,X})} + \E \abs{ 1 - \exp(-tr
\skalar{w-u\wedge w,X})} \\
= & 1- \LTa(tr[u-u\wedge w]) + 1- \LTa(tr[w-u\wedge w]) 
\end{align*}
Due to symmetry, it is enough to consider $1-\LTa(tr[u-u\wedge w])$. Using
inequality \eqref{ineq5} and then \eqref{ineq3} resp. \eqref{ineq4}, we infer
\begin{align*}
1-\LTa(tr[u-u\wedge w]) \le\ & 1- \LTa(tr \abs{u - u \wedge w} \deins) \\
\le\ & \begin{cases}
1 - \LTa(t \abs{u - u \wedge w} \deins) &  r < 1 \\
r ( 1 - \LTa(t \abs{u - u \wedge w}\deins)) &   r \ge 1
\end{cases}
\end{align*}
Since by assumption,
$$ \limsup_{t \to 0} \frac{1- \LTa(t\abs{u - u\wedge w}\deins)}{t^\alpha
\abs{u-u\wedge w}^\alpha L(t \abs{u-u \wedge w})} \le \overline{K}$$
with $L$ slowly varying at 0, there is $t_0 >0$ and $K' > \overline{K}$ such
that {\red  $$ \frac{1- \LTa(t\abs{u - u\wedge w}\deins)}{t^\alpha L(t)
} \le K' (1 \vee r)  \abs{u-u\wedge w}^\alpha \le K' (1 \vee r) \abs{u-w}^\alpha $$}
for all $t \in [0, t_0]$. 
This proves the first assertion.

Turning now to the second assertion, write $F(x)=1-\LTa(x)$. Then for all $t
\le t_0$,
\begin{align*}
& \abs{\frac{1- \LTa(tru)}{1-\LTa(tu)} - \frac{1- \LTa(trw)}{1-\LTa(tw)}} \\
\le & \abs{\frac{F(tru)}{F(tw)}}\abs{\frac{F(tw)-F(tu)}{F(tu)}} +
\abs{\frac{F(trw)}{F(tw)}}\abs{\frac{F(tru)-F(trw)}{F(trw)}} \\
\le & (1\vee r) \abs{\frac{F(tw)-F(tu)}{t^\alpha L(t)}}\frac{t^\alpha
L(t)}{F(tu)} + (1\vee r) \abs{\frac{F(tru)-F(trw)}{(tr)^\alpha
L(tr)}}\frac{(tr)^\alpha L(tr)}{F(tru)} \\
\le & (1 \vee r) K' \abs{u-w}^\alpha \frac{s^\alpha L(t)}{F(tu)} + {\red (1 \vee r)}
K'
\abs{u-w}^\alpha \frac{(tr)^\alpha L(tr)}{F(tru)},
\end{align*}
{\red where we used \eqref{ineq1} and \eqref{ineq2} to estimate $\abs{F(trw)/F(tw)}$ by $(1 \vee r)$, and subsequently the estimate for $\abs{\frac{F(tu)-F(tw)}{(t)^\alpha
L(t)}}$ obtained above.}
To estimate further, observe that by \eqref{ineq8}
$$ \frac{F(tu)}{F(t\deins)}  \ge \min_i u_i, $$
hence
\begin{align*}
\ldots \le & {\red (1\vee r)} K' \abs{u-w}^\alpha (\min_i u_i)^{-1} \left(
\frac{t^\alpha L(t)}{F(t\deins)} + \frac{(tr)^\alpha L(tr)}{F(tr\deins)} \right)
\end{align*}
The term in the bracket is bounded by  $2/\underline{K}$ for $t \to 0$, hence
there is $t_1$, depending on $r$, such that the expression is bounded by
$4/\underline{K}$ for all $t \le t_1$.  To make the bound independent of $u$,
replace $(\min_i u_i)^{-1}$ by $K_C$. Finally, choose $K=\max\{K',
K'/\underline{K}\}.$
\end{proof}


\subsection{A Compact Subset of $\Cf{\Sp \times \R}$}

Now we are going to construct a compact subset  $\mathcal{J}_\alpha^K \subset\Cf{\Sp \times \R}$, such that there is $s_0 \in \R$ with $\D[L,s] \in \mathcal{J}_\alpha^K$ for all $s \le s_0$. 
Its definition is given below, subsequently, we prove that it is compact and that it eventually contains the $(\D[L,s])_{s \le s_0}$. The definition is subject to assumptions \eqref{A0}--\eqref{A3}, which guarantee the existence of $H^\alpha$.

 \begin{defn}\label{def:J}
For $\alpha \in (0,1)$, $K >0$ let $\mathcal{J}_\alpha^K$ be the set of continuous
functions
$$ g : \Sp \times \R \to [0, \infty)$$
satisfying
\begin{enumerate}
  \item $\sup_{u \in \Sp} g(u,0) H^\alpha(u)  \le K$, \label{1}
  \item $t \mapsto g(u,t) e^{\alpha t}$ is increasing for all $u \in \Sp$,
  \label{2}
  \item $t \mapsto g(u,t) e^{(\alpha-1)t}$ is decreasing for all $u \in \Sp$,
  \label{3}
  \item $ u \mapsto g(u,t) H^\alpha(e^t u)$ is $\alpha$-H\"older with constant
  $(1 \vee e^t)K$ for each $t \in \R$ \label{4}
\end{enumerate}
\end{defn}

The idea of this construction goes back to \cite{DL1983} for the one-dimensional case, in fact, properties \eqref{1}--\eqref{3} are the same as in (ibid., Lemma 2.11). The fundamental new contribution here is to take care of the directional dependence on $u \in \Sp$, which necessitates the assumption of H\"older continuity in the directional component. Note that we needed $\LTa$ to be $L$-$\alpha$-regular in order to prove H\"older continuity of $\D[L,s]$; and therefore had to show first that any fixed point is regular. This step is not needed for the one-dimensional arguments.

\begin{lem}\label{lem:J}
Assume \eqref{A0}--\eqref{A3}. The set $\mathcal{J}_\alpha^K$ is a compact subset of $\Cf{\Sp \times \R}$ w.r.t. to the
topology of uniform convergence on compact sets.
\end{lem}

\begin{proof}
The assertion will follow from the general Arzel\`{a}-Ascoli theorem for locally
compact metric spaces, see e.g. \cite[Theorem 7.18]{Kelley1955}. Properties
\eqref{1}-\eqref{3} together imply the uniform bounds, valid for all $g \in
\mathcal{J}_\alpha^K$
\begin{align}\label{uniform_bound_J}
g(u,t) \le \begin{cases}
K H^\alpha(u)^{-1} e^{(1-\alpha)t} & t \ge 0, \\
K H^\alpha(u)^{-1} e^{-\alpha t} & t \le 0 .
\end{cases}
\end{align}
Properties \eqref{1}-\eqref{4} are closed even under pointwise convergence of
functions, thus $\mathcal{J}_\alpha^K$ is particulary closed under compact uniform
convergence. Turning to equicontinuity, fix $(u_0, t_0) \in \Sp \times \R$ and
$\epsilon >0$ and consider first the
variation in $t$.  Let $\delta >0$. Then for any $g \in \mathcal{J}_\alpha^K$, it follows
from property \eqref{2} that for all $u \in \Sp$ and $t \in [t_0 - \delta, t_0 + \delta]$, 
$$ g(u,t)e^{\alpha(t_0-\delta)} \le g(u,t)e^{\alpha t} \le
g(u_,t_0 + \delta)e^{\alpha(t_0 + \delta)},$$ thus
$ g(u,t) \le g(u, t_0 + \delta)e^{2 \alpha \delta}.$
Similarly, from property \eqref{3}, 
$ g(u,t) \ge g(u, t_0+\delta)e^{2(\alpha-1)\delta}$ and consequently
\begin{align*} \abs{g(u,t)-g(u,t_0)} \le\  g(u,t_0 +
\delta)e^{2\alpha\delta} - g(u, t_0 +
\delta)e^{2(\alpha-1)\delta} \nonumber \le\  M\left(e^{2\alpha \delta}
- e^{2(\alpha-1) \delta}\right)  , \end{align*}
where the uniform bound $M$ exists due to \eqref{uniform_bound_J}. Hence there
is $\delta_1 >0$ such that
\begin{equation} \abs{g(u,t)-g(u,t_0)} < \frac{\epsilon}{2} \label{eqn:equicontinuous1}\end{equation}
for all $t \in B_{\delta_1}(t_0)$ and all $u \in \Sp$.
Considering the variation in $u$, it follows again from \eqref{uniform_bound_J}
that for $h(u,t):=g(u,t)H^\alpha(e^t u)$,
$$L := \sup \{ h(u,t) \ : \ g \in
\mathcal{J}_\alpha^K,\ (u,t) \in \Sp \times [t_0 - \delta_1, t_0 + \delta_1] \} < \infty.
$$ 
Using property \eqref{4}, we infer that for all $u \in
 \Sp$,
\begin{align*}
\abs{g(u,t_0) -g(u_0,t_0) } \le\ & \frac{\abs{h(u,t_0)-
h(u_0,t_0)}}{H^\alpha(e^{t_0}u_0)} + h(u_,t_0)\abs{ \frac{1}{H^\alpha(e^{t_0}u_0)} -
\frac{1}{H^\alpha(e^{t_0}u_0)} } \nonumber \\ 
\le\ & \frac{K(1\vee e^{t_0}) \abs{u-u_0}^\alpha}{H^\alpha(e^{t_0}u_0)}
+ L \abs{ \frac{1}{H^\alpha(e^{t_0}u_0)} -
\frac{1}{H^\alpha(e^{t_0}u_0)} } \end{align*}
Hence there is $\delta_2>0$ such that
\begin{equation} \abs{g(u,t_0) -g(u_0,t_0) } \le \epsilon/2
\label{eqn:equicontinuity2}\end{equation} for all $u \in B_{\delta_2}(u_0)$.
Combining \eqref{eqn:equicontinuous1} and \eqref{eqn:equicontinuity2}, it holds
that for all $(u,t) \in B_{\delta_2}(u_0) \times B_{\delta_1}(t_0)$, $$
\abs{g(u,t)- g(u_0,t_0)} \le \abs{g(u,t) - g(u,t_0)} + \abs{g(u,t_0)-g(u,t)} \le
\epsilon. $$ 
This proves the equicontinuity, hence Arzel\`{a}-Ascoli applies and yields the
assertion.
\end{proof}

The next result in particular proves Lemma \ref{lem:arzelamain}.

\begin{lem}\label{lem:uniform convergence}
Assume \eqref{A0}--\eqref{A3} and let $\LTa$ be $L$-$\alpha$-regular for some $\alpha \in (0,1]$. Then there is $s_0 \in \R$ and $K >0$
such that $\left(\D[L,s](u,t)\right)_{s \le s_0} \subset \mathcal{J}_\alpha^K$.
\end{lem}

\begin{proof}
We have to check properties \eqref{1}-\eqref{4}:
\begin{itemize}
  \item[\eqref{1}] $\sup_{u \in \Sp} \D[L,s](u,0) H^\alpha(u) \le
  \frac{1- \LTa(e^s \deins)}{e^{\alpha s}L(e^s)} \le K_s$, with {\red $K_s$
  bounded by $\overline{K}$ asymptotically }
  \item[\eqref{2}] Just observe that $\D[L,s](u,t) e^{\alpha t}= \frac{1-
  \LTa(e^{s+t}u)}{H^\alpha(u)e^{\alpha s}L(e^s)}$ is increasing as a
  function of $t$.
  \item[\eqref{3}] Recall that $(e^{-s})e^{-t}(1- \LTa(e^{s+t}u))$ is a LT,
  hence decreasing. Consequently,\\ $\D[L,s](u,t) e^{(\alpha -1)t} =
  e^{-t}\frac{1-\LTa(e^{s+t}u)}{H^\alpha(u)e^{\alpha s}L(e^s)}$ is
  decreasing as a function of $t$ as well.
  \item[\eqref{4}] This is the content of Lemma \ref{lem:Hoelder}. It gives
  $e^{s_0} >0$ and $K >0$ such that for all $s < s_0$, 
  $$ \abs{\D[L,s](u,t)H^\alpha(e^t u) - \D[L,s](w,t)H^\alpha(e^t w)} \le K(1 \vee
  e^t) \abs{u-w}^\alpha .$$
\end{itemize}
Possibly by making $s_0$ smaller, $K_s \le K$ for all $s \le s_0$, i.e. property
\eqref{1} holds with this $K$ as well.
\end{proof}

\section{Proof of Theorem \ref{thm:uniqueness_final}: Choquet-Deny Arguments}
\label{kreinmilman}

This section contains the technical cornerstone in the proof of Theorem \ref{thm:uniqueness_final}. 
We have proved so far that for any $L$-$\alpha$-regular FP $\LTa$, its
associated sequence $\D[L,s]$ has convergent subsequences (for $s \to
- \infty$), now we are going to identify their limits as bounded harmonic functions for the associated Markov random walk $(U_n, S_n)_{n \in \No}$. The following Choquet-Deny type result holds for $(U_n, S_n)_{n \in \No}$:

\begin{prop}\label{Choquet} Assume \eqref{A0}--\eqref{A4} and that $H \in \Cbf{\Sp \times \R}$ satisfies
\begin{enumerate}[(i)]
  \item $H(u,t) = \E_{u} H(U_1,t-S_1)$ for all $(u,t) \in \Sp \times \R$, and
  \item \label{strong continuity} for all $z \in \interior{\Sp}$,
  $$ \lim_{y \to z} \sup_{t \in \R} \abs{H(y,t)- H(z,t)}=0.$$
\end{enumerate}
Then $H$ is constant.
\end{prop}

\begin{proof}[Source:] \cite[Theorem
2.2]{Mentemeier2013a} \end{proof}

In order to apply this Choquet-Deny type result, we will introduce a subset $H_{\alpha,c}^K$ of $\mathcal{J}_\alpha^K$ (Definition and Lemma \ref{defnlem:H}) which contains all possible subsequential limits of $\D[L,s]$ (Lemma \ref{lem:regularfixedpointslimit}). Then we prove that $H_{\alpha,c}^K$  is a compact convex set, and we identify its extremal points by using Proposition \ref{Choquet} (Lemma \ref{lem:extremal_points}).  Finally, we prove Theorem \ref{thm:uniqueness_final}.

\subsection{The Set Containing the Subsequential Limits}

We start by introducing the subset $\mathcal{H}_{\alpha,c}^K$ of $\mathcal{J}_\alpha^K$ which will
contain the subsequential limits of $\D[L,s]$ for $L$-$\alpha$-regular fixed
points.

\begin{defnlem}\label{defnlem:H}
Let \eqref{A0}--\eqref{A3} and \eqref{A7} hold.  For $\alpha, c \in (0,1]$, define the subset
$\mathcal{H}_{\alpha,c}^K \subset \mathcal{J}_\alpha^K $ as follows: A function $g \in \mathcal{J}_\alpha^K$
is in $\mathcal{H}_{\alpha,c}^K$, if it satisfies the following additional properties:
\begin{itemize}
  \item[(1')]\label{1'} $\sup_{u \in \Sp} g(u,0)H^\alpha(u) = c$ and
  $g(u,0)H^\alpha(u) \ge \min_i u_i$ for all $u \in \Sp$.
  \item[(5)]\label{5} For all $(u,t) \in \Sp \times \R$, $$ g(u,t) = m(\alpha)
  \E_{u}^\alpha\ g(U_1,t-S_1).$$
  \item[(6)]\label{6} Introducing 
  $$ L_t : \Sp \times \R \to \Rp, \qquad (u,z) \mapsto
  \frac{g(u,t+z)}{g(u,z)},$$
  the following holds: For all $t \in \R$, all compact $C \in \interior{\Sp}$,
  all $u,w \in C$:
  $$ \sup_{z \in \R} e^{-\alpha t} \abs{L_t(u,z)-L_t(w,z)} \le 4 K K_C {\red (1 \vee
  e^{t}) }\abs{u-w}^\alpha,$$
  with $K_C := \left( \min \{ y_i \ : \ y \in C, \ i=1, \dots, d \}
  \right)^{-1}$.
\end{itemize}
Here, validity of (1') and (5) implies that the function $L_t$ is
well defined and continuous on $\Sp \times \R$.

The set $\mathcal{H}_{\alpha,c}^K$ is a compact subset
of $\Cf{\Sp \times \R}$ w.r.t the compact uniform convergence.
\end{defnlem}

Property (6) will provide the uniform continuity 
needed in the Choquet-Deny-Lemma \ref{Choquet}.

\begin{proof}
Note that \eqref{A6} together with \eqref{A2} implies that $[0,1] \in I_\mu$, hence $\Prob_u^\alpha$ is well defined for all $\alpha \in (0,1]$.

The function $L_t$ is well defined and continuous as soon as $g(u,z)>0$ for all
$(u,z) \in \Sp \times \R$. For $u \in \interior{\Sp}$, this is a direct
consequence of (1'), combined with the lower bounds
$$ g(u,t) \ge \begin{cases} g(u,0)e^{-\alpha t} & t \ge 0, \\
g(u,0)e^{(1-\alpha)t} & t \le 0.
\end{cases}
$$ But due to property (2) of condition $\condC$, $\Prob_u^\alpha\{U_n \in
\interior{\Sp}\} >0$ for all $u \in \Sp$ and some $n$, hence using property (5), $g(u,t)>0$
everywhere.

Since $\mathcal{J}_\alpha^K$ is compact, it suffices to show that the
subset $\mathcal{H}_{\alpha,c}^K$ is closed. It is readily checked that properties (1'),
(6) persist to hold even under pointwise convergence of functions $g_n \to g$.
In order to show the closedness of property (5), uniform integrability of the
sequence $g_n(U_1, t-S_1)$ w.r.t the measures $\Prob_{u}^\alpha$ is needed. This
is the content of the subsequent lemma.
\end{proof}

\begin{lemma}\label{lem:uniform integrable}
Assume \eqref{A0}--\eqref{A3} and \eqref{A7} and let $\alpha \in (0,1]$.
Then for all $(u,t) \in \Sp \times \R$, the family $\left\{ g(U_1, t-S_1) \ : \
g \in \mathcal{J}_{\alpha}^K \right\}$ is uniformly integrable w.r.t~$\Prob_{u}^\alpha$.
\end{lemma}

\begin{proof}
Recalling the uniform bounds \eqref{uniform_bound_J}, valid for all $g \in
\mathcal{J}_\alpha^K$ and the finiteness of 
$$ C:= \sup_{y \in \Sp} H^\alpha(y)^{-1} = \frac{1}{\inf_{y \in \Sp}
H^\alpha(y)}$$ due to Proposition \ref{prop:Ps}, it is
sufficient to show that 
$$ e^{(1-\alpha) (t-S_1)} \1[{t\ge S_1}] + e^{-\alpha (t-S_1)} \1[t \le S_1]$$
is integrable w.r.t. $\Prob_{u}^\alpha$. Using the definition of
$\Prob_{u}^\alpha$, \eqref{eq:manytoone1},
\begin{align*}
\E_{u}^\alpha e^{(1-\alpha)(t-S_1)} ~= &~ \frac{1}{H^\alpha(u)} \E_{u}^0
H^\alpha(U_1) e^{\alpha (t-S_1)} e^{(1-\alpha)(t-S_1)} \\
~\le&~  C (\E N)  e^t \E
\abs{\mM^\top u} ~\le~ C  (\E N) e^t \E \norm{\mM} , \\
\E_{u}^\alpha e^{-\alpha (t-S_1)} ~=&~ \frac{1}{H^\alpha(u)} \E_{u}^0
H^\alpha(U_1) e^{\alpha (t- S_1)} e^{-\alpha (t-S_1)} ~\le~ C \E N, 
\end{align*}
hence it is integrable due to  assumption \eqref{A6}.
\end{proof}

\begin{lemma}\label{lem:regularfixedpointslimit}
Let $\LTfp$ be a $L$-$\alpha$-regular FP of $\STh$, with
associated sequence $\D[L,s]$. Then for any sequence $(s_k)_{k \in \N}$ with
$s_k \to -\infty$, there is a subsequence $(s_n)_{n \in \N} \subset (s_k)_{k
\in \N}$ such that $\D[\infty]= \lim_{n \to \infty} \D[{L,s_n}]$ exists and is
an element of $\mathcal{H}_{\alpha,c}^K$ for some $c >0$. The convergence is uniform on compact subsets of~ $\Sp \times \R$.
\end{lemma}

\begin{proof}
By Lemma \ref{lem:uniform convergence}, $(\D[L,s])_{s \le s_0} \in \mathcal{J}_{\alpha}^{K}$
for some $K>0$ and $s_0 \in \R$. Hence for any sequence $s_k \to -\infty$ there
is a  subsequence $(s_n)_{n \in \N}$, such that $\D[L,s_n]$ converges and its limit $\D[\infty]$ 
is again an element of $\mathcal{J}_{\alpha}^K$. By Lemma \ref{lem:J}, the convergence is uniform on compact
sets. Thus the burden of the proof is to show that the additional properties
(1'), (5) and (6) hold for the limit $\D[\infty]$.

\medskip

\Step{1, Property (1'):} Using \eqref{ineq8}, we
infer that $$1-\LTfp(e^s u) \ge 1- \LTfp(e^s \min_{i} u_i \deins)  \ge \min_{i}
u_i (1- \LTfp(e^s \deins)).$$
Thus for any $s$,
$$ \D[L,s](u,0) H^\alpha(u) = \frac{1- \LTfp(e^s u)}{e^{\alpha s}L(e^s)} \ge
\min_{i} u_i \frac{1-\LTfp(e^s\deins)}{e^{\alpha s}L(e^s)},$$
and this is bounded from below by $(\min_{i} u_i )\underline{K}$ for $s \to -\infty$, since $\LTfp$ is
$L$-$\alpha$-regular. This proves the lower bound for $\D[\infty](u,0)$. Due to
property
\eqref{1}, it is also bounded from above, thus $c:= \sup_{u \in \Sp}
\D[\infty](u,0)H^{\alpha}(u)$ exists.

\medskip

\Step{2, Property (5)}: Write
\begin{equation}
G(x):= \E \left[ \prod_{i=1}^N
\LTfp(\mT_i^\top x) +  \sum_{i=1}^N (1- \LTfp(\mT_i^\top
 x)) - 1 \right] .
\end{equation}
$G(x) \ge 0$ by a simple translation of the arguments in \cite[Lemma
2.4]{DL1983}.  We use that $\STh \LTfp = \LTfp$, a linearization  and the many-to-one identity \eqref{spinaltreeidentity} to derive the
following:
\begin{align*}
\D[L,s](u,t)  = &  \frac{1-
\LTfp(e^{s+t}u)}{H^\alpha(e^tu)L(e^s)e^{\alpha s}} =  \frac{1-
\E \prod_{i=1}^N \LTfp(e^{s+t}\mT_i^\top u)}{H^\alpha(e^tu)L(e^s)e^{\alpha s}} \\
= &  \frac{\E \sum_{i=1}^N (1-
\LTfp(e^{s+t}\mT_i^\top u))}{H^\alpha(e^tu) L(e^s)e^{\alpha s}} -
\frac{G(e^{s+t}u)}{H^\alpha(e^tu) L(e^s)e^{\alpha s}}  \\
= &  \frac{1}{H^\alpha(u)} \, \E \left( \sum_{i=1}^N \frac{(1-
\LTfp(e^{s+t}\mT_i^\top u))}{H^\alpha(e^t \mT_i^\top u) L(e^s)e^{\alpha s}} H^\alpha( \mT_i^\top u) \right) -
\frac{G(e^{s+t}u)}{H^\alpha(e^tu) L(e^s)e^{\alpha s}}  \\
= & m(\alpha) \E_u^\alpha \left( \frac{1-
\LTfp(e^{s+t}\mM^\top u)}{H^\alpha(e^t \mM^\top u) L(s) e^{\alpha s}} \right) - \frac{G(e^{s+t}u)}{H^\alpha(e^tu) L(e^s)e^{\alpha s}} \\
  = & m(\alpha) \E_{u}^\alpha \D[L,s](U_1,t-S_1) -
 \frac{L(e^{s+t})}{L(e^s)H^\alpha(u)}
 \frac{G(e^{s+t}u)}{L(e^{s+t})e^{\alpha(s+t)}}
\end{align*}
Due to Lemma \ref{lem:uniform integrable}, the sequence $\D[{L,s_n}](U_1,t- S_1)$
is uniformly integrable w.r.t $\Prob_{u}^\alpha$, hence it remains to show that the
second part tends to zero for $s \to - \infty$. The following argument follows closely the ideas of \cite[Lemma
2.6]{DL1983}. Since $L$ is slowly varying at
$0$, the quotient $ L(e^{s+t})/L(e^s)$ is bounded when $s$ tends to $- \infty$. 

Consider $G(ru)$, $r \in \Rp$, $u \in \Sp$. Defining the increasing
function $$f:\Rp \times \to \Rp ,\qquad  f(s) = 
e^{-s} + s -1 
, $$
and using the inequality $s \le e^{-(1-s)}$ as well as \eqref{ineq7}, we
calculate
\begin{align*}
G(ru) & \le~ \E \left[ \exp\left(-(\sum_{i=1}^N (1-\LTfp(\mT_i^\top ru))\right)
+ \sum_{i=1}^N (1- \LTfp(\mT_i^\top ru)) -1 \right] \\
&=~ \E\, f\left(\sum_{i=1}^N (1- \LTfp(\mT_i^\top ru))\right) \le \E\, f \left(
\sum_{i=1}^N ( \norm{\mT_i} \vee 1) (1-\LTfp(r\deins)) \right)
\end{align*}
Writing $C(\T)=\sum_{i=1}^N (\norm{\mT_i} \vee 1)$, use that $\E C(\T) \le
(\E N)(1+ \E \norm{\mM}) < \infty$, the regular variation of $\LTfp$ and $\lim_{s
\to 0} f(s)/s$ to deduce that
\begin{align*} 0 \le \limsup_{r \to 0} \sup_{u \in \Sp}
\frac{G(ru)}{L(r)r^\alpha} \le & \limsup_{r \to 0} \E\left[\frac{
f\Bigl(C(T) (1-\LTfp(r\deins))  \Bigr)}{C(\T)(1-\LTfp(r\deins))} C(\T)
\frac{1- \LTfp(r\deins)}{L(r)r^\alpha} \right]= 0\end{align*}
Consequently, for the limit
$\D[\infty](u,t) = m(\alpha) \E_{u}^\alpha \D[\infty](U_1,t- S_1)$.

\medskip

\Step{3, Property (6):} Fix $t \in \R$, $C \subset \interior{\Sp}$ compact and
compute for $u,w \in C$, $z \in \R$:
\begin{align*}
 e^{-\alpha t} \abs{\frac{\D[L,s](u,t+z)}{\D[L,s](u,z)} - \frac{\D[L,s](w,t+z)}{\D[L,s](w,z)}}
=  \abs{\frac{1- \LTfp(e^{s+t+z}u)}{1- \LTfp(e^{s+z}u)} - \frac{1-
\LTfp(e^{s+t+z}w)}{1- \LTfp(e^{s+z}w)}} .
\end{align*}
Using Lemma \ref{lem:Hoelder}, there is $s_1 \in \R$ such that the right hand side is bounded by 
$$ 4 K K_C {\red (1\vee e^t )} \abs{u-w}^\alpha$$
as soon as
$e^{s+z} \le s_1$. For any fixed $z$, this condition is satisfied eventually
when taking the limit $s_n \to - \infty$. Hence in the limit,
$$ e^{- \alpha t} \abs{\frac{\D[L,s](u,t+z)}{\D[L,s](u,z)} -
\frac{\D[L,s](w,t+z)}{\D[L,s](w,z)}} \le 4 K K_C {\red(1 \vee e^{t})}
\abs{u-w}^\alpha$$ for all {\red $z \in \R$. }
\end{proof}

\subsection{Extremal Points of $\mathcal{H}_{\alpha,c}^K$}

As a compact subset of a locally convex topological space, namely $\Cf{\Sp
\times \R}$, the set $\mathcal{H}_{\alpha,c}^K$ (if non-void) is contained in the convex
hull of its extremal points due to the Krein-Milman theorem \cite[Theorem
V.8.4]{Dunford1958}. Using Proposition \ref{Choquet}, we now compute all
possible extremal points.

\begin{lemma}\label{lem:extremal_points}
Let \eqref{A0}--\eqref{A4}, \eqref{A7} hold and $\alpha \in (0,1]$.
The extremal points of $\mathcal{H}_{\alpha,c}^K$ are contained in the set
\begin{equation}
\mathcal{E}_{\alpha,c}  := \left\{ (u,t) \mapsto c
\frac{H^\chi(e^tu)}{H^\alpha(e^tu)} \ : \ \chi \in (0,1], \ m(\chi)=1. \right\}
\end{equation}
\end{lemma}

Recall that the functions $H^s(\cdot)$ are $s$-homogeneous, thus
$$\frac{H^\chi(e^tu)}{H^\alpha(e^tu)} ~=~ \frac{e^{\chi t} H^\chi(u)}{e^{\alpha t} H^\alpha(u)} ~=~ e^{(\chi-\alpha)t} \frac{ H^\chi(u)}{ H^\alpha(u)}.$$

\begin{proof}
Let $g \in \mathcal{H}_{\alpha,c}^K$ be extremal.

\medskip

\Step[1]: Use property (5) to compute for all $u \in \Sp$, $s,t \in \R$
\begin{align}
g(u,t+s) =\ & m(\alpha) \E_{u}^\alpha g(U_1,t+s-S_1) 
 \nonumber \\ =\ & m(\alpha) \int g(y, t+s-z)\ \Prob_u^\alpha(U_1
\in dy, S_1 \in dz) \label{convexkombi1} \\ 
=\ & m(\alpha) \int \frac{g(y,t+s-z)}{g(y,s-z)} g(u,s)
\frac{g(y,s-z)}{g(u,s)}\ \Prob_u^\alpha(U_1 \in dy,  S_1 \in dz)
\label{convexkombi}
\end{align}
Recall that by Lemma \ref{defnlem:H}, $g>0$,
thus the denominators are positive.
Using \eqref{convexkombi1} with $t=0$, it follows that
$$ m(\alpha) \int \frac{g(y,s-z)}{g(u,s)}\ \Prob_u^\alpha (U_1 \in dy, S_1 \in
dz) =1.$$
Hence \eqref{convexkombi} is a convex combination of functions
$g_{y,z}(u,s,t)=\frac{g(y,t+s-z)}{g(y,s-z)}g(u,s)$. Consequently, since $g$ is extremal,
\begin{equation}
\label{gextrem} \frac{g(u,t+s)}{g(u,s)}=\frac{g(y,t+s-z)}{g(y,s-z)}
\end{equation} for all $u \in {\Sp}$, $t,s \in \R$ and all $(y,z)
\in \supp \ \Prob_u^\alpha ((U_1, S_1) \in \cdot)$. But this support is the same
as $\supp \ \Prob_u((U_1, S_1) \in \cdot)$. This yields that $L_t(u,s):=
\frac{g(u,t+s)}{g(u,s)}$ satisfies \begin{equation} \label{Gtharmonic} L_t(u,s)
=
\E_{u}{L_t(U_1,s-S_1)}. \end{equation}

\medskip

\Step[2]: Proposition \ref{Choquet} will be applied in
order to show that $L_t$ is constant on ${\Sp} \times \R$,
i.e. equation \eqref{gextrem} holds for all $u,y \in {\Sp}$, $z,
s, t \in \R$.  The aperiodicity of $\mu$, i.e. assumption \eqref{A4}, enters here. Property (6) yields
condition (ii) of the proposition, while \eqref{Gtharmonic} is its condition (i).  It
remains to show that $L_t$ is bounded (for fixed t). If $t \le 0$, by property
\eqref{2}, $g(u,t+s)e^{\alpha (t+s)} \le g(u,s) e^{\alpha s}$, thus 
$$ 0 < L_t(u,s) = e^{- \alpha t} \frac{g(u,t+s)e^{\alpha
(t+s)}}{g(u,s)e^{\alpha s}} \le e^{-\alpha t}.$$ For $t \ge 0$, use property
\eqref{3} for an analogue argument. Referring also to Lemma \ref{defnlem:H},
$L_t \in \Cbf{\Sp \times \R}$, hence as a bounded harmonic function, it is
constant.

\medskip

\Step[3]: Validity of \eqref{gextrem} for any $u,y \in {\Sp}$, $t,s,z \in \R$
implies that for some $\tilde{f}: {\Sp} \to (0,\infty)$, $a \in \Rp$, $b
\in \R$, $$ g(u,t)= \tilde{f}(u) a e^{bt}.$$ 
Considering properties \eqref{2} and \eqref{3} it follows that $b \in
[\alpha-1, \alpha]$, i.e. $b = \chi - \alpha$ for some $\chi \in [0,1]$.
 Rewriting $a \tilde{f}(u)=: H^\alpha(u)^{-1}f(u)$, 
it follows that \begin{equation}\label{eqn:formula_g} g(u,t) =
\frac{f(u)}{H^\alpha(u)} e^{(\chi-\alpha)t} = \frac{f(u)e^{\chi
t}}{H^\alpha(e^tu)} .\end{equation} It remains to compute the possible values of
$f$ and $\chi$. Therefore, use property (5) which gives together with the comparison formula \eqref{eq:manytoone10}
\begin{align*}
f(u) =\ & e^{-\chi t}\, H^\alpha(e^tu) \, m(\alpha) \E_u^\alpha \left(
\frac{f(U_1)}{H^\alpha(e^{t-S_1}U_1)}e^{\chi (t-S_1)} \right) \\ 
=\ &  H^\alpha(e^tu)\, m(\alpha)\,
\frac{1}{H^\alpha(e^tu)k(\alpha)} \, \E_{u}^0\, \left({H^\alpha(e^{t-S_1}U_1)
\frac{f(U_1)}{H^\alpha(e^{t-S_1}U_1)} e^{-\chi S_1} }\right) \\
=\ & (\E N)\, \E_{u}^0\, {f(U_1)e^{-\chi S_1}} = (\E N)\, \Erw{f(\mM^\top \as
u)\abs{\mM^\top u}^\chi} \\ =\ & (\E N) P_*^\chi f(u).
\end{align*}
This means that $f$ is an eigenfunction of $P_*^\chi$ with eigenvalue
$\frac1{\E N}$. Referring to the definition of $\mathcal{H}_{\alpha,c}^K$, $f >0$. By 
Proposition \eqref{prop:Ps}, scalar multiples of $H^\chi$ are the only strictly
positive eigenfunctions of $\Pst[\chi]$. Thus $f =c H^\chi$ where $c$ is
given by property \eqref{1}. The eigenvalue of $\Pst[\chi]$ corresponding
to $H^\chi$ is $k(\chi)$. If now $k(\chi)=\frac1{\E N}$, then
$m(\chi)=(\E N) k(\chi)=1$, which shows that all extremal points of
$\mathcal{H}_{\alpha,c}^K$ are in $\mathcal{E}_{\alpha,c}^K$.
\end{proof}


\subsection{Proof of Theorem \ref{thm:uniqueness_final}}

Now we can give the proof of Theorem \ref{thm:uniqueness_final}. For the readers convenience, we repeat its statement.
 
 \begin{thm*}
Assume \eqref{A1}--\eqref{A4} and \eqref{A7}. Let $\alpha \in (0,1]$, not necessarily satisfying \eqref{A5}.
\begin{enumerate}
  \item If there is an $L$-$\alpha$-regular FP of $\STh$, then  $m(\alpha)=1$, $m'(\alpha) \le 0$. \label{a}
 \item  If $\LTa$ is an\label{b}
 $L$-$\alpha$-regular FP of $\ST$, then for all fixed $s >0$, \begin{equation}\tag{\ref{eq:regvarLT0}} \lim_{r \to 0} \, \sup_{u,v \in \Sp}
 \abs{\frac{1-\LTa(sru)}{1- \LTa(rv)}- s^\alpha \frac{H^\alpha(u)}{H^\alpha(v)}} =0.\end{equation}
  \item If $\LTfp$ is a\label{c}
$L$-$\alpha$-elementary FP  of $\ST$, then there is $K >0$ such that
\begin{equation}\tag{\ref{eq:limLTnull}} \lim_{r \to 0} \sup_{u \in \Sp}\abs{\frac{1- \LTfp(ru)}{L(r) r^\alpha} - K H^\alpha(u)} =0  . \end{equation}
\end{enumerate}
\end{thm*}


\begin{proof}
\eqref{a}: 
Considering Lemma \ref{lem:extremal_points}, there are at most two values
$\chi_1, \chi_2 \in (0,1]$, {\red $\chi_1 < \chi_2$},  $m(\chi_1)=m(\chi_2)=1$, such that every function
in $\mathcal{H}_{\alpha,c}^K$ can be written as a convex combination 
\begin{equation}\label{convexcombinations} (u,t) \mapsto
\frac{c}{H^\alpha(u)} \left( \lambda \est[\chi_1](u) e^{(\chi_1-\alpha) t} +
(1-\lambda) \est[\chi_2](u) e^{(\chi_2 - \alpha) t} \right)
\end{equation}
for $\lambda \in [0,1]$. Observe that unless $\alpha \in \{\chi_1, \chi_2\}$,
none of this convex combinations is a bounded function in
$t$ for fixed $u$.

Let $\LTb$ be an $L$-$\alpha$-regular FP. Recall the notation $\eins:=\sqrt{d}^{-1}
(1, \cdots, 1)^\top \in \Sp$. By Lemma \ref{lem:regularfixedpointslimit}, there is a subsequence
$s_n \to -\infty$, such that 
$$ \D[\infty](\eins,t) = \lim_{n \to \infty} \D[L,{s_n}](\eins,t) =
\lim_{n \to \infty} \frac{1-
\LTb(e^{s_n+t}\eins)}{H^\alpha( \eins)L(e^{s_n})e^{\alpha {s_n+t}}}.$$
Considering the definition of $L$-$\alpha$-regularity, the function
$\D[\infty](\eins,\cdot)$ is bounded from below and above by $\underline K$ resp.
$\overline K$. Hence by the above, $\alpha \in \{\chi_1, \chi_2\}.$

%

 Supposing that $\alpha=\chi_2$, the upper bound
 $$ \limsup_{r \to 0} \frac{1- \LTb(r\deins)}{L(r)r^\alpha} \le \overline K < \infty$$
 still implies, using the Tauberian theorem for LTs
\cite[XIII.5]{Feller1971}, that  if $\fprv$ is a random variable with LT $\LTb$,
then for any $\epsilon >0$ there is $C$ such that $$\P{\skalar{\fprv,\deins} > r}
\leq C r^{-\chi_2+\epsilon}.$$ Thus there is $\chi_1 < s < \chi_2 \le 1$ with
$m(s)<1$ and $\E \abs{\fprv}^s \le \E \skalar{\fprv,\deins}^s < \infty .$ But it can be deduced from \cite[Lemma 3.3]{NR2004} (see \cite[Section
4]{Mentemeier2013} for details), that the unique FP of $\STh$ with
finite $s$-moment for $m(s)<1$ and $s \le 1$ is $\delta_0$. 
Hence, $\alpha = \chi_1$. This proves $m(\alpha)=1$, $m'(\alpha) \le 0$.

\medskip

\eqref{b}: Moreover, the formula \eqref{convexcombinations} for functions in
$\mathcal{H}_{\alpha,c}^K$, in particular for the limit $\D[\infty]$, simplifies to 
$$ \D[\infty](u,t) =  c \left( \lambda + (1-\lambda)
\frac{\est[\chi_2](u)}{\est[\alpha](u)} e^{(\chi_2 - \alpha) t} \right).$$

Reasoning as before, the only possible choice is $\lambda=1$, since
otherwise $\D[\infty]$ would be unbounded. This proves that any subsequential limit of
$\D[L,s]$ is a positive constant function, nevertheless, the value of the
constant may depend on the subsequence. But this suffices to prove regular variation,
since for any subsequence $t_n$ such that $\D[L,{t_n}]$ converges,
$$ \lim_{n \to \infty} \frac{1-\LTb(e^{s+ t_n} u)}{1- \LTb(e^{t_n}v)} = \lim_{n
\to \infty} \frac{\D[L,{t_n}](s,u)}{\D[L,{t_n}](0,v)}
\frac{H^\alpha(e^s u)}{H^\alpha(v)} = \frac{H^\alpha(e^s u)}{H^\alpha(v)}, $$
i.e. the limit is independent of the particular subsequence. Since every
subsequential limit is the same, the asserted limit for $t \to 0$ exists. 

The convergence is uniform, since $\D[L,{t_n}] \to \D[\infty]$ uniform on the compact set
$\Sp \times [0,s]$ by Lemma \ref{lem:regularfixedpointslimit}.

%

\medskip

\eqref{c}:
If now $\LTfp$ is $L$-$\alpha$-elementary, then 
$$\lim_{s \to -\infty} \D[L,s](\eins,t) ~=~ \lim_{s \to - \infty} \frac{1- \LTfp(e^{s+t}\eins)}{L(e^s) e^{s+t}} ~=~ K, $$ hence for any subsequential limit $\D[\infty]$, it holds that $t \mapsto \D[\infty](\eins,t) \equiv
K$, thus $\lambda =1$ and consequently$ \D[\infty] \equiv K$ on $\Sp \times \R$. 
This gives that any subsequence $\D[L,{s_n}]$ with $s_n \to - \infty$ has the
same limit $K$, hence the compact uniform convergence $\D[L,s] \to K$. 
In particular,
$$ \lim_{s \to - \infty} \frac{1- \LTfp(e^{s+0}u)}{H^\alpha(u) L(e^s)e^{\alpha
s}} =K $$
uniformly on the compact set $\Sp \times \{0\}$. Replacing $t=e^s$, this gives
the assertion.
\end{proof}

\section{Uniqueness of Fixed Points of the Inhomogeneous Equation}\label{sect:inhom}

In this section, we finish the proof of the Theorem \ref{thm:main thm} concerning the inhomogeneous equation. We are going to prove the following result.

\begin{thm}\label{thm:uniquenessIFP}
Assume \eqref{A0}--\eqref{A8} and $\alpha <1$ with $m'(\alpha) <0$. Then a r.v. $X$ is a fixed point of \eqref{eq:SFPE} if and only if 
its LT is of the form
$$ \E \exp(-r \skalar{u,X}) ~=~ \LTfp_{Q,K}(ru) :=\E \exp(-Kr^\alpha W(u) - r\skalar{u,W^*}), \qquad u \in \Sp,\, r \in \Rnn$$
for some $K \ge 0$.  
\end{thm}

Recall that existence of fixed points for the inhomogeneous equation, i.e. the "if"-part in the Theorem above, has been proved in Theorem \ref{thm:existence hom}. Moreover, under the assumptions of Theorem \ref{thm:uniquenessIFP}, it follows from Theorems \ref{thm:existence hom}, \ref{thm:allfixedpointsregular} and \ref{thm:regularfp} that a r.v. $X$ is a fixed point of \eqref{eq:SFPE} with $Q \equiv 0$ if and only if 
its LT is of the form
$$ \E \exp(-r \skalar{u,X}) ~=~ \LTfp_{0,K}(ru) :=\E \exp(-Kr^\alpha W(u)), \qquad u \in \Sp,\, r \in \Rnn$$
for some $K \ge 0$. Note that $\LTfp_{0,0} \equiv 1$ is the trivial fixed point.

Using the characterization of fixed points of $\STh$, uniqueness of fixed points for $\STi$ will follow from a one-to-one correspondence given below. For probability laws $\rho$ and $\eta$ on $\Rdnn$, define
$$ l_s(\rho, \eta) ~:=~ \inf \{ \E \abs{X-Y}^s \ : \ \law{X}=\rho, \ \law{Y}=\eta
\}.$$ {\red On the subspace of probability laws with a finite $s$-th moment, this quantity is always finite and defines the so-called {\em minimal $L_s$-metric}, which is a particular case of a
Wasserstein distance. But it can also be used to measure distances between random variables with an infinite moment of order $s$, then finiteness of $l_s(\rho, \eta)$ implies that $\rho$ and $\eta$ have similar tail behavior, as will be seen in the proof of Theorem \ref{thm:uniquenessIFP} below.}

\begin{prop}\label{prop:bijection_fpi_fph}
Let $s \in (0,1]$ and $\E \norm{\mM}^s + \abs{Q}^s < \infty$. Suppose that
$m(s) <1$. Then the following holds:
\begin{enumerate}
  \item For any $\hfpd \in \Pset(\Rdnn)$ such that $\STh \hfpd = \hfpd$, there
  exists exactly one $\ifpd \in \Pset(\Rdnn)$ such that
  $$ \STi \ifpd = \ifpd \quad \text{ and } \quad l_s(\hfpd, \ifpd) < \infty. $$
  \item For any $\ifpd \in \Pset(\Rdnn)$ such that $\STi \ifpd = \ifpd$, there
  exists exactly one $\hfpd \in \Pset(\Rdnn)$ such that
  $$ \STh \hfpd = \hfpd \quad \text{ and } \quad l_s(\hfpd, \ifpd) < \infty. $$
\end{enumerate}
\end{prop}

\begin{proof}[Source:] The result can easily be  obtained from
\cite[Theorem 3.1]{Ruschendorf2006}, where the one-dimensional situation is
covered.
\end{proof}

\renewcommand{\hfprv}{X_{0}}
\renewcommand{\ifprv}{X_Q}

\begin{proof}[Proof of Theorem \ref{thm:uniquenessIFP}]
Assumptions \eqref{A5}--\eqref{A8} together with $\alpha <1$ and $m'(\alpha)<0$ imply the assumptions of Proposition \ref{prop:bijection_fpi_fph}, for $s=\alpha +\epsilon$ for some $\epsilon >0$. 
Writing $\Fset[Q]$ and $\Fset[0]$ for the set of fixed points of $\STi$ resp. $\STh$, we know that $\Fset[0] =\{ \eta_{0,K} \, : \, K \ge 0\}$, where $\eta_{0,K}$ is the probability measure with LT $\LTfp_{0,K}$. By Proposition \ref{prop:bijection_fpi_fph}, the induced mapping $\mathfrak{P} : \Fset[Q] \to \Fset[0], \mathfrak{P}\eta_Q = \eta_0$ is bijective, thus it suffices to show that $\mathfrak{P}\left( \{ \eta_{Q,K} \, : \, K \ge 0 \}\right) = \Fset[0]$.

Therefore, let us study further the property $l_s(\eta_Q, \mathfrak{P}\eta_Q) < \infty$.  Let $\eta_Q \in \Fset[Q]$ be arbitrary and 
Let $(\ifprv, \hfprv)$ be a coupling of $\eta_Q$ and $\mathfrak{P}\eta_Q$ with $\E \abs{\hfprv - \ifprv}^s < \infty$. 
Using the inequality
$\abs{a^s-b^s} \le \abs{a-b}^s$ which is valid for $s \in [0,1]$ and $a,b \in \Rnn$, 
it follows that for all $u \in \Rdnn$ 
$$ \E \abs{\skalar{u,\ifprv}^s - \skalar{u,\hfprv}^s } \le \E \abs{\skalar{u,
\ifprv-\hfprv}}^s \le \abs{u}^s \E \abs{\ifprv-\hfprv}^s < \infty .$$ Referring
to \cite[Lemma 9.4]{Goldie1991},  $$ \int_0^\infty \frac{1}{r} \ \Bigl(
r^{s} \abs{\P{\skalar{u,\ifprv}>r}-\P{\skalar{u,\hfprv}>r}} \Bigr) dr < \infty.$$ 
From the fact that $\int_1^\infty \frac1r dr $ diverges, 
it follows that necessarily $$ \limsup_{r \to \infty} r^{s} \abs{
\P{\skalar{u,\ifprv}>r} - \P{\skalar{u,\hfprv}>r}} = 0 .$$
Since $s > \alpha$, in particular
\begin{equation}\label{eq:sametails} \lim_{r \to \infty}  \abs{
r^{\alpha} \P{\skalar{u,\ifprv}>r} - r^{\alpha} \P{\skalar{u,\hfprv}>r}} = 0 . \end{equation}

Consider now $\ifprv \eqdist \eta_{Q,K}$ for some $K \ge 0$. By  Theorem \ref{thm:propertiesfp}, Eq. \ref{eq:kestenregvar}, 
$$ \lim_{r \to \infty} r^{\alpha} \P{\skalar{u,\ifprv}>r} ~=~ \frac{K}{\Gamma(1-\alpha)} H^\alpha(u), $$
and the only fixed point of $\STh$ with the same tail behavior is $\eta_{0,K}$ with the same $K$. 
Hence, Eq. \eqref{eq:sametails} implies that $\mathfrak{P}\eta_{Q,K}=\eta_{0,K}$ for all $K \ge 0$, which shows that $\mathfrak{P}\{\eta_{Q,K} \, : \, K \ge 0\} = \Fset[0]$. Since $\mathfrak{P}$ is bijective, we conclude that $\Fset[Q]=\{ \eta_{Q,K} \, : \, K \ge 0\}.$
\end{proof}

\section{Critical Case}
\label{sect:crit}

In this section, we prove the final part of Theorem \ref{thm:main thm} and sow that in the situation $m(\alpha)=1$ with
$m'(\alpha)=0$ for $\alpha \in (0,1]$, there still exists a nontrivial fixed
point of $\STh$, thus extending the results of \cite{BDGM2014} to the situation $m(1)=1$,
$m'(1)=0$. The existence of a fixed point in the boundary case is proven by the
same approximation argument as in \cite[Theorem 3.5]{DL1983}. This is why we just sketch the main ideas and refer the interested reader to \cite[Section 10]{Mentemeier2013} for details.

For $\chi \in (0, \alpha)$, define a biased version of $\STh$ by
\begin{equation}
\STh^\chi \ : \ \nu \mapsto \law{\sum_{i=1}^N m(\chi)^{-1/\chi} \mT_i X_i},
\end{equation}
where $X_i$ are iid with law $\nu$ and independent of $\T$. Writing
$$\T_{\chi}=(\mT_{\chi,i})_{i\ge1} =
\left(m(\chi)^{-1/\chi}\mT_{i}\right)_{i \ge 1}, $$ define $\mu_\chi$ and $m_\chi$ in terms of $\T_\chi$ as $\mu$ and $m$ were defined in terms of $\T$.

Then it is readily checked that \eqref{A0}--\eqref{A4} carry over (see \cite[Lemma
10.3]{Mentemeier2013}), that $m_\chi(\chi)=1$ and $m'_{\chi}(\chi) <0$ and that
\eqref{A6a} for $\T$ imply the validity of
\eqref{A6} for $\T_\chi$ and with $\alpha$ replaced by $\chi$. 

Hence Theorem \ref{thm:existence hom} applied to $\STh^\chi$ gives the existence of a
nontrivial fixed point with LT $\LTfp_\chi$, say. Fix $u_0 \in \Sp$. Then,
possibly after rescaling, $\LTfp_\chi(u_0)=1/2$. 
In this manner, construct a family $(\LTfp_\chi)_{\chi \in (0, \alpha)}$, such
that $\LTfp_\chi(u_0)=1/2$ and $\STh^\chi \LTfp_\chi = \LTfp_\chi$ for all $\chi
\in (0,\alpha)$. 

For any sequence $\chi_n \to \alpha$, there is a convergent subsequence
$\LTfp_{\chi_{n_k}}$ with limit $\LTfp$, which is again a LT of a
(sub-)probability measure, with $\LTfp(u_0)=1/2$. 
It can be checked that $\STh \LTfp (ru) = \LTfp (ru)$ for all $(u,r) \in \Sp
\times \Rp$ (see \cite[Lemma 10.4]{Mentemeier2013} for details). This is used to
infer that $\LTfp(0^+)=1$, hence $\LTfp$ is the LT of a probability measure on
$\Rdnn$, and it is nontrivial due to $\LTfp(u_0)=1/2$.

In the particular case $\alpha =1$, it is shown in \cite[Theorem 2.3]{BDGM2014} (under some restrictions on $N$)
that the existence of a nontrivial FP with finite expectation is
equivalent to $m'(1)<0$. Thus if $m'(1)=0$, then the nontrivial FP
constructed above necessarily has infinite expectation. It is a.s. finite since
$\LTfp(0^+)=1$.

Further properties of fixed points in the boundary case will be studied in \cite{KM2014}.

\section{Proofs of the Results from Subsections \ref{subsect:branchingmodel} to \ref{subsect:stopping}}\label{sect:proofs1}

\begin{proof}[Proof of Lemma {\ref{Lemma Rn}}]\label{proof_lemma_Rn}
$R_n \ge 0$ for all $n \in \N$, thus it
suffices to show that $\limsup_{n \to \infty} R_n = 0$. Writing
$R_{m,l} = \max_{\abs{w}=ml} \norm{\matrix{L}(w)}$, it follows
that $$ \limsup_{n \to \infty} R_n \le \sum_{k=0}^{l-1} \sum_{\abs{v}=k}
\norm{\matrix{L}(v)} \limsup_{m \to \infty} \tshift{R_{m,l}}{v}.$$ 
Thus it is enough to consider $\limsup_{m \to \infty} \tshift{R_{m,l}}{v}$ 
some $l \in \N$ and $\abs{v}\le l$. By the assumption $\alpha \in
\interior{I_\mu}$, there is $s > \alpha \in I_{\mu}$, such that $m(s)<1$. Referring to the definition of $m(s)$,
there is $l \in \N$ such that $$ \rho(s) := \E \sum_{\abs{v}=l} \norm{\matrix{L}(v)}^s = (\E N)^l \E
\norm{\mPi_l}^s < 1 .$$
Fix this $l$.
Define $Z_0=1$ and 
$$ Z_m = \sum_{\abs{v}=l} \norm{\matrix{L}(v)}^s \tshift{Z_{m-1}}{v} =
\sum_{\abs{v}=ml} \prod_{k=1}^m \norm{ \tshift{\matrix{L}(v|kl)}{v|(k-1)l} }^s$$
as the sum over the norms of the weights, taken in blocks of $l$
generations. Hence $\E Z_1 = \rho(s)$ and $(\tshift{R_{m,l}}{v})^s \le
\tshift{Z_m}{v}$ for all $m \in \N$, $v \in \tree$.

Considering the filtration $\F_m := \B_{ml}$ and using the
independence of $\tshift{\Ttree}{v}$ and $\B_{\abs{v}}$, it can easily be seen
that $\tilde{Z}_m := \rho(s)^{-m} Z_m$ is a nonnegative $\F_m$-martingale.
Thus it converges to a random
variable $\tilde{Z}$ and by Fatou's lemma, $\E\ \tilde{Z} \le \E\ \rho(s)Z_1 = 1 $. In
particular, $\tilde{Z}$ is almost sure finite, and this gives the final estimate
\begin{equation*} \limsup_{m \to \infty} (\tshift{R_{m,l}}{v})^s \le \limsup_{m
\to \infty} \rho(s)^m \tshift{\tilde{Z}_m}{v} = 0 \quad \Pfs \end{equation*}
for all $\abs{v}\le l$.
\end{proof}

\subsection*{Proof of Proposition \ref{thm:mean_convergence_Wn}}\label{sect:Wn}

The following result is the main tool to prove
the mean convergence of $W_n(u)$.

\begin{prop}\label{prop:Biggins}
Let $u \in \Sp$. For $r >0$, let 
$$ A(r) = \sum_{n=0}^\infty \1(H^\alpha(e^{-S_n}U_n) > r^{-1}).$$
Suppose that there is a random variable $Z$ such that 
\begin{equation}\label{Zbound} \P{\frac{\sum_{i=1}^N H^\alpha(\mT_i^\top
x)}{H^\alpha(x)} > s} \le \P{Z>s} \qquad \forall x \in \Rdnn \setminus \{0\}, s
\ge 0
\end{equation}
(stochastic domination) and a
function $L$, slowly varying at infinity, such that 
\begin{equation} \label{eq:L}
\sup_{r >0}
\frac{A(r)}{L(r)} < \infty \qquad \Prob_u^\alpha\text{-a.s.}
\end{equation} 

If $\E Z L(Z) <\infty$, then $\E W(u)= W_0(u)$. 
\end{prop}

\begin{proof}[Source:]
\cite[Theorem 1.1 (i)]{Biggins2004}, adopted to the present notation.
\end{proof}

\begin{proof}[Proof of Proposition \ref{thm:mean_convergence_Wn}]
We show that under the assumptions of Proposition \ref{thm:mean_convergence_Wn},
Proposition \ref{prop:Biggins} applies for any $u \in \Sp$. 
First, we prove that the {\red slowly varying function $L(r)=1+\log(1+ r)$ satisfies
\eqref{eq:L}. Since $\sup_{u \in \Sp}
H^\alpha(u) \le 1$,
\begin{align*}
A(r) \le  \sum_{n=0}^\infty \1( e^{-\alpha S_n} > r^{-1}) 
\le \sum_{n=0}^\infty \1(\alpha S_n <  \log r) \le \tau(\log(1+ r)), 
\end{align*}
where $$\tau(s) := \sup\{n \in \No : \alpha S_n \le s\}.$$ 
The assumptions of the  strong law of large numbers for $S_n$ under $\Prob_u^\alpha$, Proposition
\ref{prop:SLLN}, are exactly the assumptions imposed here, hence $\lim_{n \to \infty} \frac{S_n}{n} = m'(\alpha) <0$ $\Prob_u^\alpha$-a.s. Now on the one
hand, $\sup_{r \in (0,1)} A(r)/L(r)$ is}
readily bounded by $\tau(0)$, which is finite since $S_n$ is transient. On the
other hand, it is as well a consequence of the strong law of large numbers  that 
 $$ \lim_{s \to \infty}
\frac{\tau(s)}{s}= \frac{1}{(-m'(\alpha))} \qquad
\Prob_u^\alpha\text{-a.s}$$ (see the argument in \cite[bottom of
p.219]{Breiman1968}) and consequently, $\sup_{r >1} A(r)/L(r)$
is bounded $\Prob_u^\alpha$-a.s., too. 

As the second step, observe that, upon defining $$ Z := C \sum_{i=1}^N
\norm{\mT_i^\top}^\alpha$$ with $C:= \sup_{u \in \Sp}
H^\alpha(u)^{-1} < \infty$, \eqref{Zbound} is satisfied. The finiteness of
$E Z L(Z)$ is then a direct consequence of assumption \eqref{A6}.

\medskip

Having thus proven that $W_n(u)$ converges $\Pfs$ to a nontrivial limit $W(u)$ for all $u \in \Sp$, let us discuss whether $u \mapsto W(u)$ is measurable. We can write
$W_n(u):=w_n(u, (\mL(v)_{v \in \tree}))$ as a measurable function of $u$ and the branch weights. Fix a countable dense subset $\mathbb{S}:= \{u_k : k \in \N\} \subset \Sp$. Then there is a measurable exceptional set $E \subset \Omega$ with $\P{E}=0$, such that for all $\omega \in E^c$
$ \lim_{n \to \infty} w_n(u, (\mL(v)_{v \in \tree})(\omega))$ exists for all $u \in \S$. Using a sandwich argument, the limit exists for all $u \in \Sp$ and $\omega \in E^c$. Define $w$ as this limit on $E^c$, and let $w \equiv 1$ on $E$. Then $w$ is a measurable function on $\Sp \times \Mset^\tree$ and $W(u)=w(u, (\mL(v))_{v \in \tree})$ $\Pfs$.
\end{proof}

\subsection*{Proof of Lemma \ref{prop:martingale_stopping_lines}}
%
%

\begin{proof}[Proof of Lemma \ref{prop:martingale_stopping_lines}]
The first part of the proof is valid for any anticipating and $\Pfs[u]$ dissecting HSL $\sline$. Following the lines of the proof of \cite[Lemma
6.1]{Biggins1997}, write $$E_{\sline}(j) = \{ v \in \tree \ : \ \abs{v}=j, v \in \sline \} $$
and $$ A_{\sline}(j) = \{ v \in \tree \ : \ \abs{v}= j, v \text{ has no
ancestor in } \sline \}.$$ Consequently, if $v \in A_{\sline}(j)$, then
$\sigma((\T(v|k))_{k=0}^{\abs{v}}) \subset \B_{\sline}$ as well as
$(\T(vw))_{w \in \tree}$ and $B_{\sline}$ are independent for $v \in E_{\sline}(j)$. Then for
$m \in \N$, 
\begin{align*}& \E[ W_m(u) | \B_{\sline}]\\
 = & \E \left[\left. \sum_{j=1}^m \sum_{v \in
E_{\sline}(j)}
\sum_{\abs{w}=m-j} H^\alpha(\mL(vw)^\top u) + \sum_{v \in A_{\sline}(m)}
H^\alpha(\mL(v)^\top u) \right| \B_{\sline} \right] \\
= &  \sum_{j=1}^m \sum_{v \in E_{\sline}(j)} \E \left[ \left. 
\sum_{\abs{w}=m-j} H^\alpha(\tshift{\mL(w)}{v}^\top \mL(v)^\top u) \right|
\B_{\sline} \right] + \sum_{v \in A_{\sline}(m)}
H^\alpha(\mL(v)^\top u) \\
= &  \sum_{j=1}^m \sum_{v \in E_{\sline}(j)}  H^\alpha(\mL(v)^\top u) + \sum_{v \in A_{\sline}(m)}
H^\alpha(\mL(v)^\top u) \qquad \Pfs
\end{align*}
Since $\sline$ is $\Pfs$ dissecting, in the limit $m \to \infty$,
$$ \E [ W(u)| \B_{\sline} ] = \sum_{v \in \sline} H^\alpha(\mL(v)^\top u) =
W_{\sline}(u) \qquad \Pfs$$
this convergence being valid as well in $L^1(\Prob)$, for the right hand side
can be bounded by the uniform integrable sequence $2 W_m(u)$ in every step. 

Now for the sequence of stopping lines $\slineu[t]$, we have that $W(u)$ is
measurable w.r.t to $\B_\infty$. Then by \cite[Theorem 5.21]{Breiman1968} 
$$ \lim_{t \to \infty} W_{\slineu[t]}(u) = \lim_{t \to \infty} \E[W(u) | \B_{\slineu[t]}] =
\E[W(u) | \B_\infty] = W(u)$$ $\Pfs$ and in $L^1(\Prob)$.
\end{proof} 

\section{Proof of Theorem \ref{thm:l1_conv_wf}}\label{sect:proofs2}

In this section, we prove Theorem \ref{thm:l1_conv_wf}, which states that
$$ \lim_{t \to \infty} W_{\slineu[t]}^f ~=~ \lim_{t \to \infty}\sum_{v \in
\slineu[t]} H^\alpha (\mL(v)^\top u) f(U^u(v),S^u(v)-t) ~=~ \gamma W(u) $$
in $\Prob$-probability and in $L^1(\Prob)$, where $\gamma = \int f(y,s) \rho(dy,ds)$ is the limit of $\E_u^\alpha f(U(t), R(t))$ for $t \to \infty$, and $f$ is any bounded continuous function on $\Sp \times \R$, or a slight generalization thereof. 

As a first step, we need the following stronger version of Theorem 
\ref{thm:kestenforW}. Write 
$$ F(u,t):=\frac{\E W_{\slineu[t]}^f}{H^\alpha(u)}.$$

\begin{prop}\label{prop:uniform_conv_mrt}
Under the assumptions of Theorem \ref{thm:kestenforW}, it holds that
$$ \lim_{t \to \infty}  \sup_{u \in \Sp} \abs{F(u,t)- \gamma}=0.$$
\end{prop}

\begin{proof}
Recall that due to the many-to-one identity \eqref{eq:spinaltreejump}, $F(u,t) = \E_u^\alpha f(U(t),R(t))$, and that, subject to the assumptions of Theorem \ref{thm:kestenforW}, Kesten's renewal theorem applies in order to show that $\lim_{t \to \infty} \E_u^\alpha f(U(t),R(t)) = \gamma$. \cite{Melfi1992} proved (ibid., Theorem 2) that under the same assumptions, the convergence is uniform in $u \in \Sp$. This implies the uniform convergence asserted above.
\end{proof}



Subsequently, the assumptions of Theorem
\ref{thm:l1_conv_wf} will be in force throughout.

\begin{lem}\label{lem:ui}We have the uniform integrability
$$ \lim_{q \to \infty} \sup_{u \in \Sp} \sup_{t \in \Rp} \, \Erw{W_{\slineu[t]}^f(u)
\1[{\{W_{\slineu[t]}^f(u) > q \}}]} = \lim_{q \to \infty} \sup_{u \in \Sp} \sup_{t \in \Rp}
\Erw{W_{\slineu[t](u)} \1[{\{W_{\slineu[t]}(u) > q \}}]} =0.$$
\end{lem}

\begin{proof}
Since $0 \le W_{\slineu[t]}^f (u)\le \abs{f}_\infty W_{\slineu[t]}(u)$  and $f$ is assumed to be bounded, it
satisfies to prove the second equality.

Let $u_1, \dots, u_d \in \Sp$ be the standard basis of $\Rd$, then for all $u
\in \Sp$, $n \in \No$ it follows right from the definition
\eqref{def:martingale} of $W_n(u)$ that
$$ W_n(u) \le \sum_{i=1}^d W_n(u_1),$$
and consequently, due to the $\Pfs$-convergence,
$ W(u) \le \sum_{i=1}^d W(u_i).$
Now for fixed $q$, the function $g(t):= t \1[(q,\infty)](t)$ is convex.
Using the estimate above, the conditional Jensen inequality and Lemma
\ref{prop:martingale_stopping_lines}, we compute that for all $u \in \Sp$, $t
\in \Rnn$
\begin{align*}
\E\, g(W_{\slineu[t]}(u)) = &~ \E\, g(\E[ W(u) | \B_{\slineu[t]}]) \le \E\, \E[ g(W(u)) |\B_{\slineu[t]}] \\
 = &~ \E\, g(W(u))  
 \le~  \E\, g \left( \sum_{i=1}^d W(u_i) \right). 
\end{align*}
Since $\sum_{i=1}^d W(u_i)$ is $\Prob$-integrable, the last expression tends to
zero for $q \to \infty$.
\end{proof}


\begin{lem}\label{lem:limitWtf}
Let $(t_n)_{n \in \N} \subset \Rp$ be a sequence with $\lim_{n \to \infty} t_n = \infty$. Then we have
$$ \lim_{n \to \infty} \sum_{v \in \slineu[t_n/2]} e^{-\alpha S^u(v)} \est[\alpha](U^u(v)) F(U^u(v),t_n-S^u(v)) ~=~ \gamma W(u) \qquad \Pfs.$$
\end{lem}

\begin{proof} This follows from the $\Pfs$ convergence of $W_{\slineu[t_n/2]}^f$, proved in Lemma \ref{prop:martingale_stopping_lines}, together with the uniform convergence of $F$ to $\gamma$, shown in Proposition \ref{prop:uniform_conv_mrt}.
\end{proof}

Set $$\xi_{t}(u):= \frac{W_{\slineu[t]}^f(u)}{\est[\alpha](u)F(u,t)}.$$
Then each $\xi_t(u)$ has mean one, and the assertion of uniform convergence in Proposition \ref{prop:uniform_conv_mrt}
together with Lemma \ref{lem:ui} above imply the following corollary:
\begin{cor}\label{cor:ui}
We have the uniform integrability
$$ \lim_{q \to \infty} \sup_{u \in \Sp} \sup_{t \in \Rp} \, \Erw{ \xi_t(u)
\1[{\{ \abs{\xi_t(u)} > q \}}] } = 0.$$
\end{cor}

\begin{lem}\label{lem:wtf_br}
If $r<t$, then for all $u \in \Sp$, it holds $\Pfs$ that
\begin{equation} \label{eq:wtf_br} W_{\slineu[t]}^f(u) ~=~ \sum_{v \in \slineu[r]} e^{-\alpha S^u(v)} \est[\alpha](U^u(v)) \,
 F(U^u(v),t_n-S^u(v)) \, {\tshift{\xi_{t-S^u(v)}(U^u(v)}{v}}.
\end{equation}
\end{lem}

Note that 
 that $W_s^f(u):=f(u,-s)$ for $s<0$.

\begin{proof}
If $r <t$, then $\slineu[r] \prec \slineu[t]$ in the sence that for every $x \in
\slineu[t]$ there is $v \in \slineu[r]$ with $v \prec x$, i.e. there is $w \in
\tree$ s.t. $x=vw$. Hence, we have the general decomposition
\begin{align*}
W_{\slineu[t]}^f =&~ \sum_{v \in \slineu[r]} \sum_{w \in \tree} H^\alpha(\mL(vw)^\top \as
u) f(U^u(vw), S^u(vw)-t) \1[\{S^u(vw) > t, S^u(vw|k) \le t\, \forall k <
\abs{vw} \}] \\ =&~ \sum_{v \in \slineu[r]} \sum_{w \in \tree}
 e^{-\alpha(S^u(v)+\tshift{S^{U(v)}(w)}{v})} \, \est[\alpha]\left(\tshift{\mL(w)^\top}{v}
 \as U(v)\right) \times \\
 & ~ \phantom{\sum_{v \in \sline[r]} \sum_{w \in \tree}}
 f\left(\tshift{\mL(w)^\top}{v} \as U^u(v), S^u(v) + \tshift{S^{U(v)}(w)}{v}-t \right)  \times\\
 &~ \phantom{\sum_{v \in \sline[r]} \sum_{w \in \tree}}
 \1[\{\tshift{S^{U(v)}(w)}{v}  > t - S^u(v),\ \tshift{S^{U(v)}(w|k)}{v} \le t - S^u(v)\ \forall k
 < \abs{w} \}] \\
 =&~ \sum_{v \in \slineu[r]} e^{-\alpha S^u(v)}
 {\tshift{W_{\sline[t-S^u(v)]^{U^u(v)}}^f(U^u(v))}{v}} \\
=&~ \sum_{v \in \slineu[r]} e^{-\alpha S^u(v)} \est[\alpha](U^u(v)) \, F(U^u(v),t-S^u(v)) \, 
 {\tshift{\xi_{t-S^u(v)}(U^u(v)}{v}} 
\end{align*}
\end{proof}

Now we can give the proof of Theorem \ref{thm:l1_conv_wf}.

\begin{proof}[Proof of Theorem \ref{thm:l1_conv_wf}]
Due to Lemma \ref{lem:wtf_br}, Eq. \eqref{eq:wtf_br}, for $\Prob$-a.e. $\omega \in \Omega$, $W_{\slineu[t_n/2]}^f(u)$ constitutes a triangular array with respect to the probabilities $\P{\cdot \, | \B_{\slineu[t_n/2]} }(\omega)$. By Corollary \ref{cor:ui} and Lemma \ref{lem:limitWtf} we can use
 \cite[Corollary 5]{Cohn1994} for the triangular array to infer the convergence 
$$ \P{ \abs{W_{\slineu[t_n]}^f(u) - \gamma W(u) } > \epsilon | \B_{\slineu[t_n/2]} }(\omega) = 0$$
for $\Prob$-a.e. $\omega$. Using dominated convergence, we infer the convergence $W_{\slineu[t_n]}^f(u) \to \gamma W(u)$ in $\Prob$-probability. Together with the uniform integrability of $W_{\slineu[t]}^f$, proved in Lemma \ref{lem:ui}, this yields $L^1(\Prob)$-convergence.
\end{proof}


\appendix

\section*{List of Symbols}

\input{symbols.tex} 

\section{Inequalities for Laplace Transforms}\label{subsect:ineq_laplace}
If $\LTa$ is the LT of a r.v. $\Rnn$, then $t^{-1}(1-\LTa(t))$
is again a LT of
a measure on $\Rnn$ (see \cite[XIII (2.7)]{Feller1971}). Consequently, it is
decreasing and thus for all $t \in \Rnn$, $0 < a < 1$:
\begin{equation*}
\frac{1-\LTa(at)}{at} \ge \frac{1-\LTa(t)}{t} \quad \Rightarrow 1-\LTa(at) \ge a (1-\LTa(t)),
\end{equation*}
as well as, for $b\ge 1$,
\begin{equation*}
1- \LTa(bt) \le b (1-\LTa(t)) .
\end{equation*}
This proves the first four inequalities in the subsequent lemma:

\begin{lemma}\label{lem:LaplaceInequalities2}
Let $\LTa$ be the Laplace transform of a distribution on $\Rdnn$, $u \in \Sp$,
$t \in \Rnn$ and $\matrix{A} \in M(d \times d, \Rnn)$. Then
\begin{align}
1 - \LTa(atu) & \le 1- \LTa(tu) & \text{ for } a < 1, \label{ineq1} \\
1-\LTa(atu) & \ge a(1-\LTa(tu)) & \text{ for } a < 1, \label{ineq2} \\
1- \LTa(btu) & \ge 1- \LTa(tu) & \text{ for } b >1, \label{ineq3} \\
1- \LTa(btu) & \le b(1-\LTa(tu)) & \text{ for } b > 1. \label{ineq4} \\
1- \LTa(tu) & \le 1- \LTa(t\deins) \label{ineq5}\\
1- \LTa(t\matrix{A}u) & \le 1- \LTa(t\abs{\matrix{A}u}\deins) \le 1-
\LTa(t\norm{\matrix{A}}\deins) \label{ineq6}\\ 
1 - \LTa(t\matrix{A}u) & \le 
\left(\norm{\matrix{A}}\vee 1 \right) (1-\LTa(t\deins)) \label{ineq7} \\ 1 -
\LTa(tu) & \ge 1 - \LTa(t(\min_i u_i) \deins) \ge (\min_i u_i) (1-\LTa(t\deins))
\label{ineq8}
\end{align}
\end{lemma}

\begin{proof}
Let $Z$ be a r.v. with LT $\LTa$. For all $u \in \Sp$, $\skalar{u,Z} \le \skalar{\deins,Z}$. Thus
\begin{align*}
1 - \LTa(tu) & = \Erw{1-e^{-t \skalar{u,Z}}} = \int_0^\infty t e^{-tr} \P{\skalar{u,Z} > r} dr \\
& \le \int_0^\infty t e^{-tr} \P{\skalar{\deins,Z} > r} dr = 1- \LTa(t\deins) . 
\end{align*}
From \eqref{ineq5} and \eqref{ineq1} now \eqref{ineq6} follows:
\begin{align*}
1- \LTa(t\matrix{A}u) & = 1-\LTa(t\abs{\matrix{A}u}\matrix{A} \as u) \stackrel{\eqref{ineq5}}{\le} 1-\LTa(t\abs{\matrix{A}u}\deins) \\ & = 1-\LTa(t\frac{\abs{\matrix{A}u}}{\norm{\matrix{A}}}\norm{\matrix{A}}\deins) \stackrel{\eqref{ineq1}}{\le} 1- \LTa(t\norm{\matrix{A}}\deins)  .
\end{align*}
Then \eqref{ineq7} follows by applying \eqref{ineq1} resp. \eqref{ineq4} in \eqref{ineq6}.

In order to prove \eqref{ineq8}, observe that 
$$ \skalar{u,Z} = \sum_{i=1}^d u_i Z_i \ge \min_i u_i \sum_{i=1}^d Z_i = \min_i u_i \skalar{\deins,Z} .$$
Then the argument is the same as given for \eqref{ineq5}, with an additional use
of \eqref{ineq2}.
\end{proof}

\end{document}

%% file: Makros_FPMST.tex
\newcommand{\Cov}{\mathbb{K}}

\newcommand{\E}{\mathbb{E}}
\newcommand{\Prob}{\mathbb{P}}

\newcommand{\Pkern}[1][]{\ensuremath{{}^{ #1}\!P}}

\newcommand{\Pfs}[1][]{\ensuremath{\mathbb{P}_{#1}\text{-a.s.}}}
\newcommand{\Erw}[2][]{\ensuremath{\mathbb{E}_{#1} \left( {#2} \right)}}

\newcommand{\N}{\mathbb{N}}

\newcommand{\Q}{\mathbb{Q}}
\newcommand{\R}{\mathbb{R}}
\newcommand{\B}{\mathfrak{B}}
\newcommand{\C}{\mathbb{C}}
\newcommand{\V}{\mathbb{V}}

\newcommand{\A}{\mathcal{A}}
\newcommand{\F}{\mathcal{F}}
\renewcommand{\S}{\mathcal{S}}

\newcommand{\llam}{{l}}

\newcommand{\supp}{\mathrm{supp}\,  }

\renewcommand{\epsilon}{\varepsilon}
\renewcommand{\rho}{\varrho}

\newcommand{\esl}[1]{\overline{ #1 }}

\newcommand{\1}[1][]{\mathbf{1}_{#1}}
\newcommand{\norm}[1]{\ensuremath{\left\| {#1} \right\|}}

\newcommand{\abs}[1]{\ensuremath{\left| {#1} \right|}}
\newcommand{\skalar}[1]{\langle #1 \rangle} 
\newcommand{\infnorm}[1]{\abs{#1}_\infty}

\newcommand{\eqdist}{\stackrel{\mathcal{L}}{=}}

\newcommand{\follows}{\Rightarrow}

\newcommand{\Step}[1][]{\textsc{Step #1}}

\newcommand{\ST}{\mathcal{S}}
\newcommand{\STh}{\mathcal{S}_0}
\newcommand{\STi}{\mathcal{S}_Q}
\newcommand{\Pset}[1][]{\mathcal{P}_{#1}}
\newcommand{\measureset}[2][]{\mathcal{M}^{#1}\left( #2 \right)}
\newcommand{\law}[1]{\mathcal{L}\left( #1 \right)}
\newcommand{\tree}{\mathfrak{T}}
\newcommand{\Ttree}{\mathcal{T}}

\newcommand{\Ytree}{\mathcal{Y}}

\newcommand{\Id}{\matrix{Id}}
\newcommand{\Rd}{\R^d}
\newcommand{\Rdnn}{{\R^d_\ge}}

\newcommand{\Rp}{\R_>}
\newcommand{\Rnn}{\R_\ge}

\newcommand{\Cf}[2][]{\mathcal{C}^{#1}\left( #2 \right)}
\newcommand{\Cbf}[2][]{\mathcal{C}_b^{#1}\left( #2 \right)}
\newcommand{\Ccf}[2][]{\mathcal{C}_c^{#1}\left( #2 \right)}
\newcommand{\Cof}[2][]{\mathcal{C}_0^{#1}\left( #2 \right)}

\newcommand{\Bbmf}[1]{\mathcal{B}^1\left( #1 \right)}
\newcommand{\Zolset}[1]{\mathcal{D}_{#1}}
\newcommand{\dto}{\stackrel{d}{\to}}
\newcommand{\vto}{\stackrel{v}{\to}}
\newcommand{\Fset}[1][]{\mathcal{F}_{#1}}
\newcommand{\Fseta}[1][\alpha]{\mathcal{F}^{#1}}
\newcommand{\Eigenspace}[1]{\mathrm{Eig}(#1)}
\newcommand{\Eigenspaceo}[1]{\mathrm{Eig}_0(#1)}
\newcommand{\Eigenspacenn}[1]{\mathrm{Eig}^+(#1)}

\newcommand{\U}[1][]{\mathbb{U}_{#1}}
\newcommand{\condC}{(C)}

\newcommand{\QQ}{\mathbb{Q}}
\newcommand{\Sp}{\mathbb{S}_\ge}
\newcommand{\Sd}{\mathbb{S}}
\newcommand{\Mset}{\mathcal{M}_\ge}
\newcommand{\interior}[1]{\breve{#1}}
\newcommand{\est}[1][s]{e_*^{#1}}
\newcommand{\Pst}[1][s]{P_*^{#1}}
\newcommand{\nust}[1][s]{\nu_*^{#1}}
\newcommand{\pist}[1][s]{\pi_*^{#1}}

\newcommand{\es}[1][s]{e^{#1}}
\newcommand{\Ps}[1][s]{P^{#1}}
\newcommand{\nus}[1][s]{\nu^{#1}}
\newcommand{\mM}{\mathbf{M}}
\newcommand{\mT}{\mathbf{T}}
\newcommand{\mPi}{\matrix{\Pi}}
\newcommand{\mL}{\matrix{L}}
\newcommand{\mA}{\matrix{A}}
\renewcommand{\matrix}[1]{\mathbf{#1}}

\newcommand{\PiA}{\matrix{\daleth}}
\newcommand{\Rfp}{R}

\newcommand{\fpd}{\eta}
\newcommand{\fprv}{Y}
\newcommand{\LTfp}{\psi}

\newcommand{\hfpd}{\eta_0}
\newcommand{\hfprv}{Y_0}

\newcommand{\ifpd}{\eta_Q}
\newcommand{\ifprv}{Y_Q}

\newcommand{\as}{\cdot}

\newcommand{\LTa}[1][]{\phi^{#1}}
\newcommand{\LTb}{\varphi}

\newcommand{\T}{T}
\newcommand{\K}{K}
\newcommand{\No}{\mathbb{N}_0}

\newcommand{\deins}{\mathbf{\vartheta_d}}
\newcommand{\eins}{\mathbf{\vartheta_1}}
\newcommand{\tshift}[2]{\left[ #1 \right]_{#2}}
\newcommand{\stablen}[2]{\tilde{S}_{#1}(#2)}
\newcommand{\lalpha}{\lambda^\alpha}
\newcommand{\ldd}{\lambda^{d^2}}
\renewcommand{\L}{\mathfrak{L}}

\newcommand{\minkern}{\Psi}
\newcommand{\minmeas}{\Phi}

\newcommand{\tv}[1]{\mathrm{tv}\left[ #1 \right]}

\renewcommand{\O}{\mathbb{O}}

\newcommand{\minkernXV}{\Upsilon}

\newcommand{\ma}{\matrix{a}}

\newcommand{\sline}[1][]{\mathcal{I}_{#1}}
\newcommand{\slineu}[1][]{\mathcal{I}_{#1}^u}

\newcommand{\D}[1][]{D_{#1}}

\renewcommand{\P}[2][]{\ensuremath{\mathbb{P}_{#1} \left( {#2} \right)}}

\newglossaryentry{leftderivative}{type=symbols,name={\ensuremath{m'(s^{-})}},
symbol={$m'(s^{-})$},sort=m, description={left derivative of $m$ in $s$}}

\newglossaryentry{boundary}{type=symbols,name={\ensuremath{\partial B}},
symbol={$\partial B$},sort=d, description={topological boundary of the set $B$}}

\newglossaryentry{rho}{type=symbols,name={\ensuremath{\rho(\cdot,\cdot)}},
symbol={$\rho(\cdot,\cdot)$},sort=rho, description={Prohorov metric}}

\newglossaryentry{K}{type=symbols,name={\ensuremath{K}},
symbol={$K$},sort=K, description={Chapter \ref{chapter:MST}: scaling constant.
Chapter \ref{chapter:RDE}:\newline $K=\frac{1}{l(\beta) \beta} \int_{\Sd}
\frac{1}{\est[\beta]({y})}
\Erw{(\skalar{y,R}^+)^{\kappa}-(\skalar{y,\mT R}^+)^{\kappa}}
\pist(dy)$ }}

\newglossaryentry{N(t)}{type=symbols,name={\ensuremath{N(t)}},
symbol={$N(t)$},sort=N(t), description={first exit time $N(t)= \inf\{ n \geq 0 : V_n > t\}$ }}

\newglossaryentry{R(t)}{type=symbols,name={\ensuremath{R(t)}},
symbol={$R(t)$},sort=R(t), description={residual lifetime process $R(t) = (V_{N(t)} - t)\1[\{
N(t) < \infty \}]$ }}

\newglossaryentry{Z(t)}{type=symbols,name={\ensuremath{Z(t)}},
symbol={$Z(t)$},sort=Z(t), description={jump process $Z(t)= X_{N(t)}\1[\{ N(t) < \infty
\}]$ }}

\newglossaryentry{taun}{type=symbols,name={\ensuremath{(\tau_n)_{n \in
\N}}}, symbol={$(\tau_n)_{n \in \N}$},sort=taun, description={hitting times of
$(X_n)_{n \ge 0}$ in $B_\delta(x_0)$ }}

\newglossaryentry{sigman}{type=symbols,name={\ensuremath{(\sigma_n)_{n \in
\No}}}, symbol={$(\sigma_n)_{n \in \No}$},sort=sigman, description={regeneration
epochs for $(X_n, U_n)_{n \in \No}$ under $\Q_x^\beta$ \gls{wrt} $\minkernXV$.
}}

\newglossaryentry{varpin}{type=symbols,name={\ensuremath{(\varpi_n)_{n \in
\No}}}, symbol={$(\varpi_n)_{n \in \No}$},sort=wn,
description={feasible sequence of regeneration epochs for $(X_n, \mM_n, Q_n)_{n
\in \No}$ under $\O_x$ 
\gls{wrt} $\Xi$, $X_{\varpi_{n}-1} \in B_\delta(X_0)$. }}

\newglossaryentry{Jn}{type=symbols,name={\ensuremath{(J_n)_{n \in \No}}},
symbol={$(J_n)_{n \in \No}$},sort=Jn, description={iid B(1,$\xi$) r.v.s,
determining whether regeneration occurs, see Lemma \ref{regenerationlemma} }}

\newglossaryentry{R(i)}{type=symbols,name={(R\ldots)},
symbol={(R\ldots)},sort=0r, description={properties of the regenerative
structure, see Lemma \ref{regenerationlemma} }}

\newglossaryentry{Xi}{type=symbols,name={\ensuremath{\Xi}},
symbol={$\Xi$},sort=Xi, description={minorizing kernel: $\hat{O}(x, \cdot) \ge
\Xi(x, \cdot)$. There are $L, \varsigma>
0$ such that\newline $ \Xi(x,B_\delta(x_0) \times B
\times \Rd) = L \int_{B_\varsigma(\Id)} \1[B_\delta(x_0) \times B](\mA \as x, \mA \mA_x)  \ldd(d\mA)  
$  }}

\newglossaryentry{Ohat_id}{type=symbols,name={\ensuremath{\hat{O}}},
symbol={$\hat{O}$},sort=Oh, description={$\hat{O}((x,\mA,q),A\times B \times C) = \P{
(\Pi_1 \as x, \Pi_1,Q_1) \in A \times B \times C}$, Markov
transition operator of $(X_n, \mM_n, Q_n)_{n \in \No}$ under $\O_x$}}

\newglossaryentry{Ox}{type=symbols,name={\ensuremath{\O_x}},
symbol={$\O_x$},sort=Ox, description={$\O_x = \delta(x) \otimes \delta(\Id) \otimes
 \delta(0) \otimes \bigotimes_{n=1}^\infty \varrho$ }}

\newglossaryentry{minkernXV}{type=symbols,name={\ensuremath{\minkernXV}},
symbol={$\minkernXV$},sort=Upsilon, description={minorizing kernel: $ \hat{Q}_s((x,u), \cdot) \ge \xi_s
\minkernXV(x, \cdot )$ There is compact $I \subset \R$ with $\supp \minkernXV(x,
\cdot) \subset B_\delta(x_0) \times I $ for all $x \in \Sd$ }}

\newglossaryentry{Qhat_id}{type=symbols,name={\ensuremath{\hat{Q}_s}},
symbol={$\hat{Q}_s$},sort=Qhs, description={$\hat{Q}_s((x,u),A\times B) $
\newline\phantom{Yak}$= \frac{1}{\est(x) \kappa(s)} \Erw {\est(\mM_1 \as x)
\abs{\mM_1 x}^s \1[A](\mM_1 \as x) \1[B](\log \abs{\mM_1 x})}$. Markov transition operator of $(X_n, U_n)_{n \in \No}$ under $\Q_x^s$}}

\newglossaryentry{SA}{type=symbols,name={\eqref{SA}},
symbol={\eqref{SA}},sort=0s, description={If \eqref{irred},\eqref{density},\eqref{eqn:MC1'} and \eqref{eqn:MC2'}
hold, they hold with $n_0=n=1$ }}

\newglossaryentry{Cfconj}{type=symbols,name={\ensuremath{\Cf{E}'}},
symbol={$\Cf{E}'$},sort=Cf', description={conjugate space of
$(\Cf{E},\infnorm{\cdot})$: regular bounded signed measures on $E$, equipped
with the total variation norm }}

\newglossaryentry{r(Q)}{type=symbols,name={\ensuremath{r(Q)}},
symbol={$r(Q)$},sort=r(Q), description={spectral radius of the operator Q }}

\newglossaryentry{MC1}{type=symbols,name={\eqref{eqn:MC1}},
symbol={\eqref{eqn:MC1}},sort=0mc1, description={$\Prob \left( (\Pi_1 \as y, \Pi_1) \in \cdot \right) \ge
 \xi \minkern(y, \cdot)$ }}
 
 \newglossaryentry{MC2}{type=symbols,name={\eqref{eqn:MC2}},
symbol={\eqref{eqn:MC2}},sort=0mc2, description={$\Prob \left( \Pi_1 \as y \in \cdot, \Pi_1 \in C \right) \ge \xi
 \minmeas$ }}

\newglossaryentry{Q^n}{type=symbols,name={\ensuremath{Q^n}},
symbol={$Q^n$},sort=Qn, description={$Q^n = \sum_{k=1}^n \PiA_{k-1} Q_k$ }}

\newglossaryentry{R^n}{type=symbols,name={\ensuremath{R^n}},
symbol={$R^n$},sort=Rn, description={$\Rfp^{n}=\sum_{k>n}
\left(\prod_{j=n+1}^{k-1}\mT_{(j)}\right)Q_k$ }}

\newglossaryentry{minkern}{type=symbols,name={\ensuremath{\minkern}},
symbol={$\minkern$},sort=Psi, description={minorizing kernel: $ \Prob \left(
(\Pi_1 \as y, \Pi_1) \in \cdot \right) \ge \xi \minkern(y, \cdot)$ (under
\eqref{SA}),\newline there is compact $C$ with $\supp \minkern(y, \cdot) \subset
B_\delta(x) \times C $ for all $y \in \Sd$ }}

\newglossaryentry{minmeas}{type=symbols,name={\ensuremath{\minmeas}},
symbol={$\minmeas$},sort=Phi, description={minorizing measure: $ \Prob
\left( \Pi_1 \as y \in \cdot, \Pi_1 \in C \right) \ge \xi \minmeas$ (under
\eqref{SA}), $\minkern(y, \cdot \times C) = \minmeas$ for all $y \in \Sd$ }}

\newglossaryentry{existsbeta}{type=symbols,name={\eqref{cond:existence_of_beta}},
symbol={\eqref{cond:existence_of_beta}},sort=0A, description={$\exists_{s_0 >0} \quad \E \inf_{u \in \Sd} \abs{\mT^\top x}^{s_0} \ge 1$}}

\newglossaryentry{Gamma0}{type=symbols,name={\ensuremath{\Gamma_0}},
symbol={$\Gamma_0$},sort=Gamma0, description={invertible matrix, appearing in
\eqref{density}}}

\newglossaryentry{tv}{type=symbols,name={\ensuremath{\tv{\cdot}}},
symbol={$\tv{\cdot}$},sort=tv, description={total variation norm:
$\tv{\nu, \eta}=\sup{f : \abs{f}\le 1} \abs{\nu(f)-\eta(f)}$}}

\newglossaryentry{l(s)}{type=symbols,name={\ensuremath{l(s)}},
symbol={$l(s)$},sort=l(s), description={$l(s)= \E_{\pist}^s V_1 = \frac{\kappa'(s^-)}{\kappa(s)}$}}

\newglossaryentry{l(0)}{type=symbols,name={\ensuremath{l(0)}},
symbol={$l(0)$},sort=l(s)1, description={$ l(0) =\lim_{n \to \infty} \frac{1}{n}
\log \abs{\Pi_n x} < 0 \quad \Pfs$,\newline upper Lyapunov exponent}}

\newglossaryentry{betamom_T}{type=symbols,name={\eqref{beta_moments}},
symbol={\eqref{beta_moments}},sort=0t, description={$\E \norm{\mT}^\beta \left( 
\abs{\log \norm{\mT}} + \abs{\log
\norm{\mT^{-1}}} \right) < \infty$}}

\newglossaryentry{betamom_Q}{type=symbols,name={\eqref{beta_moments_Q}},
symbol={\eqref{beta_moments_Q}},sort=0q, description={$0 < \E
\abs{Q}^\beta < \infty$}}

\newglossaryentry{Rneqr}{type=symbols,name={\eqref{nottrivial}},
symbol={\eqref{nottrivial}},sort=0r, description={$\forall\ r \in \Rd \quad
\P{\mT r + Q = r}<1$}}

\newglossaryentry{abs}{type=symbols,name={\ensuremath{\abs{\cdot}}},
symbol={$\abs{\cdot}$},sort=aaa, description={absolute value / euclidean norm on
$\R^d$}}

\newglossaryentry{skalar}{type=symbols,name={\ensuremath{\skalar{\cdot,\cdot}}},
symbol={$\skalar{\cdot,\cdot}$},sort=aaa, description={euclidean scalar
product on $\R^d$}}

\newglossaryentry{Cf}{type=symbols,name={\ensuremath{\Cf{E}}},
symbol={$\Cf(E)$},sort=C, description={set of continuous mappings $f : E
\to \R$}}

\newglossaryentry{Ccf}{type=symbols,name={\ensuremath{\Ccf{E}}},
symbol={$\Ccf(E)$},sort=Ccf, description={set of compactly supported continuous
mappings $f : E \to \R$}}

\newglossaryentry{Cbf}{type=symbols,name={\ensuremath{\Cbf{E}}},
symbol={$\Cbf(E)$},sort=Cbf, description={set of bounded continuous mappings $f
: E \to \R$}}

\newglossaryentry{Cmf}{type=symbols,name={\ensuremath{\Cf[m]{E}}},
symbol={$\Cf[m](E)$},sort=C, description={set of $m$-times continuously
differentiable mappings $f : E \to \R$}}

\newglossaryentry{Cof}{type=symbols,name={\ensuremath{\Cof{E}}},
symbol={$\Cof(E)$},sort=C0f, description={set of continuous mappings $f : E
\to \R$, that vanish at infinity}}

\newglossaryentry{infnorm}{type=symbols,name={\ensuremath{\infnorm{\cdot}}},
symbol={$\infnorm{\cdot}$},sort=aaa, description={$\infnorm{f}=\sup_{x
\in E} \abs{f(x)}$ for $f \in \Cf{E}$.}}

\newglossaryentry{opnorm}{type=symbols,name={\ensuremath{\norm{\cdot}}},
symbol={$\norm{\cdot}$},sort=aaa, description={operator norm:
$\norm{A}:=\sup_{\abs{x}=1} \abs{Ax}$}}

\newglossaryentry{MdR}{type=symbols,name={\ensuremath{M(d\times d, E)}},
symbol={$M(d\times d, E)$},sort=Md, description={set of $d \times d$ matrices
with entries in E }}

\newglossaryentry{mT}{type=symbols,name={\ensuremath{(\mT_i)_{i=1}^N}},
symbol={$\mT_i$},sort=Tmatrix, description={random matrices in $M(d \times d,\Rnn)$, w.l.o.g. identically distributed}}

\newglossaryentry{ST}{type=symbols,name={$\ST$}, symbol={$\ST$},sort=ST,
description={smoothing transform }}

\newglossaryentry{Pset}{type=symbols,name={\ensuremath{\Pset}},sort=Pset,
symbol={$\Pset$}, description={set of probability measures (on the specified
measurable space)}}

\newglossaryentry{Rnn}{type=symbols,name={\ensuremath{\Rnn}},sort=R,
symbol={$\Rnn(\cdot)$}, description={nonnegative real numbers $[0,\infty)$}}

\newglossaryentry{law}{type=symbols,name={\ensuremath{\mathcal{L}}},sort=law,
symbol={$\law{\cdot}$}, description={law (distribution) of the specified random
variable }}
\newglossaryentry{eta}{type=symbols,name={\ensuremath{\fpd}}, sort=eta,
symbol={$\fpd$}, description={probability measure, fixed point of $\ST$ }}

\newglossaryentry{T}{type=symbols,name={\ensuremath{\T}},sort=T,
symbol={$\T$}, description={$\T=(\mT_i)_{i=1}^N$}}

\newglossaryentry{Fset}{type=symbols,name={\ensuremath{\Fset}},sort=Fset,
symbol={$\Fset$}, description={set of fixed points of $\ST$}}

\newglossaryentry{dirac}{type=symbols,name={\ensuremath{\delta}},sort=dirac,
symbol={$\delta$}, description={Dirac measure in the specified point}}

\newglossaryentry{Fseta}{type=symbols,name={\ensuremath{\Fseta[\aleph]}},sort=Fseta,
symbol={$\Fseta[\aleph]$}, description={set of $\aleph$-elementary fixed points
of $\ST$}}

\newglossaryentry{mu}{type=symbols,name={\ensuremath{\mu}},
symbol={$\mu$},sort=mu, description={$\mu = \law{\mT_1}= \dots = \law{\mT_N}$}}

\newglossaryentry{mus}{type=symbols,name={\ensuremath{\mu^*}},
symbol={$\mu^*$},sort=mus, description={$\mu^* = \law{\mT_1^\top}$}}

\newglossaryentry{Mn}{type=symbols,name={\ensuremath{(\mM_n)_{n \in \N}}},
symbol={$(\mM_n)_{n \in \N}$},sort=M, description={sequence of \gls{iid}
random matrices with distribution $\mu^*$}}

\newglossaryentry{Tn}{type=symbols,name={\ensuremath{(\mT_{(n)})_{n \in \N}}},
symbol={$(\mT_{(n)})_{n \in \N}$},sort=Tn, description={sequence of \gls{iid}
random matrices with distribution $\mu$}}

\newglossaryentry{kappa}{type=symbols,name={\ensuremath{\kappa(s)}},
symbol={$\kappa(s)$},sort=kappa, description={$\kappa(s)=  \lim_{n \to \infty}
\left(\E \norm{\PiA_n}^s \right)^\frac1n = \lim_{n \to
\infty} \left( \E \norm{\Pi_n}^s \right)^\frac1n$}}

\newglossaryentry{m}{type=symbols,name={\ensuremath{m(s)}},
symbol={$m(s)$},sort=m, description={spectral function, Chapter
\ref{chapter:MST}: $m(s)= N \kappa(s) $, Chapter \ref{chapter:RDE}:
$m(s)=\kappa(s)$}}

\newglossaryentry{logmom_id}{type=symbols,name={\eqref{logmoments_RDE}},
symbol={\eqref{logmoments_RDE}},sort=0lo, description={$\E \log^+ \norm{\mT} +
\log^+ \abs{Q} < \infty $}}

\newglossaryentry{l<0}{type=symbols,name={\eqref{liapunov_exponent_negativ}},
symbol={\eqref{liapunov_exponent_negativ}},sort=0l, description={$l = \Pfs-\lim_{n \to \infty} \frac1n \log \norm{\PiA_n} < 0$}}

\newglossaryentry{PiA}{type=symbols,name={\ensuremath{\PiA_n}},
symbol={$\PiA_n$},sort=PiA, description={$\PiA_n = \mT_{(1)} \cdot \ldots \cdot \mT_{(n)}$}}

\newglossaryentry{R}{type=symbols,name={\ensuremath{R}},
symbol={$R$},sort=R, description={$\Rfp = \sum_{n=1}^\infty \PiA_{n-1} Q_n$,
$R \eqdist \mT R +Q$}}

\newglossaryentry{varrho}{type=symbols,name={\ensuremath{\varrho}},
symbol={$\varrho$},sort=rhov, description={$\varrho = \law{\mT,Q}$}}

\newglossaryentry{irred}{type=symbols,name={\eqref{irred}},
symbol={\eqref{irred}},sort=0i, description={$\forall_{x \in \Sd} \ \forall_{\text{open } U \subset \Sd}\ \max_{n \in \N}
\P{\Pi_n \as x \in U}>0$}}

\newglossaryentry{density}{type=symbols,name={\eqref{density}},
symbol={\eqref{density}},sort=0d, description={$\exists_{\Gamma_0 \in GL(d,\R)} \ \exists_{c, p >0} \ \exists_{n_0 \in
\N} \ \P{\Pi_{n_0} \in \cdot} \ge p \1[{B_c(\Gamma_0)}] \ldd$}}

\newglossaryentry{Imu}{type=symbols,name={\ensuremath{I_\mu}},
symbol={$I_\mu$},sort=Imu, description={$I_\mu = \{ s \ge 0 \ : \ \E
\norm{\mT_1}^s < \infty \}$}}

\newglossaryentry{sinfty}{type=symbols,name={\ensuremath{s_\infty}},
symbol={$s_\infty$},sort=sinfty, description={$s_\infty = \sup I_\mu$}}

\newglossaryentry{alpha}{type=symbols,name={\ensuremath{\alpha}},
symbol={$\alpha$},sort=alpha, description={$\alpha :=\inf \{s \in I_\mu :
m(s)\le 1\}. $ $\alpha \in \interior{I_\mu} \follows$ $m(\alpha)=1,$
$m'(\alpha) \le 0$}}

\newglossaryentry{limdownarrow}{type=symbols,name={\ensuremath{\lim_{t
\downarrow 0}}}, symbol={$\lim_{t
\downarrow 0}$},sort=lim, description={right sided limit in zero, $\lim_{t
\downarrow 0} = \lim_{t \to 0, t>0}$}}

\newglossaryentry{beta}{type=symbols,name={\ensuremath{\beta}},
symbol={$\beta$},sort=beta, description={$\beta :=\sup \{s \in I_\mu :
m(s)\le 1\}. $ $\beta \in \interior{I_\mu} \follows$ $m(\beta)=1,$
$m'(\beta) \ge 0$}}

\newglossaryentry{interior}{type=symbols,name={\ensuremath{\interior{}}},
symbol={\interior{}},sort=aaa, description={topological
interior of the specified set }}

\newglossaryentry{STi}{type=symbols,name={\ensuremath{\STi}},
symbol={$\STi$},sort=STQ, description={inhomogeneous smoothing transform
\newline $\STi \ : \ \nu \mapsto \law{\sum_{i=1}^N \mT_i Y_i + Q}$}}

\newglossaryentry{STh}{type=symbols,name={\ensuremath{\STh}},
symbol={$\STh$},sort=ST0, description={homogeneous smoothing transform
$\STh \ : \ \nu \mapsto \law{\sum_{i=1}^N \mT_i Y_i}$}}

\newglossaryentry{LTZ}{type=symbols,name={\ensuremath{\LTa_Z}},
symbol={$\LTa_Z$},sort=phi, description={Laplace transform of $Z$,
$\LTa_Z(x)=\E\ {\exp\left( - \skalar{x,Z} \right)}$}}

\newglossaryentry{LTeta}{type=symbols,name={\ensuremath{\LTa_\eta}},
symbol={$\LTa_\eta$},sort=phie, description={Laplace transform of $\eta$,
$\LTa_\eta(x)=\int_{\Rdnn} e^{-\skalar{x,y}} \eta(dy)$}}

\newglossaryentry{measuresetc}{type=symbols,name={\ensuremath{\measureset[c]{E}}},
symbol={\measureset[c]{\Rdnn}},sort=Mc, description={set of measures on $E$
with total mass less or equal $c$}}

\newglossaryentry{measuresetpm}{type=symbols,name={\ensuremath{\measureset[\pm]{E}}},
symbol={\measureset[\pm]{E}},sort=Mpm, description={ (vector space) of
regular bounded signed measures on $E$ }}

\newglossaryentry{vaguec}{type=symbols,name={\ensuremath{\vto}},
symbol={$\vto$},sort=aaa, description={vague convergence}}

\newglossaryentry{weakc}{type=symbols,name={\ensuremath{\dto}},
symbol={$\dto$},sort=aaa, description={convergence in distribution, weak
convergence}}

\newglossaryentry{Rp}{type=symbols,name={\ensuremath{\Rp}},sort=Rp,
symbol={$\Rp$}, description={positive half-line $(0,\infty)$}}

\newglossaryentry{N}{type=symbols,name={\ensuremath{\N}},sort=N,
symbol={$\N$}, description={positive integers $\{1, 2, \dots \}$}}

\newglossaryentry{No}{type=symbols,name={\ensuremath{\No}},sort=N0,
symbol={$\No$}, description={natural numbers $\{0, 1, 2, \dots\}$}}

\newglossaryentry{eqdist}{type=symbols,name={\ensuremath{\eqdist}},
symbol={$\eqdist$},sort=aaa, description={equal in distribution}}

\newglossaryentry{Sd}{type=symbols,name={\ensuremath{\Sd}},
symbol={$\Sd$},sort=Sd, description={unit sphere in $\Rd$}}

\newglossaryentry{Sp}{type=symbols,name={\ensuremath{\Sp}},
symbol={$\Sp$},sort=Sp, description={$\Sd \cap \Rdnn$}}

\newglossaryentry{gotimesh}{type=symbols, name={\ensuremath{g \otimes h}},
symbol={$g \otimes h$}, sort=aaa, description={$g \otimes h (x,y) = g(x) h(y)$}}

\newglossaryentry{nuotimeseta}{type=symbols, name={\ensuremath{\nu \otimes
\eta}}, symbol={$\nu \otimes \eta$}, sort=aaa, description={product measure}}

\newglossaryentry{Salpha1}{type=symbols, name={\ensuremath{S_\alpha(\sigma,
\lambda, b)}}, symbol={$S_\alpha(\sigma, \lambda, b)$}, sort=Salpha,
description={one-dimensional stable distribution with scale parameter
$\sigma$, skewness parameter $\lambda$ and shift parameter $b$}}

\newglossaryentry{Salphad}{type=symbols, name={\ensuremath{S_\alpha(\K \nu,
 b)}}, symbol={$S_\alpha(\K \nu, b)$}, sort=SalphaK,
description={multivariate stable distribution with scale parameter
$\K$, spectral measure $\nu$ and shift parameter $b$}}

\newglossaryentry{Salphan}{type=symbols, name={\ensuremath{\stablen{\alpha}{\K
\nu, 0}}}, symbol={$\stablen{\alpha}{\K
\nu, 0}$}, sort=Salphan,
description={multivariate stable distribution with \gls{LT} $\exp\left( -\K
\int_{\Sp} \skalar{x,y}^\alpha \nu(dy) \right) $, $x \in \Rdnn$}}

\newglossaryentry{Dn}{type=symbols, name={\ensuremath{D^{\mathbf{n}}}}, symbol={$D^{\mathbf{n}}$}, 
sort=Dn,
description={$D^{\mathbf{n}} f(x) =
\frac{\partial^{\sum \mathbf{n}_i}}{\partial_1^{\mathbf{n}_1} \ldots
\partial_d^{\mathbf{n}_d}} f (x)$ for $\mathbf{n} \in \No^d$}}

\newglossaryentry{CcfR0}{type=symbols, name={\ensuremath{\Ccf{\overline{\Rdnn} \setminus \{0\}}}},
symbol={$\Ccf{\overline{\Rdnn} \setminus \{0\}}$}, sort=CcfR0,
description={functions $f \in \Cbf{\Rdnn}$ which are supported away from the
origin}}

\newglossaryentry{lalpha}{type=symbols, name={\ensuremath{\lalpha}},
symbol={$\lalpha$}, sort=lebesgue, description={
$\lalpha(ds)=\frac{1}{s^{1+\alpha}} ds$}}

\newglossaryentry{deins}{type=symbols, name={\ensuremath{\deins}},
symbol={$\deins$}, sort=thetad, description={$\deins := (1, \dots,1)^\top
\in \Rd$}}

\newglossaryentry{eins}{type=symbols, name={\ensuremath{\eins}},
symbol={$\eins$}, sort=theta1, description={$\eins := \sqrt{d}^{-1}(1,
\dots,1)^\top \in \Rd$}}

\newglossaryentry{awedgeb}{type=symbols, name={\ensuremath{ \wedge }},
symbol={$ \wedge $}, sort=aaa, description={$(a \wedge b)_i=\min\{a_i, b_i\}$,
$i=1, \ldots, d$ for $a,b \in \Rd$}}

\newglossaryentry{aveeb}{type=symbols, name={\ensuremath{ \vee }},
symbol={$ \vee $}, sort=aaa, description={$(a \vee b)_i=\max\{a_i, b_i\}$,
$i=1, \ldots, d$ for $a,b \in \Rd$}}

\newglossaryentry{U}{type=symbols, name={\ensuremath{ \U }},
symbol={$ \U $}, sort=U, description={renewal measure}}

\newglossaryentry{Ux}{type=symbols, name={\ensuremath{ \U[x] }},
symbol={$ \U[x] $}, sort=Ux, description={Markov renewal measure,
$\U[x]=\sum_{n=0}^\infty \P[x]{(X_n, V_n) \in \cdot}$, \newline
$g*\U[x](t)=\Erw[x]{\sum_{n=0}^\infty g(X_n, t-V_n)}$}}

\newglossaryentry{Ps}{type=symbols, name={\ensuremath{ \Ps }},
symbol={$ \Ps$}, sort=Ps, description={$\Ps f(x)= \Erw{\abs{\mT_1 x}^s f(\mT_1 \as x)} $}}

\newglossaryentry{Pst}{type=symbols, name={\ensuremath{ \Pst }},
symbol={$ \Pst$}, sort=Pst, description={$P^s_* f(x) =
\Erw{\abs{\mT_1^\top x}^s f(\mT_1^\top \as x)}= \Erw{\abs{\mM_1 x}^s
f(\mM_1 \as x)}$}}

\newglossaryentry{Mset}{type=symbols, name={\ensuremath{ \Mset }},
symbol={$ \Mset$}, sort=Mset, description={$\Mset=M(d \times d, \Rdnn)$}}

\newglossaryentry{supp}{type=symbols, name={\ensuremath{ \supp }},
symbol={$ \supp$}, sort=Mset, description={support of a measure $\mu$: 
$\{ x \ : \ \forall \text{ open } O \text{ with } x \in O, \mu(O)>0\} $}}

\newglossaryentry{suppmu}{type=symbols, name={\ensuremath{ [\supp \mu] }},
symbol={$ [\supp \mu]$}, sort=Mset, description={smallest closed subsemigroup
containing $\supp \mu$}}

\newglossaryentry{Wperp}{type=symbols, name={\ensuremath{ {}^\perp }},
symbol={$ {}^\perp$}, sort=aaa, description={For $W \subset \Rd$,
$W^\perp = \{x \in \Rd \ : \ \skalar{x,y}=0\ \forall y \in W \}$}}

\newglossaryentry{Beps}{type=symbols, name={\ensuremath{ B_\epsilon(x) }},
symbol={$ B_\epsilon(x)$}, sort=B, description={$B_\epsilon(x) := \{ y \ : \
\abs{x-y}< \epsilon\}$}}

\newglossaryentry{XnVn}{type=symbols, name={\ensuremath{ (X_n, V_n)_{n \in \No}}},
 symbol={$\(X_n, V_n)_{n \in \No} $  }, sort=Xn, description={Markov random
 walk; $X_n := \Pi_n \as X_0$, $V_n := \log \abs{\Pi_n X_0}$}}

\newglossaryentry{ldd}{type=symbols, name={\ensuremath{\ldd}}, symbol={$\ldd $},
sort=ldd, description={Lebesgue measure on $M(d \times d,\R) \subset
\R^{d^2}$}}

\newglossaryentry{lambdaA}{type=symbols, name={\ensuremath{\lambda_\matrix{A}}}, 
symbol={$\lambda_\matrix{A}$}, sort=lambdaA, description={Perron-Frobenius
eigenvalue of $\matrix{A}$}}

\newglossaryentry{Hgamma}{type=symbols, name={\ensuremath{H^\gamma(E)}}, 
symbol={$H^\gamma(E)$}, sort=Hgamma, description={set of $\gamma$-H\"older
functions on E: \newline
 $H^\gamma = \{ f \in \Cf{E} \ : \ \sup_{x,y \in {{ E }}}
\frac{\abs{f(x)-f(y)}}{\abs{x-y}^\gamma} < \infty \}$}}

\newglossaryentry{uA}{type=symbols, name={\ensuremath{u_\matrix{A}}}, symbol={$u_\matrix{A} $},
sort=uA, description={normalized Perron-Frobenius eigenvector of $\matrix{A}$}}

\newglossaryentry{complement}{type=symbols, nonumberlist,
name={\ensuremath{A^c}}, symbol={$A^c$}, sort=aaa, description={For $A \subset E$: $A^c = E \setminus
A$}}

\newglossaryentry{Lambda}{type=symbols, name={\ensuremath{\Lambda(\Gamma)}},
symbol={$\Lambda(\Gamma) $}, sort=Lambda,
description={$\Lambda(\Gamma) =\overline{\{u_\matrix{A} \ : \ \matrix{A} \in
\Gamma
\cap
\interior{\Mset} \}}$}}

\newglossaryentry{Ck}{type=symbols, name={\ensuremath{ C_k}},
 symbol={$ C_k $}, sort=Ck, description={$ C_k = \left\{ x \in S \ : \
 \P[x]{\frac{V_m}{m} \ge \frac1k \ \forall m \ge k} \ge \frac{1}{2} \right\} $}}

\newglossaryentry{V}{type=symbols, name={\ensuremath{V}},
 symbol={$V$}, sort=V, description={$V= \Rd \setminus{\{0\}} $}}

\newglossaryentry{NA}{type=symbols, name={\ensuremath{ N(A)}},
 symbol={$ N(A) $}, sort=NA, description={number of renewals (visits of the
 random walk) in the set $A$}}
 
 \newglossaryentry{es}{type=symbols, name={\ensuremath{ \es}},
 symbol={$ \es $}, sort=es, description={$\Ps \es = \kappa(s) \es$,
 $\abs{\es}_\infty=1$}}

 \newglossaryentry{est}{type=symbols, name={\ensuremath{ \est}},
 symbol={$ \est $}, sort=est, description={$\Pst \est = \kappa(s) \est$,
 $\abs{\est}_\infty=1$}}

\newglossaryentry{nus}{type=symbols, name={\ensuremath{ \nus}},
 symbol={$ \nus $}, sort=nu, description={probability measure, $\Ps \nus =
 \kappa(s) \nus$}}

\newglossaryentry{nust}{type=symbols, name={\ensuremath{ \nust}},
 symbol={$ \nust $}, sort=nust, description={probability measure, $\Pst \nust =
 \kappa(s) \nust$}}
 
 \newglossaryentry{mPi}{type=symbols, name={\ensuremath{ \mPi_n}},
 symbol={$ \mPi_n $}, sort=Pi, description={$\mPi_n=\mM_n \ldots \mM_1$}}
 
\newglossaryentry{Qx}{type=symbols, name={\ensuremath{ \Q_x}},
 symbol={$ \Q_x $}, sort=Qx, description={$\Q_x = \delta_x \otimes \bigotimes_{n=1}^\infty \mu^*$}}
 
\newglossaryentry{Ex}{type=symbols, name={\ensuremath{\E_x}}, symbol={$\E_x$},
sort=Ex, description={expectation symbol of $\Q_x$}}

\newglossaryentry{Exs}{type=symbols, name={\ensuremath{\E_x^s}},
symbol={$\E_x^s$}, sort=Exs, description={expectation symbol of $\Q_x^s$,
satisfying\newline $\E_x^s \left(f((X_i,
V_i)_{i=1}^n) \right) = \frac{1}{\est(x)\kappa^n(s)}\Erw[x]{e^{s
V_n}\est(X_n)((X_i,
V_i)_{i=1}^n)}$ \newline for all bounded measurable functions
$f$ and all $n \in \N$.}}

\newglossaryentry{Qsx}{type=symbols, name={\ensuremath{\Q_x^s}},
symbol={$\Q_x^s$}, sort=Qsx, description={probability measure on $(\Omega, \A)$,
projective limit of $\ _n\Q_x^s$}}

\newglossaryentry{Qnsx}{type=symbols, name={\ensuremath{\ _n\Q_x^s}},
symbol={$\ _n\Q_x^s$}, sort=Qnsx, description={$ \ _n\QQ_x^s((X_0, (\mM_i)_{i=1}^n) \in
A):=$\newline
$\frac{1}{\est(x)\kappa^n(s)}\Erw[x]{e^{s V_n}\est(X_n)\1[A](X_0,
(\mM_i)_{i=1}^n)}$}}

\newglossaryentry{Qkernel}{type=symbols, name={\ensuremath{ Q^s_*}},
symbol={$ Q^s_*$}, sort=Qk, description={Markov transition kernel on $\Cf{S}$,
$Q^s_* f(x) =
\frac{1}{\est(x)\kappa(s)}\Pst (\est f)(x)$}}

\newglossaryentry{pist}{type=symbols, name={\ensuremath{ \pist}},
symbol={$ \pist$}, sort=pist, description={stationary distribution for $Q^s_*$,
$ \pist(dx) = \est(x) \nust(dx)$}}

\newglossaryentry{Q}{type=symbols, name={\ensuremath{ \QQ^s}},
symbol={$ \QQ^s$}, sort=Q, description={$\QQ^s = \int \QQ^s_x \pist(dx)$,
$(X_n)_{n
\in \No}$ is stationary under $\QQ^s$}}

\newglossaryentry{ghat}{type=symbols, name={\ensuremath{ \hat{g}}},
symbol={$ \hat{g}$}, sort=g, description={For $g: \Sp \times \R \to
\R$, $\hat{g}(t) =\sup_{u \in S} \abs{g(u,t)}$}}

\newglossaryentry{tree}{type=symbols, name={\ensuremath{ \tree}},
symbol={$ \tree$}, sort=t, description={$N$-ary Ulam-Harris tree, $\tree := \bigcup_{n=0}^\infty \{1, \dots, N\}^n$}}

\newglossaryentry{Ttree}{type=symbols, name={\ensuremath{ \Ttree}},
symbol={$ \Ttree$}, sort=Ttree, description={$\Ttree:=(\T(v))_{v \in \tree}$,
sequence of \gls{iid} copies of $\T$}}

\newglossaryentry{Ytree}{type=symbols, name={\ensuremath{ \Ytree}},
symbol={$ \Ytree$}, sort=Ytree, description={$\Ytree:=(Y(v))_{v \in \tree}$,
sequence of \gls{iid} copies of a \gls{rv} $Y$}}

\newglossaryentry{L}{type=symbols, name={\ensuremath{ \matrix{L}(v)}},
symbol={$ \matrix{L}(v)$}, sort=L, description={matricial path weights in the
weighted branching tree, recursively defined by $\matrix{L}(\emptyset)=\Id$, 
$\matrix{L}(vi)= \matrix{L}(v)\mT_i(v)$}}

\newglossaryentry{Id}{type=symbols, name={\ensuremath{ \Id}},
symbol={$ \Id$}, sort=Id, description={identity matrix}}

\newglossaryentry{Yn}{type=symbols, name={\ensuremath{ Y_n}},
symbol={$ Y_n$}, sort=Yn, description={weighted branching process associated
with $\Ttree \otimes \Ytree$,\newline $Y_n := \sum_{\abs{v}=n} \matrix{L}(v)Y(v)
+
\sum_{k=0}^{n-1} \sum_{\abs{v}=k}
\matrix{L}(v)Q(v)$}}

\newglossaryentry{tshift}{type=symbols, name={\ensuremath{ \tshift{\cdot}{v}}},
symbol={$ \tshift{\cdot}{v}$}, sort=aaa, description={tree shift operator }}

\newglossaryentry{vlim}{type=symbols, name={\ensuremath{ v\mathrm{-lim}}},
symbol={$ v\mathrm{-lim}$}, sort=v, description={vague limit }}

\newglossaryentry{dlim}{type=symbols, name={\ensuremath{ d\mathrm{-lim}}},
symbol={$ d\mathrm{-lim}$}, sort=d, description={weak limit, limit in
distribution }}

\newglossaryentry{Psets}{type=symbols, name={\ensuremath{ \Pset[s](E)}},
symbol={$ \Pset[s](E)$}, sort=Psets, description={set of probability measures on
$E$ with finite $s$-th moment }}

\newglossaryentry{Fsets}{type=symbols, name={\ensuremath{ \Fset[s]}},
symbol={$ \Fset[s]$}, sort=Fsets, description={$\Fset \cap \Pset[s](\Rd)$ }}

\newglossaryentry{Dk}{type=symbols, name={\ensuremath{ D^k}},
symbol={$ D^k$}, sort=Dk, description={$D^k f$: $k$-th Fr\'echet derivative of
$f$ }}

\newglossaryentry{Psetsz}{type=symbols, name={\ensuremath{ \Pset[s,z](E)}},
symbol={$ \Pset[s,z](E)$}, sort=Psets, description={set of probability measures
on $E$ with finite $s$-th moment and fixed mixed moment sequence $z$ }}

\newglossaryentry{EZn}{type=symbols, name={\ensuremath{ \E Z^\mathbf{n}}},
symbol={$ \E Z^\mathbf{n}$}, sort=EZn, description={$\E Z^\mathbf{n} =
\Erw{Z_{1}^{\mathbf{n}_1} \cdots Z_{d}^{\mathbf{n}_d}}$ for a multiindex $\mathbf{n} \in \No^d$ }}

\newglossaryentry{Zolset}{type=symbols, name={\ensuremath{ \Zolset{s}}},
symbol={$ \Zolset{s}$}, sort=Ds, description={$\left\{ f \in
\Cf[k]{\Rd} \ : \ \forall_{x,y \in \Rd} \ \norm{D^k f(x) - D^k f(y)} \le \abs{x-y}^{s-k}
\right\}$ }}

\newglossaryentry{zetas}{type=symbols, name={\ensuremath{ \zeta_{s}}},
symbol={$ \zeta_{s}$}, sort=zetas, description={Zolotarev metric }}

\newglossaryentry{Psetsw}{type=symbols, name={\ensuremath{ \Pset[s,w](E)}},
symbol={$ \Pset[s,w](E)$}, sort=Psetsw, description={set of probability measures
on $E$ with finite $s$-th moment and expectation $w$ }}

\newglossaryentry{Psetsws}{type=symbols, name={\ensuremath{
\Pset[s,w,\Sigma](E)}}, symbol={$ \Pset[s,w,\Sigma](E)$}, sort=Psetsws,
description={set of probability measures on $E$ with finite $s$-th moment,
expectation $w$ and covariance matrix $\Sigma$ }}

\newglossaryentry{Cov}{type=symbols, name={\ensuremath{
\Cov}}, symbol={$ \Cov$}, sort=K,
description={covariance matrix }} 

\newglossaryentry{eigenspace}{type=symbols, name={\ensuremath{
\Eigenspace{\matrix{A},\lambda}}}, symbol={$ \Eigenspace{\matrix{A},\lambda}$},
sort=Eig, description={eigenspace of
$\matrix{A}$ for eigenvalue $\lambda$ }}

\newglossaryentry{eigenspace0}{type=symbols, name={\ensuremath{
\Eigenspaceo{\matrix{A},\lambda}}}, symbol={$
\Eigenspaceo{\matrix{A},\lambda}$}, sort=Eig0,
description={$\Eigenspace{\matrix{A}, \lambda} \cup \{0\}$ }}

\newglossaryentry{N0sigma}{type=symbols, name={\ensuremath{
N(0, \Sigma)}}, symbol={$ N(0, \Sigma)$}, sort=N0S,
description={multivariate Normal distribution
with expectation $0$ and covariance matrix $\Sigma$ }}

\newglossaryentry{wd}{type=symbols, name={\ensuremath{
w_d}}, symbol={$ w_d$}, sort=wd,
description={Wasserstein metric }}  

\newglossaryentry{Psetwd}{type=symbols, name={\ensuremath{
\Pset[w_d](\eta)}}, symbol={$ \Pset[w_d](\eta)$}, sort=Psetwd,
description={$ \Pset[w_d](\eta) = \{ \nu \in \Pset(E)  \
: \ w_d(\eta, \nu) < \infty \} $ }} 

\newglossaryentry{ls}{type=symbols, name={\ensuremath{
l_s}}, symbol={$ l_s$}, sort=ls,
description={ $l_s(\nu, \eta) := \inf \{ \E \abs{Y-Z}^s \ : \ \law{Y}=\nu,
\law{Z}=\eta .\}$  }} 

\newglossaryentry{iota}{type=symbols, name={\ensuremath{
\iota(\mA)}}, symbol={$\iota(\mA)$}, sort=iota,
description={ $\iota(\matrix{A}) =\inf_{x \in \Sp}\abs{\mA
x}$  }} 

\newglossaryentry{Ttreen}{type=symbols, name={\ensuremath{
\Ttree_n}}, symbol={$\Ttree_n$}, sort=Ttreen,
description={ filtration of $\Ttree$, $\Ttree_n = \sigma\left( (T(v))_{\abs{v}
\le n} \right)$ }}

\newglossaryentry{Psetseta}{type=symbols, name={\ensuremath{
\Pset[w_d](\eta)}}, symbol={$ \Pset[s](\eta)$}, sort=Psetseta,
description={$ \Pset[s](\eta) = \{ \nu \in \Pset(\Rd) \ : \ l_s(\nu, \eta)< \infty \} $ }}

\newglossaryentry{FsetQ}{type=symbols, name={\ensuremath{ \Fset[Q]}},
symbol={$ \Fset[Q]$}, sort=FsetQ, description={set of fixed points of
$\STi$ }}

\newglossaryentry{Fset0}{type=symbols, name={\ensuremath{ \Fset[0]}},
symbol={$ \Fset[0]$}, sort=Fset0, description={set of fixed points of
$\STh$ }}

\newglossaryentry{momcond}{type=symbols, name={(s-moments)},
symbol={(s-moments)}, sort=0m, description={$\Erw{\norm{\mT_1}^s
+ \abs{Q}^s}< \infty$ }}

\newglossaryentry{eigcond}{type=symbols, name={(eigenvalue)},
symbol={(eigenvalue)}, sort=0e, description={$w = N \E \mT_1 w + \E Q$ }}

\newglossaryentry{varcond}{type=symbols, name={(variance)},
symbol={(variance)}, sort=0v, description={$\Sigma = N \Erw{\mT_1 \Sigma
\mT_1^\top}$ }}

\newglossaryentry{esl}{type=symbols, name={\ensuremath{\esl{x}}},
symbol={$\esl{x}$}, sort=aaa, description={$\esl{x}= \abs{x}^{-1} x$ for $x \in
\Rd$ }}

\newglossaryentry{barg}{type=symbols, name={\ensuremath{\bar{g}}},
symbol={$\bar{g}$}, sort=aaa, description={$\bar{g}(y,t) =
\int_{-\infty}^t e^{- (t-s)} g(y,s) ds$ }}

\newglossaryentry{actions}{type=symbols, name={\ensuremath{\as}},
symbol={$\as$}, sort=aaa, description={action of matrices on the sphere,
$\matrix{A} \cdot x = \esl{\matrix{A}x}$ }}

\newglossaryentry{Bbmf}{type=symbols, name={\ensuremath{\Bbmf{E}}},
symbol={$\Bbmf{E}$}, sort=Bbmf, description={bounded Borel-measurable functions
$E \to \R$ with $\infnorm{f} \le 1$. }}

\newglossaryentry{SToffunction}{type=symbols, name={\ensuremath{\ST f(x)}},
symbol={$\ST f(x)$}, sort=STf, description={$\ST f(x) = \Erw{\prod_{i=1}^N
f(\mT_i^\top x)}$ }}

\newglossaryentry{Salphae}{type=symbols, name={\ensuremath{\stablen{\alpha}{\K
\est[\alpha], 0}}}, symbol={$\stablen{\alpha}{\K\est[\alpha]
, 0}$}, sort=Salphae,
description={multivariate stable distribution with \gls{LT} $\LTa(tu)=\exp\left(
-\K t^\alpha \est[\alpha](u) \right) $}}

\newglossaryentry{Dphi}{type=symbols, name={\ensuremath{D_{\chi, \LTa}}},
symbol={$D_{\chi, \LTa}$}, sort=Dphi, description={$D_{\chi, \LTa}(u,t)=  \frac{e^{\chi t}}{\est(u)} (1-
\LTa(e^{-t}u))$ }}

\newglossaryentry{Gphi}{type=symbols, name={\ensuremath{G_{\chi, \LTa}(u,t)}},
symbol={$G_{\chi, \LTa}(u,t)$}, sort=Gphi, description={$= 
\frac{e^{\chi t}}{\est(u)}\Erw{\prod_{i=1}^N \LTa(e^{-t}\mT_i^\top u) + \sum_{i=1}^N\left(
1- \LTa(e^{-t}\mT_i^\top u) \right) -1}$ }}

\newglossaryentry{Dan}{type=symbols, name={\ensuremath{D_{\alpha,n}}},
symbol={$D_{\alpha,n}$}, sort=Dan, description={$D_{\alpha,n}= D_{\alpha, \ST^n \LTa_0} $ }}

\newglossaryentry{Gan}{type=symbols, name={\ensuremath{G_{\alpha,n}}},
symbol={$G_{\alpha,n}$}, sort=Gan, description={$G_{\alpha,n}= G_{\alpha, \ST^n \LTa_0} $ }}

\newglossaryentry{Pkernalpha}{type=symbols, name={\ensuremath{\Pkern[\alpha]}},
symbol={$\Pkern[s]$}, sort=Pkernalpha, description={$\Pkern[s] f(u,t)
= \E_u^s f(X_1, t-V_1$) }}

\newglossaryentry{Uxalpha}{type=symbols, name={\ensuremath{\U[x]^s}},
symbol={$\U[x]^s$}, sort=Uxs, description={Markov renewal measure of
$(X_n, V_n)_{n \in \No}$ under $\Q_x^s$  }}

\newglossaryentry{MlogM}{type=symbols, name={(M~logM)},
symbol={(M~logM)}, sort=0m, description={$\E
(1+\norm{\mM_1}) \left(1+ \abs{\log\norm{\mM_1}} + \abs{\log \iota(\mM_1)}
\right) < \infty$}}

\newglossaryentry{Wn}{type=symbols, name={\ensuremath{W_n}},
symbol={$W_n$}, sort=Wn, description={Biggins martingale for $\alpha=1$: $W_n =
\sum_{\abs{v}=n} \matrix{L}(v) w$, $w = N\E \mT_1 w$}}

\newglossaryentry{Wnu}{type=symbols, name={\ensuremath{W_n(u)}},
symbol={$W_n(u)$}, sort=Wnu, description={Biggins martingale for
$\alpha<1$:\newline $W_n(u) = \sum_{\abs{v}=n} \int_{\Sp}
\skalar{\matrix{L}(v)^\top u, y}^\alpha \nu^\alpha(dy)$}}

\newglossaryentry{Wu}{type=symbols, name={\ensuremath{W(u)}},
symbol={$W(u)$}, sort=Wu, description={\gls{as} limit of $W_n(u)$}}

\newglossaryentry{Wn*}{type=symbols, name={\ensuremath{W_n^*}},
symbol={$W_n^*$}, sort=Wns, description={$ W_n^*= \sum_{k=0}^{n-1} \sum_{\abs{v}=k} \matrix{L}(v) Q(v) $}}

\newglossaryentry{W*}{type=symbols, name={\ensuremath{W^*}},
symbol={$W^*$}, sort=Ws, description={\gls{as} limit of $W_n^*$}}

\newglossaryentry{phi0}{type=symbols, name={\ensuremath{\LTa_0}},
symbol={$\LTa_0$}, sort=phi0, description={(In Section \ref{sect:existence}:)
$\LTa_0 (tu) = \exp\left( -\K t^\alpha \est[\alpha](u) \right) $}}

\newglossaryentry{psi}{type=symbols, name={\ensuremath{\LTfp}},
symbol={$\LTfp$}, sort=psi, description={(In Section \ref{sect:existence}:)
$\LTfp = \lim_{n \to \infty} \ST^n \LTa_0$}}

\newglossaryentry{STchi}{type=symbols, name={\ensuremath{\ST_\chi}},
symbol={$\ST_\chi$}, sort=STchi, description={$\ST_\chi \ : \ \nu \mapsto \law{\sum_{i=1}^N
\frac{1}{m(\chi)^{\frac{1}{\chi}}}\mT_i X_i}$}}

\newglossaryentry{Tchi}{type=symbols, name={\ensuremath{\T_\chi}},
symbol={$\T_\chi$}, sort=Tchi, description={$\T_\chi = (\mT_{\chi,1}, \dots, \mT_{\chi,N}) =
{m(\chi)^{-\frac{1}{\chi}}}(\mT_1, \dots, \mT_N)$}}

\newglossaryentry{muchi}{type=symbols, name={\ensuremath{\mu_\chi}},
symbol={$\mu_\chi$}, sort=muchi, description={$\mu_\chi =
\law{{m(\chi)^{-\frac{1}{\chi}}}\mT_1}$}}

\newglossaryentry{mchi}{type=symbols, name={\ensuremath{m_\chi}},
symbol={$m_\chi$}, sort=mchi, description={spectral function of $\T_\chi$}}

\newglossaryentry{Jchi}{type=symbols, name={\ensuremath{J_\chi}},
symbol={$J_\chi$}, sort=Jchi, description={compact subset of $\Cf{\Sp \times
\R}$, see Def. \ref{defnJchi}}}

\newglossaryentry{hs}{type=symbols, name={\ensuremath{h_s}},
symbol={$h_{s}$}, sort=hs, description={$h_s(u,t) =
\frac{D_{\chi,\LTa}(u,s+t)}{e^{\chi s}(1-\LTa(e^{-s}\deins))} = \frac{e^{\chi
t}}{e_*^\chi(u)}\frac{1- \LTa(e^{-(s+t)}u)}{1- \phi(e^{-s}\deins)}$}}

\newglossaryentry{Hchi}{type=symbols, name={\ensuremath{H_{\chi,c}}},
symbol={$H_{\chi,c}$}, sort=Hchi, description={compact subset of $J_\chi$, see
Def. \ref{def:H_chi}}}

\newglossaryentry{logM}{type=symbols, name={\eqref{logM}},
symbol={\eqref{logM}}, sort=0l, description={$\E \abs{\log \norm{\mM_1}} + \abs{\log
\iota(\mM_1)} + \abs{\log \iota(\mT_1)} < \infty$}}

\newglossaryentry{Lt}{type=symbols, name={\ensuremath{L_t}},
symbol={$L_t$}, sort=Lt, description={$L_t(u,r)= \frac{g(u,t+r)}{g(u,r)}$ for
$g \in H_{\chi,c}$}}

\newglossaryentry{Echi}{type=symbols, name={\ensuremath{E_{\chi,c}}},
symbol={$E_{\chi,c}$}, sort=Hchi, description={extremal points of
$H_{\chi,c}$,\newline $ E_{\chi,c}= \left\{ (u,t) \mapsto
c\frac{e_*^\gamma(u)}{e_*^\chi(u)} e^{(\chi-\gamma)t} \ : \ \gamma \in (0,1], m(\gamma)=1 \right\} .$}}

\newglossaryentry{FsetaQ}{type=symbols,name={\ensuremath{\Fseta[\alpha]_Q}},sort=FsetaQ,
symbol={$\Fseta[\alpha]_Q$}, description={set of $\alpha$-elementary fixed
points of $\STi$}}

\newglossaryentry{Fseta0}{type=symbols,name={\ensuremath{\Fseta[\alpha]_0}},sort=Fseta0,
symbol={$\Fseta[\alpha]_0$}, description={set of $\alpha$-elementary fixed
points of $\STh$}}

\newglossaryentry{Rn}{type=symbols,name={\ensuremath{R_n}},sort=Rn,
symbol={$R_n$}, description={$ R_n = \max_{\abs{v}=n} \norm{L(v)}$}}

\newglossaryentry{Lapl}{type=symbols,name={\ensuremath{\L}},sort=L,
symbol={$\L$}, description={Laplace transform mapping $\eta \mapsto \LTa_\eta$}}

\newglossaryentry{LRdnn}{type=symbols,name={\ensuremath{\L(\Rdnn)}},sort=LRd,
symbol={$\L(\Rdnn)$}, description={space of \glspl{LT} of distributions on
$\Rdnn$}}

\newglossaryentry{c_s}{type=symbols,name={\ensuremath{c_s}},sort=cs,
symbol={$c_s$}, description={$c_s = {\abs{\int \skalar{\cdot,y}^s
\nus(dy)}_\infty}$}}

\newglossaryentry{Omega}{type=symbols,name={\ensuremath{\Omega}},sort=Omega,
symbol={$\Omega$}, description={$ \Omega = S \times \mathcal{M}^\N $}}

\newglossaryentry{eigenspacep}{type=symbols, name={\ensuremath{
\Eigenspacenn{\Pst[\alpha],\lambda}}}, symbol={$
\Eigenspace{\Pst[\alpha],\lambda}$}, sort=Eigp, description={positive
eigenfunctions of $\Pset[\alpha]$ for eigenvalue $\lambda$ }}

%% file: symbols.tex
\begin{tabular}{cp{0.85\textwidth}}
$\deins, \eins$ & $\deins=(1, \dots,1)^\top \in \Rdnn$, $\eins=d^{-1/2}
\deins \in \Sp$ \\

$\eqdist$ & same law \\

$\interior{A}$ & topological interior of the set $A$\\

$\tshift{\cdot}{v}$ & shift operator in $\tree$, see \eqref{eq:tshift} \\  


$\alpha$-regular & see Definition \ref{defn:regular.elementary}. \\

$\B_n$ & filtration, $\B_n=\sigma\Bigl((\T(v))_{\abs{v}<n} \Bigr)$
\\
$ \B_{\slineu[t]} $ &  $\B_{\slineu[t]} := \sigma\left(
U(\emptyset), \{ \T(v) \ : \ v \text{ has no ancestor in } \slineu[t] \}\right). $
\\

$\condC$ & condition imposed on the supp of $\mu$,
 see Definition \ref{defn:condc} \\ 

 $\D, \D[L,s]$ & $\D(x) := \frac{1-\LTa(x)}{H^\alpha(x)}$,  $\D[L,s](u,t) = \frac{1- \LTa(e^{s+t}u)}{H^\alpha(e^tu) e^{\alpha
 s}L(e^s)}$
\\
 
$ \E^\alpha_u$ & expectation symbol of $\Prob^\alpha_u$, see \eqref{eq:manytoone1}.
\\
 $ \mathcal{E}_{\alpha,c}$ & extremal points of $H_{\alpha,c}^K$, see Lemma
\ref{lem:extremal_points} \\



 $H^s$ & $H^s(x) = \int_{\Sp} \skalar{x,y}^s \, \nus[s](dy)$ for all $x \in \Rdnn$. It is a 
 $s$-homogeneous function, i.e. $H^s(x)=\abs{x}^s H^s(x/\abs{x})$ and satisfies $(\Pst H^s)(u)=k(s) H^s(u)$, $u \in \Sp$. \\

$I_\mu$ & $I_\mu = \{ s \ge 0 \ : \ \E \norm{\mM}^s < \infty \}$ \\
$\iota(\ma)$ & $\iota(\ma) := \inf \{ x \in
 \Sp \ : \ \abs{\ma x} \}$ \\
$ \slineu[t] $ & stopping line,  $\slineu[t]:=\{ v \in \tree \ : \  S^u(v)> t \text{
and } S^u(v|k) \le t \, \forall\, k<\abs{v} \}.$ \\
 
$\mathcal{J}_\alpha^K$ & see Definition \ref{def:J} \\ 

$k(s)$ & dominant eigenvalue of $\Ps$ and $\Pst$ satisfies $m(s)=\E N k(s)$ \\
$K_C$ & $K_C = \left( \min_{y \in C} \min_i y_i \right)$ for compact $C
\subset \interior{\Sp}$\\

$ \mL(v)$ & recursively defined by $\mL(\emptyset)= \Id$ and $\mL(vi)=\mL(v) \mT_i(v)$ \\
$L$-$\alpha$-elementary & see Definition \ref{defn:regular.elementary} \\
$\lambda_\ma$ & Perron-Frobenius eigenvalue of $\ma \in \interior{\Mset}$ \\

$M(u)$ & Disintegration, see Definition \ref{def:disintegration} \\
$\Mset$ & set $M(d \times d, \Rnn)$ of nonnegative $d\times d$ matrices \\
$m(s)$ & $m(s) = \E N \lim_{n \to \infty} \left( \E \norm{\mM_{1} \cdots \mM_{n}}^s
 \right)^{1/n}$, where $(\mM_{i})_{i \in \N}$ are i.i.d. with law $\mu$ \\
 $ (\mM_n)_{n \in \N} $ & sequence of random matrices (i.i.d. with law $\mu$ under $\Prob$) c.f. \eqref{eq:manytoone1}\\
 $\mu$ & law on $\Mset$, defined by $\int  f(\ma)  \mu(d\ma) ~:=~ (\E N)^{-1} \, \E \left( \sum_{i=1}^N f(\mT_i)\right).$  \\
 
 $\N, \No$ & $\N = \No \setminus \{0\}$. \\
$\nus$ & probability measure on $\Sp$, satisfy $(\Ps)' \nus = k(s)
\nus$  \\ 

$\Prob_u$ & notation for initial states, $\P[u]{U(\emptyset)=u, U_0=u, S_0=0}=1$
\\ $\Ps$, $\Pst$ & operators on $\Cf{\Sp}$ defined in  \eqref{def:Ps} resp.
\eqref{def:Pst}\\
$\mPi_n$ & $\mPi_n := \mM_n^\top \dots \mM_1^\top$ \\

$ \Prob^\alpha_u$ & exponentially shifted measure, see \eqref{eq:manytoone1}. \\

 $\Rnn, \Rp$ & $\Rnn =[0,\infty)$, $\Rp=(0,\infty)$ \\
 $\rho$ & stationary law of $(U(t), R(t))$ under $\Prob_u^\alpha$, see Theorem \ref{thm:kestenforW} \\

$S^u(v), S_n, S(t)$ & $S(v)= - \log \abs{\mL(v)^\top u}$, $S_n := -
\log \abs{\mPi_n U_0}$, $S(t)=S_{\tau_t}$ \\ 
  $\Sp$ & $\Sp = \Sd \cap \Rdnn$ intersection of the unit sphere and the
nonnegative cone in $\Rd$ \\

$\tau_t$ & $\tau_t:= \inf\{n \, : \, S_n
>t\}$ \\
$\tree$ & $\tree=\bigcup_{n=0}^\infty \N^n $ \\

$U^u(v), U_n, U(t)$ & $U(v)= \mL(v)^\top \as u$, $U_n= \mPi_n \as U_0$, $U(t)=U_{\tau_t}$ \\

$v_\ma$ & Perron-Frobenius eigenvalue of $\ma \in \interior{\Mset}$ \\

$W_n, W$ & martingale, see \eqref{def:martingale}, $W_n \to W$,
$W_0(u)=\est[\alpha](u$ \\ 
$W^*$ & particular fixed point of $\STi$, see Lemma \ref{lem:conv_Wn} \\
$W_{\slineu[t]}^f$ & $ W_{\slineu[t]}^f = \sum_{v \in \slineu[t]} H^\alpha(\mL(v)^\top U(v)) \,
f(U(v), S(v)-t),$ \\


$Z(u)$ &$Z(u):=- \log
  M(u)$
\end{tabular}